\documentclass[12pt,a4paper]{amsart}

\usepackage{amssymb}
\usepackage{array}

\newtheorem{theorem}{Theorem}[section]
\newtheorem{lemma}[theorem]{Lemma}
\newtheorem{corollary}[theorem]{Corollary}

\newtheorem{proposition}[theorem]{Proposition}

\theoremstyle{definition}
\newtheorem{definition}[theorem]{Definition}

\theoremstyle{remark}
\newtheorem{remark}[theorem]{Remark}

\numberwithin{equation}{section}
\numberwithin{figure}{section}

\def\bc{\mathbb C}

\def\bbf{\mathbb{F}}

\def\mct{\mathcal{T}}
\def\mcm{\mathcal{M}}
\def\bq{\mathbb Q}
\def\mcb{\mathcal{B}}

\def\mand{\quad\mbox{and}\quad}

\def\bp{\begin{pmatrix}}
\def\ep{\end{pmatrix}}

\def\sdi{\,\mbox{$\rhd\kern-.55em<$}\,}
\def\edots{\mathinner{\mkern1mu\raise1pt\hbox{.}\mkern2mu\raise4pt\hbox{.}\mkern2mu\raise7pt\vbox{\kern7pt\hbox{.}}\mkern1mu}}

\begin{document}

\title[Orthogonal multiple flag varieties]{Orthogonal multiple flag varieties of finite type II : even degree case}

\author{Toshihiko MATSUKI}

\thanks{Supported by JSPS Grant-in-Aid for Scientific Research (C) \# 16K05084.}

\address{Faculty of Letters\\
        Ryukoku University\\
        Kyoto 612-8577, Japan}
\email{matsuki@let.ryukoku.ac.jp}
\date{}

\begin{abstract}
Let $G$ be the split orthogonal group of degree $2n$ over an arbitrary infinite field $\bbf$ of ${\rm char}\,\bbf\ne 2$. In this paper, we classify multiple flag varieties $G/P_1\times\cdots\times G/P_k$ of finite type. Here a multiple flag variety is said to be of finite type if it has a finite number of $G$-orbits with respect to the diagonal action of $G$.
\end{abstract}

\maketitle

\section{Introduction}

Let $\bbf$ be an arbitrary commutative infinite field of ${\rm char}\,\bbf\ne 2$. Let $(\ ,\ )$ denote the symmetric bilinear form on $\bbf^{2n}$ defined by
$$(e_i,e_j)=\delta_{i,2n+1-j}$$
for $i,j=1,\ldots,2n$ where $e_1,\ldots,e_{2n}$ is the canonical basis of $\bbf^{2n}$. Define the split orthogonal group
$$G=\{g\in {\rm GL}_{2n}(\bbf) \mid (gu,gv)=(u,v)\mbox{ for all }u,v\in\bbf^{2n}\}$$
with respect to this form. Let us write $G={\rm O}_{2n}(\bbf)$ in this paper. We can also define the split special orthogonal group
$$G_0=\{g\in G\mid \det g=1\}\ (={\rm SO}_{2n}(\bbf)).$$
Note that
$$G=G_0\sqcup \bp I_{n-1} &&& 0 \\ & 0 & 1 & \\ & 1 & 0 & \\ 0 &&& I_{n-1} \ep G_0.$$
A subspace $V$ of $\bbf^{2n}$ is said to be isotropic if $(V,V)=\{0\}$.

For a sequence ${\bf a}=(\alpha_1,\ldots,\alpha_p)$ of positive integers such that $\alpha_1+\cdots+\alpha_p\le n$, we define the flag variety $M_{\bf a}$ of $G$ by
$$M_{\bf a}=\{V_1\subset\cdots\subset V_p \mid V_j\mbox{ are subspaces of }\bbf^{2n},\ \dim V_j=\alpha_1+\cdots+\alpha_j,\ (V_p,V_p)=\{0\}\}.$$

Consider a multiple flag variety $\mcm=\mcm_{{\bf a}_1,\ldots,{\bf a}_k}=M_{{\bf a}_1}\times\cdots\times M_{{\bf a}_k}$ with the diagonal $G$-action
$$g(m_1,\ldots,m_k)=(gm_1,\ldots,gm_k)$$
for $g\in G$ and $m_j\in M_{{\bf a}_j}$. We can consider the following problem.

\bigskip
\noindent{\bf Problem.} What is the condition on ${\bf a}_1,\ldots,{\bf a}_k$ for $|G\backslash \mcm|<\infty$? (A multiple flag variety $\mcm$ is said to be of finite type if  $|G\backslash \mcm|<\infty$.)

\begin{remark} (i) For the split orthogonal groups of odd degree, this problem was solved in \cite{M3}.

(ii) Magyar, Weyman and Zelevinsky solved this problem for the general linear groups in \cite{MWZ1}.

(iii) They also solved this problem for the symplectic groups in \cite{MWZ2} when $\bbf$ is algebraically closed.
\end{remark}

In this paper we solve this problem. First we prove:

\begin{proposition} Suppose $n\ge 2$ and $k\ge 4$. Then $|G\backslash \mcm|=\infty$.
\label{prop1.2}
\end{proposition}

So we have only to consider triple flag varieties $\mct=\mct_{{\bf a},{\bf b},{\bf c}}=M_{\bf a}\times M_{\bf b}\times M_{\bf c}$ with
$${\bf a}=(\alpha_1,\ldots,\alpha_p),\quad {\bf b}=(\beta_1,\ldots,\beta_q)\mand {\bf c}=(\gamma_1,\ldots,\gamma_r).$$
(When $k=2$, the decomposition $G\backslash \mcm$ is reduced to the Bruhat decomposition of $G$. So we have $|G\backslash \mcm|<\infty$.) Next we prove:

\begin{proposition} If $n\ge 3$ and $|G\backslash \mct|<\infty$, then one of ${\bf a},{\bf b},{\bf c}$ is $(1)$ or $(n)$.
\label{prop1.3}
\end{proposition}

So we may assume
\begin{equation}
{\bf a}=(1)\mbox{ or }(n)\mand q\le r \label{eq1.1}
\end{equation}
exchanging the roles of ${\bf a},{\bf b}$ and ${\bf c}$. We also prove:

\begin{proposition} If $n\ge 3,\ {\bf a}=(1),\ q\ge 2$ and $|G\backslash \mct|<\infty$, then ${\bf b}$ or ${\bf c}$ is $(k,n-k)$ with some $k$.
\label{prop1.4}
\end{proposition}

So we may assume
\begin{equation}
{\bf a}=(1)\mbox{ and }q\ge 2 \Longrightarrow {\bf b}=(k,n-k) \label{eq1.2}
\end{equation}
with some $k$ exchanging the roles of ${\bf b}$ and ${\bf c}$.

\begin{proposition} If $n\ge 4,\ {\bf a}=(n),\ q\ge 2$ and $|G\backslash \mct|<\infty$, then ${\bf b}$ or ${\bf c}$ is
$$(1,1),\quad (1,n-1)\quad\mbox{or}\quad (n-1,1).$$
\label{prop1.5}
\end{proposition}

So we may assume
\begin{equation}
{\bf a}=(n)\mbox{ and }q\ge 2 \Longrightarrow {\bf b}=(1,1),\ (1,n-1)\mbox{ or }(n-1,1) \label{eq1.3} 
\end{equation}
exchanging the roles of ${\bf b}$ and ${\bf c}$.

We also have conditions related with a property on $\bbf$:

\begin{proposition} Suppose $|\bbf^\times/(\bbf^\times)^2|=\infty$. Then we have $|G\backslash \mct|=\infty$ for the following three cases.

{\rm (i)} $\max(\alpha_1,\beta_1,\gamma_1)<n$ $($Proposition 1.4 in \cite{M3}$)$.

{\rm (ii)} ${\bf a}=(n),\ q,r\ge 2$ and $\max(\beta_1+\beta_2,\gamma_1+\gamma_2)<n$.

{\rm (iii)} ${\bf a}=(n),\ {\bf b}=(\beta)$ with $3\le \beta\le n-2,\ r\ge 4$ and $\gamma_1+\gamma_2+\gamma_3+\gamma_4<n$.
\label{prop1.6}
\end{proposition}

So we may assume the following conditions.
\begin{align}
& \quad\max(\alpha_1,\beta_1,\gamma_1)<n \Longrightarrow |\bbf^\times/(\bbf^\times)^2|<\infty. \label{eq1.4} \\
& \quad {\bf a}=(n),\ q,r\ge 2\mbox{ and }\max(\beta_1+\beta_2,\gamma_1+\gamma_2)<n \Longrightarrow |\bbf^\times/(\bbf^\times)^2|<\infty. \label{eq1.5} \\
& \quad {\bf a}=(n),\ {\bf b}=(\beta)\mbox{ with }3\le \beta\le n-2,\ r\ge 4 \ \mbox{and}\  \gamma_1+\gamma_2+\gamma_3+\gamma_4<n \label{eq1.6} \\
& \qquad\qquad\Longrightarrow |\bbf^\times/(\bbf^\times)^2|<\infty. \notag
\end{align}

Now we can state the theorem classifying triple flag varieties of finite type as follows.

\begin{theorem} Assume the conditions {\rm (\ref{eq1.1}), $\ldots$\ , (\ref{eq1.6})}. Then a triple flag variety $\mct=\mct_{{\bf a},{\bf b},{\bf c}}$ is of finite type if and only if $({\bf a},{\bf b},{\bf c})$ satisfies one of the following seven conditions.

{\rm (I-1)} ${\bf a}=(1)$ and $q=1$.

{\rm (I-2)} ${\bf a}=(1)$ and ${\bf b}=(k,n-k)$.

{\rm (II)} ${\bf a}=(n)$ and ${\bf b}=(1),\ (2),\ (3),\ (n-1),\ (n),\ (1,1),\ (1,n-1)$ or $(n-1,1)$.

{\rm (III-1)} ${\bf a}=(n),\ {\bf b}=(\beta)$ with $4\le \beta\le n-2$ and $r=1$.

{\rm (III-2)} ${\bf a}=(n),\ {\bf b}=(\beta)$ with $4\le \beta\le n-2$ and $r=2$.

{\rm (III-3)} ${\bf a}=(n),\ {\bf b}=(\beta)$ with $4\le \beta\le n-2$ and ${\bf c}$ is one of
\begin{align*}
& (1,k,n-k-1),\ (k,1,n-k-1),\ (k,n-k-1,1), \\
& (1,1,k),\ (1,k,1)\mbox{ or }(k,1,1).
\end{align*}

{\rm (III-4)} ${\bf a}=(n),\ {\bf b}=(\beta)$ with $4\le \beta\le n-2$ and ${\bf c}$ is one of
$$(1,1,1,n-3),\ (1,1,n-3,1),\ (1,n-3,1,1),\ (n-3,1,1,1)\mbox{ or }(1,1,1,1).$$
\label{th1.7}
\end{theorem}

\begin{remark} Suppose $\bbf=\bc$ and the pair $({\bf a},{\bf b})$ satisfies one of the conditions (I-1), (I-2) or (II) in Theorem \ref{th1.7}. 
Then it is shown by Stembridge in \cite{S} that the double flag variety $\mathcal{D}=M_{\bf a}\times M_{\bf b}$ has an open $B$-orbit where $B$ is a Borel subgroup of $G$. It follows from Brion and Vinberg's theorem in \cite{B} and \cite{V} that $\mathcal{D}$ has a finite number of $B$-orbits. So the triple flag variety $\mct_{{\bf a},{\bf b},(1^n)}$ is of finite type in these cases. (Note that $(\bc^\times)^2=\bc^\times$ because $\bc$ is algebraically closed.)
\label{rem1.8}
\end{remark}

This paper is organized as follows.

In Section 2, we prepare elementary results on the flag varieties of ${\rm O}_{2n}(\bbf)$.

In Section 3, we first prove Proposition \ref{prop1.2}. Next we prove Propositions \ref{prop1.3}, \ref{prop1.4}, \ref{prop1.5} and \ref{prop1.6} excluding triple flag varieties of infinite type. In Section \ref{sec3.5}, we show that the triple flag varieties $\mct_{(n),(\beta),{\bf c}}$ with $4\le \beta\le n-2$ and $r\ge 3$ are of infinite type except the ones in Theorem \ref{th1.7} (III-3) and (III-4) (Corollaries \ref{cor3.14}, \ref{cor3.16}, \ref{cor3.18} and \ref{cor3.20}).

In Section 4, we formulate propositions to show finiteness of triple flag varieties listed in Theorem \ref{th1.7}.

In Section 5, we review technical formulations and results given in \cite{M3} to show finiteness of triple flag varieties. In particular, we see that the triple flag varieties in Theorem \ref{th1.7} (III-1) and (III-2) are of finite type (Theorem \ref{th5.9} and Corollary \ref{cor5.19}).

As we remarked in Remark \ref{rem1.8}, it is known that the triple flag varieties in Theorem \ref{th1.7} (I-1), (I-2) and (II) are of finite type when $\bbf=\bc$. So we first prove finiteness of the ``new classes'' (III-3) and (III-4) in Section 6 and 7.

In Sections 8, 9, 10 and 11, we prove finiteness of the cases (I-1), (I-2) and (II). These results are not covered by \cite{S} since we consider an arbitrary infinite field $\bbf$ of ${\rm char}\,\bbf\ne 2$. Moreover we should note that we have $|\bbf^\times/(\bbf^\times)^2|=\infty$ for some infinite fields (for example $\bbf=\bq$).

In Section 12, we prepare general results on flag varieties of general linear groups for the sake of Sections 6, 7, 8, 9, 10 and 11.

\section{Preliminaries}

\begin{definition} An isotropic subspace $V$ of $\bbf^{2n}$ is said to be ``maximally isotropic'' if it is maximal with respect to the inclusion relation. Let $M$ denote the set of maximally isotropic subspaces of $\bbf^{2n}$.
\end{definition}

Clearly $U_0=\bbf e_1\oplus\cdots\oplus \bbf e_n$ is a maximally isotropic subspace of $\bbf^{2n}$. For any $g\in G$, the spaces $gU_0$ are maximally isotropic. Let $P$ denote the subgroup of $G$ defined by
$$P=\{g\in G\mid gU_0=U_0\}.$$
For $d=0,\ldots,n$, write
$$U_d=\bbf e_1\oplus\cdots\oplus \bbf e_{n-d}\oplus \bbf e_{n+1}\oplus\cdots\oplus \bbf e_{n+d}=w_dU_0$$
where $w_d$ is an element of $G$ defined by
$$w_de_i=\begin{cases} e_i & \text{if $i\le n-d$ or $i\ge n+d+1$,} \\
e_{2n+1-i} & \text{if $n-d+1\le i\le n+d$.}
\end{cases}$$
Note that $w_d\in G_0$ if and only if $d$ is even.

Define subgroups $L$ and $N$ of $P$ by
\begin{align*}
L & =\{g\in G\mid gU_0=U_0,\ gU_n=U_n\}=\{\ell(A)\mid A\in {\rm GL}_n(\bbf)\} \\
\mand N & = \{g\in G\mid ge_i=e_i\mbox{ for }i=1,\ldots,n\}=\left\{\bp I_n & * \\ 0 & I_n \ep \in G\right\}
\end{align*}
where
$$\ell(A)=\bp A & 0 \\ 0 & J_n{}^tA^{-1}J_n \ep.$$
Then we can write $P=LN$ as a semidirect product (Levi decomposition).

\begin{proposition} Let $V$ be a maximally isotropic subspace of $\bbf^{2n}$. If $\dim (V\cap U_0)=n-d$, then $V\in PU_d$. $($In particular, $V$ is maximally isotropic if and only if $V$ is isotropic and $\dim V=n$.$)$
\label{prop2.2}
\end{proposition}

\begin{proof} Suppose that $\dim (V\cap U_0)=n-d$. By the action of $L$, we may assume
$$V\cap U_0=\bbf e_1\oplus\cdots\oplus \bbf e_{n-d}.$$
Since $V$ is isotropic, we have $V\subset \bbf e_1\oplus\cdots\oplus \bbf e_{n+d}$. Let $\pi$ denote the canonical projection
$$\pi: \bbf e_1\oplus\cdots\oplus \bbf e_{n+d}\to \bbf e_{n+1}\oplus\cdots\oplus \bbf e_{n+d}.$$

If $\pi|_V$ is not surjective, then we can take a nonzero element $v$ of $\bbf e_{n-d+1}\oplus\cdots\oplus \bbf e_n$ such that $(v,\pi(V))=\{0\}$. This implies $(v,V)=\{0\}$, a contradiction to the maximality of $V$. Hence  $\pi|_V$ is surjective.

On the other hand, the kernel of  $\pi|_V$ is $V\cap U_0$. Hence we have
$$\dim V=\dim {\rm ker} \pi|_V+\dim {\rm Im} \pi|_V=(n-d)+d=n.$$

We can write
$$V=(V\cap U_0)\oplus \bbf v_1\oplus\cdots\oplus \bbf v_d$$
with some vectors $v_i\in e_{n+i}+(\bbf e_{n-d+1}\oplus\cdots\oplus \bbf e_n)$ for $i=1,\ldots,d$. Define an element $n\in N$ by
$$ne_i=\begin{cases} v_{i-n} & \text{if $n+1\le i\le n+d$,} \\
e_i & \text{otherwise.}
\end{cases}$$
Then $n^{-1}V=U_d$. Thus we have proved $V\in PU_d$.
\end{proof}

Proposition \ref{prop2.2} implies:

\begin{corollary} $($Bruhat decomposition$)$ $M=\bigsqcup_{d=0}^n PU_d$
\end{corollary}

On the other hand, since $P\subset G_0$ we have:

\begin{corollary} {\rm (i)} $M=GU_0=G_0U_0\sqcup G_0U_1$.

{\rm (ii)} $G_0U_0=\bigsqcup_{d\mbox{ is even}} PU_d,\ G_0U_1=\bigsqcup_{d\mbox{ is odd}} PU_d$
\label{cor2.4}
\end{corollary}

Considering $L$-action, we have:

\begin{corollary} For $k=1,\ldots,n-1$, let $V$ be a $k$-dimensional isotropic subspace of $\bbf^{2n}$. Then there exists a $g\in G_0$ such that $gV=\bbf e_1\oplus\cdots\oplus \bbf e_k$.
\label{cor2.5}
\end{corollary}

\begin{proposition} Let $V$ be an $n-1$-dimensional isotropic subspace of $\bbf^{2n}$. Then there are only two maximally isotropic subspaces $V_1$ and $V_2$ of $\bbf^{2n}$ containing $V$. Moreover if $gV=V$ for some $g\in G_0$, then $gV_1=V_1$ and $gV_2=V_2$.
\label{prop2.6}
\end{proposition}

\begin{proof} By Corollary \ref{cor2.5}, we may assume $V=\bbf e_1\oplus\cdots\oplus \bbf e_{n-1}$. Then there are only two maximally isotropic subspaces
$$V_1=U_0\mand V_2=U_1$$
containing $V$. The remaining assertion follows from Corollary \ref{cor2.4}.
\end{proof}

We use the following technical lemma given in \cite{M3}.

\begin{lemma} $($Lemma 2.1 in \cite{M3}$)$ Let $W$ be a nondegenerate subspace of $\bbf^{2n}$. Let $W_1,\ldots,W_k,W'_1,\ldots,W'_k$ be subspaces of $W$ such that
$$W=W_1+\cdots+ W_k=W'_1+\cdots+ W'_k.$$
Let $U_1$ be an isotropic subspace of $W^\perp$ and $U_2,\ldots,U_k$ be subspaces of $U_1$. Let $g$ be an element of $G$ such that
$$g(W_i\oplus U_i)=W'_i\oplus U_i\quad \mbox{for }i=1,\ldots,k.$$
Then we have$:$

{\rm (i)} $g(W\oplus U_1)=W\oplus U_1$.\qquad {\rm (ii)} $gU_1=U_1$.

\noindent$($ By {\rm (i)} and {\rm (ii)}, $g$ induces a linear automorphism on  $(W\oplus U_1)/U_1\cong W$ preserving the bilinear form $(\ ,\ )$.$)$
\label{lem2.7}
\end{lemma}

\section{Exclusion of multiple flag varieties of infinite type}

\subsection{Proof of Proposition \ref{prop1.2}}

Consider $\bbf^4$ with the canonical basis $f_1,\ldots,f_4$ and the symmetric bilinear form $(\ ,\ )$ such that $(f_i,f_j)=\delta_{i,5-j}$. Write $G'={\rm O}_4(\bbf)$ and $G'_0={\rm SO}_4(\bbf)$. Take maximally isotropic subspaces
\begin{align*}
U_1 & =\bbf f_1\oplus\bbf f_2,\quad U_2=\bbf f_3\oplus\bbf f_4,\quad U_3=\bbf(f_1+f_3)\oplus \bbf(f_2-f_4) \\
\mand U_{4,\lambda} & =\bbf(f_1+\lambda f_3)\oplus \bbf(f_2-\lambda f_4)
\end{align*}
with $\lambda\in\bbf$. Take one-dimensional subspaces $U'_1=\bbf f_1,\ U'_2=\bbf f_3,\ U'_3=\bbf(f_2-f_4)$ and $U'_{4,\lambda}=\bbf(f_2-\lambda f_4)$ of $U_1,\ U_2,\ U_3$ and $U_{4,\lambda}$, respectively. Define multiple flags
\begin{align*}
m^0_\lambda & =(U'_1,U'_2,U'_3,U'_{4,\lambda}), \quad 
m^1_\lambda =(U_1,U'_2,U'_3,U'_{4,\lambda}), \quad
m^2_\lambda =(U_1,U_2,U'_3,U'_{4,\lambda}), \\
m^3_\lambda & =(U_1,U_2,U_3,U'_{4,\lambda})
\mand m^4_\lambda =(U_1,U_2,U_3,U_{4,\lambda}).
\end{align*}

\begin{lemma} Suppose $G'm^i_\lambda\ni m^i_\mu$ for some $i=0,\ldots,4$. Then we have $\lambda=\mu$.
\label{lem3.1}
\end{lemma}

\begin{proof} Suppose $gm^i_\lambda=m^i_\mu$ for some $g\in G'$. When $i=1,2,3,4$, we have $gU_1=U_1$ and hence $g\in G'_0$ by Corollary \ref{cor2.4}. On the other hand, when $i=0$, we have $g(U'_1\oplus U'_2)=U'_1\oplus U'_2$ and hence $g\in G'_0$ by Corollary \ref{cor2.4}.

So we have only to consider the case of $i=4$ by Proposition \ref{prop2.6}.

Since $gU_1=U_1$ and $gU_2=U_2$,
$$g=\ell(A)=\bp A & 0 \\ 0 & J_2{}^tA^{-1}J_2 \ep\quad(J_2=\bp 0 & 1 \\ 1 & 0 \ep)$$
with some $A\in{\rm GL}_2(\bbf)$. Since $gU_3=U_3$, we have $g=\ell(A)$ with $A\in{\rm SL}_2(\bbf)$. Write
$$A=\bp a & b \\ c & d \ep.$$
Then we have
$$J_2{}^tA^{-1}J_2=J_2\bp d & -c \\ -b & a \ep J_2=\bp a & -b \\ -c & d \ep.$$
So we have
$$g(f_1+\lambda f_3)= a(f_1+\lambda f_3)+c(f_2-\lambda f_4)\mand g(f_2-\lambda f_4)=b(f_1+\lambda f_3)+d(f_2-\lambda f_4).$$
This implies $gU_{4,\lambda}=U_{4,\lambda}$ and therefore $\lambda=\mu$. 
\end{proof}

\noindent {\it Proof of Proposition \ref{prop1.2}}: We may assume $k=4$ and ${\bf a}_j=(\alpha_j)$ for $j=1,2,3,4$. We may moreover assume $\alpha_1\ge \alpha_2\ge \alpha_3\ge \alpha_4$. Define
$$j_0=\begin{cases} 0 & \text{if $\alpha_1=1$,} \\ \max\{j\mid \alpha_j\ge 2\} & \text{if $\alpha_1\ge 2$.} \end{cases}$$
Let $\varphi: \bbf^4\to \bbf^{2n}$ be the linear inclusion defined by $\varphi(f_i)=e_{i+n-2}$ for $i=1,2,3,4$. Then $\varphi$ preserves the bilinear form $(\ ,\ )$. Define multiple flags $n_\lambda=(V_1,V_2,V_3,V_{4,\lambda})\in \mcm_{(\alpha_1),(\alpha_2),(\alpha_3),(\alpha_4)}$ by
$$V_j=\begin{cases} \varphi(U'_j) & \text{if $j>j_0$,} \\ U_{[\alpha_j-2]}\oplus \varphi(U_j) & \text{if $j\le j_0$} \end{cases}$$
for $j=1,2,3$ and
$$V_{4,\lambda}=\begin{cases} \varphi(U'_{4,\lambda}) & \text{if $j_0<4$,} \\ U_{[\alpha_4-2]}\oplus \varphi(U_{4,\lambda}) & \text{if $j_0=4$}\end{cases}$$
where $U_{[\ell]}=\bbf e_1\oplus\cdots\oplus\bbf e_\ell$. Suppose $\lambda,\mu\ne 1$ and $gn_\lambda=n_\mu$ for some $g\in G$. Then we have only to show that $\lambda=\mu$. Write
$$V_0=\begin{cases} \{0\} & \text{if $\alpha_1=1$,} \\ U_{[\alpha_1-2]} & \text{if $\alpha_1\ge 2$.} \end{cases}$$
Since $\lambda,\mu\ne 1$, we have
$$U'_1+U'_2+U'_3+U'_{4,\lambda}=U'_1+U'_2+U'_3+U'_{4,\mu}=\bbf^4.$$
So we have
$$gV_0=V_0\mand g(V_0\oplus \varphi(\bbf^4))=V_0\oplus \varphi(\bbf^4)$$
by Lemma \ref{lem2.7}. Hence $g$ induces a linear automorphism on $(V_0\oplus \varphi(\bbf^4))/V_0\cong \bbf^4$ which preserves the bilinear form $(\ ,\ )$. Since $gm^{j_0}_\lambda=m^{j_0}_\mu$, we have $\lambda=\mu$ by Lemma \ref{lem3.1}. \hfill$\square$

\subsection{Proof of Proposition \ref{prop1.3}}

Consider $\bbf^6$ with the canonical basis $f_1,\ldots,f_6$ and the symmetric bilinear form $(\ ,\ )$ such that $(f_i,f_j)=\delta_{i,7-j}$. We prepare following lemmas for $G'={\rm O}_6(\bbf)$. (Lemma \ref{lem3.2'} and Lemma \ref{lem3.1'} (i) are given in \cite{M3}.)

\begin{lemma} $($Lemma 2.9 in \cite{M3}$)$ Let $m_\lambda$ be the triple flag in $\mct_{(2),(2),(2)}$ defined by
$$m_\lambda=(\bbf f_1\oplus\bbf f_2,\ \bbf f_5\oplus\bbf f_6,\ \bbf(f_1+f_3+f_5)\oplus \bbf(\lambda f_2-f_4+(1-\lambda)f_6)).$$
Then we have
$$G'm_\lambda\ni m_\mu\Longrightarrow \lambda=\mu\mbox{ or }1-\mu.$$
\label{lem3.2'}
\end{lemma}

\begin{lemma}  For the following triple flags $m_\lambda\ (\lambda\in\bbf)$ in $\mct_{(2),(2),(1,2)},\ \mct_{(2),(1,2),(1,2)}$ and $\mct_{(1,2),(1,2),(1,2)}$, respectively, we have
$$gm_\lambda=m_\mu\mbox{ for some }g\in G'\Longrightarrow \lambda=\mu.$$

{\rm (i)} $(\bbf f_1\oplus\bbf f_2, \ \bbf f_5\oplus\bbf f_6,\  \bbf(\lambda f_1-f_3+(1-\lambda) f_5)\subset$

\qquad $\bbf(f_1-f_5)\oplus \bbf(f_1-f_3)\oplus \bbf(f_2+f_4+f_6))$

{\rm (ii)} $(\bbf(f_1+f_5)\oplus \bbf(f_2-f_6),\ \bbf(f_1+f_3)\subset \bbf f_1\oplus\bbf f_2\oplus\bbf f_3,$

\qquad $\bbf(f_4+\lambda f_6)\subset \bbf f_4\oplus\bbf f_5\oplus\bbf f_6)$

{\rm (iii)} $(\bbf(f_1+f_3)\subset \bbf f_1\oplus\bbf f_2\oplus\bbf f_3,\ \bbf(f_1+f_5)\subset \bbf f_1\oplus\bbf f_4\oplus\bbf f_5,$

\qquad$\bbf(f_3+\lambda f_5)\subset \bbf f_3\oplus\bbf f_5\oplus\bbf f_6)$
\label{lem3.1'}
\end{lemma}

\begin{proof} (i) is proved in Lemma 2.12 in \cite{M3}.

(ii) The element $g$ stabilizes $\bbf f_3$ and $\bbf f_4$ since they are the orthogonal spaces of $\bbf(f_1+f_5)\oplus \bbf(f_2-f_6)$ in $\bbf f_1\oplus\bbf f_2\oplus\bbf f_3$ and  $\bbf f_4\oplus\bbf f_5\oplus\bbf f_6$, respectively. Hence we have
$$gf_3=kf_3\mand gf_4=k^{-1}f_4$$
with some $k\in\bbf^\times$. Write $g(f_1+f_3)=a(f_1+f_3)$ and $g(f_4+\lambda f_6)=b(f_4+\mu f_6)$ with $a,b\in\bbf^\times$. We have
$$1=(f_1+f_3,f_4)=(g(f_1+f_3),gf_4)=(a(f_1+f_3),k^{-1}f_4)=ak^{-1}$$
and
$$1=(f_3,f_4+\lambda f_6)=(gf_3,g(f_4+\lambda f_6))=(kf_3,b(f_4+\mu f_6))=kb.$$
Hence $a=k$ and $b=k^{-1}$. On the other hand,
$$1+\lambda=(f_1+f_3,f_4+\lambda f_6)=(g(f_1+f_3),g(f_4+\lambda f_6))=(k(f_1+f_3),k^{-1}(f_4+\mu f_6))=1+\mu.$$
Hence we have $\lambda=\mu$.

(iii) Taking the intersections of two of the three spaces
$$\bbf f_1\oplus\bbf f_2\oplus\bbf f_3,\quad \bbf f_1\oplus\bbf f_4\oplus\bbf f_5\mand \bbf f_3\oplus\bbf f_5\oplus\bbf f_6,$$
we have
$$g\bbf f_1=\bbf f_1,\quad g\bbf f_3=\bbf f_3,\quad g\bbf f_5=\bbf f_5.$$
Write $gf_1=af_1,\ gf_3=bf_3$ and $gf_5=cf_5$ with $a,b,c\in\bbf^\times$. Since $g\bbf(f_1+f_3)=\bbf(f_1+f_3)$ and since $g\bbf(f_1+f_5)=\bbf(f_1+f_5)$, we have
$$a=b=c.$$
Hence $g(f_3+\lambda f_5)=a(f_3+\lambda f_5)$. But $g(f_3+\lambda f_5)\subset \bbf(f_3+\mu f_5)$ by the assumption. Thus we have $\lambda=\mu$.
\end{proof}

\begin{corollary} Following four triple flag varieties of ${\rm O}_6(\bbf)$ are of infinite type.
$$\mct_{(2),(2),(2)},\qquad \mct_{(2),(2),(1,2)},\qquad \mct_{(2),(1,2),(1,2)},\qquad \mct_{(1,2),(1,2),(1,2)}.$$
\end{corollary}

\bigskip
\noindent {\it Proof of Proposition \ref{prop1.3}}: We have only to show that
$${\bf a},\ {\bf b}\mbox{ and }{\bf c}\mbox{ are not equal to }(1)\mbox{ or }(n) \Longrightarrow |G\backslash \mct_{{\bf a},{\bf b},{\bf c}}|=\infty.$$
We may assume
\begin{align*}
{\bf a} & =(\alpha)\mbox{ with }2\le \alpha\le n-1\mbox{ or } {\bf a}=(1,n-1), \\
{\bf b} & =(\beta)\mbox{ with }2\le \beta\le n-1\mbox{ or } {\bf b}=(1,n-1), \\
\mand {\bf c} & =(\gamma)\mbox{ with }2\le \gamma\le n-1\mbox{ or } {\bf c}=(1,n-1).
\end{align*}
Exchanging the roles of ${\bf a},\ {\bf b}$ and ${\bf c}$ if necessary, we have only to consider  the following four cases of  $({\bf a},{\bf b},{\bf c})$:

(i) $((\alpha),(\beta),(\gamma))$ with $\alpha,\beta,\gamma\in\{2,\ldots,n-1\}$,

(ii) $((\alpha),(\beta),(1,n-1))$ with $\alpha,\beta\in\{2,\ldots,n-1\}$,

(iii) $((\alpha),(1,n-1),(1,n-1))$ with $\alpha\in\{2,\ldots,n-1\}$,

(iv) $((1,n-1),(1,n-1),(1,n-1))$.

\bigskip
Let $\varphi: \bbf^6\to\bbf^{2n}$ denote the linear inclusion defined by $\varphi(f_i)=e_{i+n-3}$ for $i=1,\ldots,6$. ($\varphi$ preserves the bilinear form $(\ ,\ )$.)

(i) When $({\bf a},{\bf b},{\bf c})=((\alpha),(\beta),(\gamma))$ with $\alpha,\beta,\gamma\in\{2,\ldots,n-1\}$, define $m_\lambda=(W_1, W_2, W_3^\lambda)\in\mct_{{\bf a},{\bf b},{\bf c}}$ for $\lambda\in\bbf^\times$ by
\begin{align*}
W_1 & =\varphi(\bbf f_1\oplus\bbf f_2)\oplus U_{[\alpha-2]},\ W_2=\varphi(\bbf f_5\oplus\bbf f_6)\oplus U_{[\beta-2]} \\
\mand W_3^\lambda & =\varphi(\bbf(f_1+f_3+f_5)\oplus\bbf(\lambda f_2-f_4+(1-\lambda)f_6))\oplus U_{[\gamma-2]}.
\end{align*}
Then we have $Gm_\lambda\ni m_\mu\Longrightarrow \lambda=\mu$ or $1-\mu$ by Lemma \ref{lem2.7} and Lemma \ref{lem3.2'}.

(ii) When $({\bf a},{\bf b},{\bf c})=((\alpha),(\beta),(1,n-1))$ with $\alpha,\beta\in\{2,\ldots,n-1\}$, define $m_\lambda=(W_1, W_2, W_3^\lambda\subset W_4)\in\mct_{{\bf a},{\bf b},{\bf c}}$ for $\lambda\in\bbf^\times$ by
\begin{align*}
W_1 & =\varphi(\bbf f_1\oplus\bbf f_2)\oplus U_{[\alpha-2]},\ W_2=\varphi(\bbf f_5\oplus\bbf f_6)\oplus U_{[\beta-2]}, \\
W_3^\lambda & =\varphi(\bbf(\lambda f_1-f_3+(1-\lambda) f_5) \\
\mand W_4 & =\varphi(\bbf(f_1-f_5)\oplus \bbf(f_1-f_3)\oplus \bbf(f_2+f_4+f_6))\oplus U_{[n-3]}.
\end{align*}
Then we have $Gm_\lambda\ni m_\mu\Longrightarrow \lambda=\mu$ by Lemma \ref{lem2.7} and Lemma \ref{lem3.1'} (i).

(iii) When $({\bf a},{\bf b},{\bf c})=((\alpha),(1,n-1),(1,n-1))$ with $\alpha\in\{2,\ldots,n-1\}$, define $m_\lambda=(W_1, W_2\subset W_3, W_4^\lambda\subset W_5)\in\mct_{{\bf a},{\bf b},{\bf c}}$ for $\lambda\in\bbf^\times$ by
\begin{align*}
W_1 & =\varphi(\bbf (f_1+f_5)\oplus\bbf (f_2-f_6))\oplus U_{[\alpha-2]}, \\
W_2 & =\varphi(\bbf (f_1+f_3)),\ W_3=\varphi(\bbf f_1\oplus\bbf f_2\oplus\bbf f_3)\oplus U_{[n-3]}, \\
W_4^\lambda & =\varphi(\bbf(f_4+\lambda f_6))\mand W_5=\varphi(\bbf f_4\oplus \bbf f_5\oplus\bbf f_6)\oplus U_{[n-3]}.
\end{align*}
Then we have $Gm_\lambda\ni m_\mu\Longrightarrow \lambda=\mu$ by Lemma \ref{lem2.7} and Lemma \ref{lem3.1'} (ii).

(iv) When $({\bf a},{\bf b},{\bf c})=((1,n-1),(1,n-1),(1,n-1))$, define $m_\lambda=(W_1\subset  W_2, W_3\subset W_4, W_5^\lambda\subset W_6)\in\mct_{{\bf a},{\bf b},{\bf c}}$ for $\lambda\in\bbf^\times$ by
\begin{align*}
W_1 & =\varphi(\bbf (f_1+f_3)), \ W_2=\varphi(\bbf f_1\oplus\bbf f_2\oplus\bbf f_3)\oplus U_{[n-3]}, \\
W_3 & =\varphi(\bbf (f_1+f_5)),\ W_4=\varphi(\bbf f_1\oplus\bbf f_4\oplus\bbf f_5)\oplus U_{[n-3]}, \\
W_5^\lambda & =\varphi(\bbf(f_3+\lambda f_5))\mand W_6=\varphi(\bbf f_3\oplus \bbf f_5\oplus\bbf f_6)\oplus U_{[n-3]}.
\end{align*}
Then we have $Gm_\lambda\ni m_\mu\Longrightarrow \lambda=\mu$ by Lemma \ref{lem2.7} and Lemma \ref{lem3.1'} (iii). \hfill$\square$

\subsection{Proof of Proposition \ref{prop1.4}}

Consider $\bbf^5$ with the canonical basis $f_1,\ldots,f_5$ and the symmetric bilinear form $(\ ,\ )$ such that $(f_i,f_j)=\delta_{i,6-j}$. 
We prepare a lemma for ${\rm O}_5(\bbf)$ which follows easily from Lemma 2.5 in \cite{M3}. (We give a proof for the sake of convenience.)

\begin{lemma} Let $m_\lambda\ (\lambda\in\bbf^\times)$ be the triple flag in $\mct_{(1),(1,1),(1,1)}$ for $G'={\rm O}_5(\bbf)$ given by
$$m_\lambda=(\bbf(f_1+f_3-\frac{1}{2}f_5),\ \bbf(f_1+f_2)\subset \bbf f_1\oplus\bbf f_2,\ \bbf(f_4+\lambda f_5)\subset \bbf f_4\oplus\bbf f_5).$$
Then
$$G'm_\lambda\ni m_\mu\Longrightarrow \lambda=\mu.$$
\label{lem3.3'}
\end{lemma}

\begin{proof} Suppose $gm_\lambda=m_\mu$ for some $g\in G'$. Since $g$ stabilizes $\bbf f_1\oplus\bbf f_2$ and $\bbf f_4\oplus\bbf f_5$, we can write
$$g=\bp A & 0 & 0 \\ 0 & \varepsilon & 0 \\ 0 & 0 & J_2{}^tA^{-1}J_2 \ep$$
with some $A\in{\rm GL}_2(\bbf)$ and $\varepsilon\in\{\pm 1\}$. Since $\bbf f_2$ and $\bbf f_4$ are the orthogonal subspaces of $\bbf(f_1+f_3-\frac{1}{2}f_5)$ in $\bbf f_1\oplus\bbf f_2$ and $\bbf f_4\oplus\bbf f_5$, respectively, we have $g\bbf f_2=\bbf f_2$ and $g\bbf f_4=\bbf f_4$. Hence we can write
$$A=\bp a & 0 \\ 0 & b \ep$$
with some $a,b\in\bbf^\times$. Since $g$ stabilizes $\bbf(f_1+f_2)$ and $\bbf(f_1+f_3-\frac{1}{2}f_5)$, we have $a=b=\varepsilon$ and therefore $g=\varepsilon I_5$. Since $g\bbf(f_4+\lambda f_5)=\bbf(f_4+\mu f_5)$, we have $\lambda=\mu$.
\end{proof}

\noindent {\it Proof of Proposition \ref{prop1.4}}: Suppose ${\bf a}=(1),\ {\bf b}=(\beta_1,\beta_2)$ and ${\bf c}=(\gamma_1,\gamma_2)$ with $\beta_1+\beta_2<n,\ \gamma_1+\gamma_2<n$. Then we have only to show $|G\backslash \mct_{{\bf a},{\bf b},{\bf c}}|=\infty$.

Let $\varphi: \bbf^5\to\bbf^{2n}$ denote the linear inclusion defined by
$$\varphi(f_1)=e_{n-2},\ \varphi(f_2)=e_{n-1},\ \varphi(f_3)=e_n+\frac{1}{2}e_{n+1},\ \varphi(f_4)=e_{n+2}\mand \varphi(f_5)=e_{n+3}.$$
Then we have $(\varphi(u),\varphi(v))=(u,v)$ for $u,v\in\bbf^5$.

Define triple flags $m_\lambda\in\mct_{{\bf a},{\bf b},{\bf c}}$ for $\lambda\in\bbf^\times$ by
\begin{align*}
m_\lambda & =(\varphi(\bbf(f_1+f_3-\frac{1}{2}f_5)),\ \varphi(\bbf(f_1+f_2))\oplus U_{[\beta_1-1]}\subset \varphi(\bbf f_1\oplus\bbf f_2)\oplus U_{[\beta_1+\beta_2-2]}, \\
& \qquad \varphi(\bbf (f_4+\lambda f_5))\oplus U_{[\gamma_1-1]}\subset \varphi(\bbf f_4\oplus\bbf f_5)\oplus U_{[\gamma_1+\gamma_2-2]}).
\end{align*}
Then we have $Gm_\lambda\ni m_\mu\Longrightarrow \lambda=\mu$ by Lemma \ref{lem2.7} and Lemma \ref{lem3.3'}.
\hfill$\square$

\subsection{Proof of Proposition \ref{prop1.5}} \label{sec3.1}

Consider $\bbf^8$ with the canonical basis $f_1,\ldots,f_8$ and the symmetric bilinear form $(\ ,\ )$ such that $(f_i,f_j)=\delta_{i,9-j}$. Take maximal isotropic subspaces
\begin{align*}
U_+ & =\bbf f_1\oplus \bbf f_2\oplus \bbf f_3\oplus \bbf f_4, \quad
U_-=\bbf f_5\oplus \bbf f_6\oplus \bbf f_7\oplus \bbf f_8 \\
 \mand
V & =\bbf(f_1+f_7)\oplus \bbf(f_2-f_8)\oplus \bbf(f_3+f_5)\oplus \bbf(f_4-f_6)
\end{align*}
of $\bbf^8$. Note that the space $V$ is written as
$$V=\{u+\kappa(u)\mid u\in U_+\}$$
where $\kappa: U_+\to U_-$ is the linear bijection defined by
$$\kappa(f_1)=f_7,\quad \kappa(f_2)=-f_8,\quad \kappa(f_3)=f_5,\mand \kappa(f_4)=-f_6.$$
Consider the bilinear form $\langle\ ,\ \rangle$ on $U_+$ defined by
$$\langle u,v\rangle =(\kappa(u),v)$$
for $u,v\in U_+$. Then $\langle\ ,\ \rangle$ is an alternating bilinear form satisfying
\begin{align*}
\langle f_1,f_2\rangle & =\langle f_3,f_4\rangle=1 
\mand \langle f_1,f_3\rangle =\langle f_1,f_4\rangle=\langle f_2,f_3\rangle=\langle f_2,f_4\rangle=0.
\end{align*}

Let $K$ denote the subgroup of $G'={\rm O}_8(\bbf)$ defined by
$$K=\{\ell(A)\mid A\in{\rm GL}_4(\bbf),\ \langle Au,Av \rangle=\langle u,v \rangle\mbox{ for }u,v\in U_+\}\cong {\rm Sp}_4(\bbf)$$
where
$$\ell(A)=\bp A & 0 \\ 0 & J_4{}^tA^{-1}J_4 \ep$$
for $A\in{\rm GL}_4(\bbf)$.

\begin{lemma} $K=\{g\in G'\mid gU_+=U_+,\ gU_-=U_-,\ gV=V\}$.
\end{lemma}

\begin{proof} Let $g$ be an element of $G'$ such that
$$gU_+=U_+\quad\mbox{and that}\quad gU_-=U_-.$$
Then $g=\ell(A)$ with some $A\in {\rm GL}_4(\bbf)$.

Suppose moreover that $gV=V$. Then we have
$$g\kappa(u)=\kappa(gu)\quad\mbox{for all }u\in U_+.$$
Hence
$$\langle gu,gv \rangle=(\kappa(gu),gv)=(g\kappa(u),gv)=(\kappa(u),v)=\langle u,v \rangle$$
for $u,v\in U_+$ which implies $g\in K$.

Converse assertion is clear.
\end{proof}

\begin{lemma} Let $(U_+^{1,\lambda}\subset U_+^2,\ U_-^1\subset U_-^2)$ be one of the following six kinds of pairs of flags $(\lambda\in\bbf^\times)$ where $U'_+=\bbf(f_1+f_3)\oplus \bbf f_2\oplus \bbf f_4$ and $U'_-=\bbf f_5\oplus \bbf f_6\oplus \bbf f_7$.

{\rm (i)} $(\bbf(f_1+f_3)\oplus \bbf(f_2+\lambda f_4)\subset U_+,\ \bbf f_5\oplus \bbf f_6\subset U_-)$

{\rm (ii)} $(\bbf(f_1+f_3)\oplus \bbf(f_2+\lambda f_4)\subset U_+,\ \bbf f_5\oplus \bbf f_6\subset U'_-)$

{\rm (iii)} $(\bbf(f_1+f_3)\oplus \bbf(f_2+\lambda f_4)\subset U'_+,\ \bbf f_5\oplus \bbf f_6\subset U'_-)$

{\rm (iv)} $(\bbf(f_1+f_3)\oplus \bbf(f_2+\lambda f_4)\subset U_+,\ \bbf f_5\subset U'_-)$

{\rm (v)} $(\bbf(f_1+f_3)\oplus \bbf(f_2+\lambda f_4)\subset U'_+,\ \bbf f_5\subset U'_-)$

{\rm (vi)} $(\bbf(f_2+\lambda f_4)\subset U'_+,\ \bbf f_5\subset U'_-)$

Suppose $gV=V,\ gU_+^{1,\lambda}=U_+^{1,\mu},\ gU_+^2=U_+^2,\ gU_-^1=U_-^1$ and $gU_-^2=U_-^2$ for some $g\in G'$. Then $\lambda=\mu$.
\label{lem3.2}
\end{lemma}

\begin{proof} Since $gV=V$, we have $g\in G'_0$ by Corollary \ref{cor2.4}. It follows from Proposition \ref{prop2.6} that $gU_+=U_+$ and that $gU_-=U_-$ for all the six cases (i) through (vi) because 
$U'_+$ and $U'_-$
are three dimensional subspaces of $U_+$ and $U_-$, respectively. Hence $g\in K\ (\cong {\rm Sp}_4(\bbf))$. So the cases (ii) and (iii) are reduced to (i). Also the case (v) is reduced to (iv).

(i) Since $g(\bbf f_5\oplus \bbf f_6)=\bbf f_5\oplus \bbf f_6$ and since $\bbf f_1\oplus\bbf f_2$ is the orthogonal space of $\bbf f_5\oplus \bbf f_6$ in $U_+$, we have
$$g(\bbf f_1\oplus\bbf f_2)=\bbf f_1\oplus\bbf f_2.$$
Moreover we have
$$g(\bbf f_3\oplus\bbf f_4)=\bbf f_3\oplus\bbf f_4$$
since $\bbf f_3\oplus\bbf f_4$ is the orthogonal space of $\bbf f_1\oplus\bbf f_2$ with respect to the alternating form $\langle\ ,\ \rangle$. Hence we can write
$$g=\ell\left(\bp B & 0 \\ 0 & C \ep\right)$$
with some $B,C\in {\rm SL}_2(\bbf)$.

Consider the linear map $h_\lambda: \bbf f_1\oplus\bbf f_2\to \bbf f_3\oplus\bbf f_4$ defined by
$$h_\lambda f_1=f_3\mand h_\lambda f_2=\lambda f_4.$$
Then we can write
$$U_+^{1,\lambda}=\bbf(f_1+f_3)\oplus \bbf(f_2+\lambda f_4) =\{u+h_\lambda(u)\mid u\in \bbf f_1\oplus\bbf f_2\}.$$
Since $gU_+^{1,\lambda}=U_+^{1,\mu}$, we have
$$Ch_\lambda B^{-1}=h_\mu.$$
Taking the determinant, we have
$$\det h_\lambda=\det h_\mu$$
since $\det B=\det C=1$. Since $\det h_\lambda =\lambda$ for $\lambda\in\bbf^\times$, we have $\lambda=\mu$.

(iv) The element $g\in K$ stabilizes $\bbf f_1$ and $\bbf f_1\oplus\bbf f_2\oplus\bbf f_3$ since they are the orthogonal spaces of $\bbf f_5\oplus \bbf f_6\oplus \bbf f_7$ and $\bbf f_5$ in $U_+$, respectively. Since $\bbf f_3=\{v\in U_+\mid \langle v, \bbf f_1\oplus\bbf f_2\oplus\bbf f_3\rangle=\{0\}\}$, we also have $g\bbf f_3=\bbf f_3$.

Since $gU_+^{1,\lambda}=U_+^{1,\mu}$ and since $\bbf(f_1+f_3)=U_+^{1,\lambda} \cap(\bbf f_1\oplus\bbf f_3)=U_+^{1,\mu} \cap(\bbf f_1\oplus\bbf f_3)$, we have
$$g\bbf(f_1+f_3)=\bbf(f_1+f_3).$$
Hence we can write
$$gf_1=kf_1\mand gf_3=kf_3$$
with some $k\in\bbf^\times$.

We can write
$$g(f_2+\lambda f_4)=a(f_1+f_3)+b(f_2+\mu f_4)$$
with some $a,b\in \bbf$. Since
$$1=\langle f_1,f_2+\lambda f_4\rangle=\langle gf_1,g(f_2+\lambda f_4)\rangle =\langle kf_1,a(f_1+f_3)+b(f_2+\mu f_4)\rangle=kb,$$
we have $b=k^{-1}$. On the other hand, we have
$$\lambda=\langle f_3,f_2+\lambda f_4\rangle=\langle gf_3,g(f_2+\lambda f_4)\rangle =\langle kf_3,a(f_1+f_3)+k^{-1}(f_2+\mu f_4) \rangle=\mu.$$

(vi) As in (iv), we have
$gf_1=kf_1$ and $gf_3=kf_3$
with some $k\in\bbf^\times$. We can write
$$g(f_2+\lambda f_4)=b(f_2+\mu f_4)$$
with some $b\in \bbf$. Since
$$1=\langle f_1,f_2+\lambda f_4\rangle=\langle gf_1,g(f_2+\lambda f_4)\rangle =\langle kf_1,b(f_2+\mu f_4)\rangle=kb,$$
we have $b=k^{-1}$. On the other hand, we have
$$\lambda=\langle f_3,f_2+\lambda f_4\rangle=\langle gf_3,g(f_2+\lambda f_4)\rangle =\langle kf_3,k^{-1}(f_2+\mu f_4) \rangle=\mu.$$
\end{proof}

\begin{corollary} Following six triple flag varieties of ${\rm O}_8(\bbf)$ are of infinite type.
$$\mct_{(4),(2,2),(2,2)},\ \mct_{(4),(2,2),(2,1)},\ \mct_{(4),(2,1),(2,1)},\
\mct_{(4),(2,2),(1,2)},\ \mct_{(4),(2,1),(1,2)},\ \mct_{(4),(1,2),(1,2)}.$$
\end{corollary}

\bigskip
\noindent {\it Proof of Proposition \ref{prop1.5}}: First suppose $q=r=2$. Suppose that
\begin{equation}
{\bf b}\mbox{ and }{\bf c}\mbox{ are not equal to }(1,1),\ (1,n-1),\ (n-1,1). \label{eq3.1}
\end{equation}
Then we will prove that $\mct$ is of infinite type. Put
$$\beta'_1=\min(\beta_1,2),\ \beta'_2=\min(\beta_2,2),\ \gamma'_1=\min(\gamma_1,2)\mand \gamma'_2=\min(\gamma_2,2).$$
Then the condition (\ref{eq3.1}) is equivalent to
$$(\beta'_1,\beta'_2),\ (\gamma'_1,\gamma'_2)\ne (1,1),\  \beta_1+\beta_2-\beta'_1-\beta'_2\le n-4,\ \gamma_1+\gamma_2-\gamma'_1-\gamma'_2\le n-4.$$
Exchanging the roles of ${\bf b}$ and ${\bf c}$ if necessary, we have only to consider the following six cases of $((\beta'_1,\beta'_2),(\gamma'_1,\gamma'_2))$:

(i)\ $((2,2),(2,2))$,\quad (ii)\ $((2,2),(2,1))$,\quad (iii)\ $((2,1),(2,1))$,

(iv)\ $((2,2),(1,2))$,\quad (v)\ $((2,1),(1,2))$,\quad (vi)\ $((1,2),(1,2))$.

Let $\varphi: \bbf^8\to\bbf^{2n}$ denote the linear inclusion defined by $\varphi(f_i)=e_{i+n-4}$ for $i=1,\ldots,8$. Define subspaces $W_0,\ W_1^\lambda\ (\lambda\in\bbf^\times),\ W_2,\ W_3$ and $W_4$ of $\bbf^{2n}$ by
\begin{align*}
W_0 & =\varphi(V)\oplus U_{[n-4]},\ W_1^\lambda=\varphi(U_+^{1,\lambda})\oplus U_{[\beta_1-\beta'_1]},\ W_2=\varphi(U_+^2)\oplus U_{[\beta_1+\beta_2-\beta'_1-\beta'_2]}, \\
W_3 & =\varphi(U_-^1)\oplus U_{[\gamma_1-\gamma'_1]} \mand W_4=\varphi(U_-^2)\oplus U_{[\gamma_1+\gamma_2-\gamma'_1-\gamma'_2]}
\end{align*}
where $U_+^{1,\lambda},U_+^2,U_-^1$ and $U_-^2$ are defined in Lemma \ref{lem3.2} for each cases (i) through (vi). Then
$$m_\lambda=(W_0,\ W_1^\lambda\subset W_2,\ W_3\subset W_4)$$
is an element of $\mct_{(n),(\beta_1,\beta_2),(\gamma_1,\gamma_2)}$. By lemma \ref{lem2.7} and Lemma \ref{lem3.2}, we see that $Gm_\lambda\ni m_\mu\Longrightarrow \lambda=\mu$.

Next consider the case of $q=2$ and $r\ge 3$. Suppose that
$${\bf b}\mbox{ is not equal to }(1,1),\ (1,n-1),\ (n-1,1).$$
Then we will prove that $\mct=\mct_{(n),{\bf b},{\bf c}}$ is of infinite type.

If $\gamma_1+\gamma_2+\gamma_3<n$ or $\gamma_3>1$, then ${\bf c}'=(\gamma_1+\gamma_2,\gamma_3)$ is not equal to $(1,1),\ (1,n-1),\ (n-1,1)$. So the assertion is reduced to the case of $q=r=2$. On the other hand if $\gamma_1>1$, then ${\bf c}'=(\gamma_1,\gamma_2+\gamma_3)$ is not equal to $(1,1),\ (1,n-1),\ (n-1,1)$. So the assertion is also reduced to the case of $q=r=2$. Thus we have only to consider the remaining case that $\gamma_1+\gamma_2+\gamma_3=n$ with $\gamma_1=\gamma_3=1$. In this case, ${\bf c}'=(\gamma_1,\gamma_2)$ is not equal to $(1,1),\ (1,n-1),\ (n-1,1)$ because $n\ge 4$. So the assertion is also reduced to the case of $q=r=2$.

The cases of $3\le q\le r$ are similar. \hfill$\square$

\subsection{Case of ${\bf a}=(n),\ {\bf b}=(\beta)$ with $4\le \beta\le n-2$} \label{sec3.5}

First we prepare lemmas for $G'={\rm O}_{12}(\bbf)$. 
Consider $\bbf^{12}$ with the canonical basis $f_1,\ldots,f_{12}$ and the symmetric bilinear form $(\ ,\ )$ such that $(f_i,f_j)=\delta_{i,13-j}$. Take isotropic subspaces
\begin{align*}
U_+ & =\bbf f_1\oplus \bbf f_2\oplus \bbf f_3\oplus \bbf f_4\oplus \bbf f_5\oplus \bbf f_6, \\
U_- & =\bbf f_7\oplus \bbf f_8\oplus \bbf f_9\oplus \bbf f_{10}\oplus \bbf f_{11}\oplus \bbf f_{12} \\
 \mand
V & =\bbf(f_1+f_{11})\oplus \bbf(f_2-f_{12})\oplus \bbf(f_3+f_9)\oplus \bbf(f_4-f_{10})
\end{align*}
of $\bbf^{12}$. Note that the space $V$ is written as
$$V=\{u+\kappa(u)\mid u\in W_+\}$$
where $W_+=\bbf f_1\oplus \bbf f_2\oplus \bbf f_3\oplus \bbf f_4,\ W_-=\bbf f_9\oplus \bbf f_{10}\oplus \bbf f_{11}\oplus \bbf f_{12}$ and $\kappa: W_+\to W_-$ is the linear bijection defined by
$$\kappa(f_1)=f_{11},\quad \kappa(f_2)=-f_{12},\quad \kappa(f_3)=f_9,\mand \kappa(f_4)=-f_{10}.$$
Define the alternating bilinear form $\langle\ ,\ \rangle$ on $W_+$ as in Section \ref{sec3.1} and write
$$K_0=\{k\in {\rm GL}(W_+)\mid \langle gu,gv\rangle =\langle u,v\rangle\mbox{ for }u,v\in W_+\}\ (\cong {\rm Sp}_4(\bbf)).$$
Define a subgroup $K$ of $G'$ by
$$K=\left\{\ell\left(\bp A & 0 \\ 0 & B \ep\right)\Bigm| A\in K_0,\ B\in {\rm GL}_2(\bbf)\right\}$$
where
$$\ell(C)=\bp C & 0 \\ 0 & J_6{}^tC^{-1}J_6 \ep$$
for $C\in {\rm GL}_6(\bbf)$.

\begin{lemma} $K=\{g\in G'\mid gU_+=U_+,\ gU_-=U_-,\ gV=V\}$
\end{lemma}

\begin{proof} Suppose that $g\in G'$ satisfies $gU_+=U_+$ and $gU_-=U_-$. Then $g=\ell(C)$ with some $C\in {\rm GL}_6(\bbf)$. Moreover suppose $gV=V$. Then $g$ stabilizes $W_+$ and $W_-$. It also stabilizes $\bbf f_5\oplus\bbf f_6=\{v\in U_+\mid (v,W_-)=\{0\}\}$. Hence we can write
$$g=\ell \left(\bp A & 0 \\ 0 & B \ep\right)$$
with some $A\in {\rm GL}_4(\bbf)$ and $B\in {\rm GL}_2(\bbf)$. Since the map $A: W_+\to W_+$ preserves the alternating form $\langle\ ,\ \rangle$, we have $A\in K_0$.
\end{proof}

\begin{lemma} Let $m_\lambda: U_+^{1,\lambda}\subset U_+^2\subset U_+^3$ be one of the following four kinds of flags $(\lambda\in\bbf^\times)$ with $U'_+=\bbf f_1\oplus \bbf f_2\oplus \bbf f_3\oplus \bbf f_5\oplus \bbf (f_4+f_6)$.

{\rm (i)} $\bbf (\lambda f_1+f_3+f_5)\oplus \bbf(f_2+f_4+f_6)\subset \bbf f_1\oplus\bbf f_2\oplus \bbf(f_3+f_5)\oplus \bbf(f_4+f_6)\subset U_+$

{\rm (ii)} $\bbf (\lambda f_1+f_3+f_5)\oplus \bbf(f_2+f_4+f_6)\subset \bbf f_1\oplus\bbf f_2\oplus \bbf(f_3+f_5)\oplus \bbf(f_4+f_6)\subset U'_+$

{\rm (iii)} $\bbf (\lambda f_1+f_3+f_5)\oplus \bbf(f_2+f_4+f_6)\subset \bbf f_1\oplus \bbf(f_3+f_5)\oplus \bbf(f_2+f_4+f_6)\subset U'_+$

{\rm (iv)} $\bbf (\lambda f_1+f_3+f_5)\subset \bbf f_1\oplus \bbf(f_3+f_5)\oplus \bbf(f_2+f_4+f_6)\subset U'_+$

Suppose $gU_-=U_-,\ gV=V$ and $gm_\lambda=m_\mu$ for some $g\in G'$. Then we have $\lambda=\mu$.
\label{lem3.10}
\end{lemma}

\begin{proof} Since $gV=V$, we have $g\in G'_0$ by Corollary \ref{cor2.4}. Since $U'_+$ is five-dimensional, the condition $gU'_+=U'_+$ implies $gU_+=U_+$ by Proposition \ref{prop2.6}. So we may assume $g\in K$.

(i) Since
$$W_+\cap U_+^2=W_+\cap (\bbf f_1\oplus\bbf f_2\oplus \bbf(f_3+f_5)\oplus \bbf(f_4+f_6))=\bbf f_1\oplus\bbf f_2,$$
we have $g(\bbf f_1\oplus\bbf f_2)=\bbf f_1\oplus\bbf f_2$. We also have $g(\bbf f_3\oplus\bbf f_4)=\bbf f_3\oplus\bbf f_4$ since $\bbf f_3\oplus\bbf f_4$ is the orthogonal space of $\bbf f_1\oplus\bbf f_2$ with respect to the alternating form $\langle\ ,\ \rangle$. Hence we can write
$$g=\ell\left(\bp A & 0 & 0 \\ 0 & B & 0 \\ 0 & 0 & C \ep\right)$$
with some $A,B\in {\rm SL}_2(\bbf)$ and $C\in {\rm GL}_2(\bbf)$. (Since $g(\bbf(f_3+f_5)\oplus \bbf(f_4+f_6))=\bbf(f_3+f_5)\oplus \bbf(f_4+f_6)$, we also have $B=C$.) By the same argument as in the proof of Lemma \ref{lem3.2} (i), we can show $\lambda=\mu$.

The assertion (ii) follows from (i).

(iii) The element $g\in K$ stabilizes $\bbf f_1$ and $\bbf f_1\oplus\bbf f_2\oplus\bbf f_3$ since
$$W_+\cap U_+^2=W_+\cap (\bbf f_1\oplus \bbf(f_3+f_5)\oplus \bbf(f_2+f_4+f_6))=\bbf f_1$$
and since $W_+\cap U_+^3=\bbf f_1\oplus\bbf f_2\oplus\bbf f_3$. Since $\bbf f_3=\{v\in U_+\mid \langle v, \bbf f_1\oplus\bbf f_2\oplus\bbf f_3\rangle=\{0\}\}$, we also have $g\bbf f_3=\bbf f_3$. Hence we have
$$gf_1=kf_1\mand gf_3=\ell f_3$$
with some $k,\ell\in\bbf^\times$.

Let $\pi:U_+=W_+\oplus(\bbf f_5\oplus\bbf f_6)\to W_+$ denote the canonical projection. Since $g$ stabilizes $\pi(U_+^2)=\bbf f_1\oplus\bbf f_3\oplus \bbf(f_2+f_4)$, we can write
$$g(f_2+f_4)=af_1+bf_3+c(f_2+f_4)$$
with some $a,b,c\in\bbf$. We have
$$1=\langle f_1,f_2+f_4\rangle=\langle gf_1,g(f_2+f_4)\rangle=\langle kf_1,af_1+bf_3+c(f_2+f_4)\rangle =kc$$
and
$$1=\langle f_3,f_2+f_4\rangle=\langle gf_3,g(f_2+f_4)\rangle=\langle \ell f_3,af_1+bf_3+c(f_2+f_4)\rangle =\ell c.$$
So we have $k=\ell=c^{-1}$.

Since $g\pi(U_+^{1,\lambda})=\pi(U_+^{1,\mu})$, we have
$$g(\lambda f_1+f_3)\subset \bbf(\mu f_1+f_3)\oplus\bbf(f_2+f_4).$$
Since $g(\lambda f_1+f_3)=k(\lambda f_1+f_3)$, we have $\lambda=\mu$.

(iv) As is shown in (iii), we have $gf_1=kf_1$ and $gf_3=kf_3$ with some $k\in\bbf^\times$. Since $g\pi(U_+^{1,\lambda})=\pi(U_+^{1,\mu})$, we have
$$g(\lambda f_1+f_3)\subset \bbf(\mu f_1+f_3).$$
Since $g(\lambda f_1+f_3)=k(\lambda f_1+f_3)$, we have $\lambda=\mu$.
\end{proof}

\begin{lemma} Let $m_\lambda: U_+^{1,\lambda}\subset U_+^2\subset U_+^3\subset U_+^4$ be the flag defined by
\begin{align*}
U_+^{1,\lambda} & =\bbf (\lambda f_1+f_3+f_5), \ 
U_+^2 =\bbf f_1\oplus \bbf(f_3+f_5)\oplus \bbf(f_2+f_4+f_6), \\
U_+^3 & =\bbf f_1\oplus\bbf f_2\oplus \bbf(f_3+f_5)\oplus \bbf(f_4+f_6) 
\mand U_+^4 =U_+.
\end{align*}
Suppose $gU_-=U_-,\ gV=V$ and $gm_\lambda=m_\mu$ for some $g\in G'$. Then we have $\lambda=\mu$.
\label{lem3.11}
\end{lemma}

\begin{proof} Since $gU_+^3=U_+^3$, we have
$$g=\ell\left(\bp A & 0 & 0 \\ 0 & B & 0 \\ 0 & 0 & B \ep\right)$$
with some $A,B\in {\rm SL}_2(\bbf)$ as in the proof of Lemma \ref{lem3.10} (i). Since $gU_+^2=U_+^2$, we have $g\bbf f_1=\bbf f_1$ and $g\bbf(f_3+f_5)=\bbf(f_3+f_5)$. Hence
$$A=\bp a & * \\ 0 & a^{-1} \ep \mand B=\bp b & * \\ 0 & b^{-1} \ep$$
with some $a,b\in\bbf^\times$. On the other hand, since $g\bbf(f_2+f_4+f_6)\subset U_+^2$, we have $a^{-1}=b^{-1}$ and hence $a=b$. Since $U_+^{1,\mu}=gU_+^{1,\lambda}=\bbf(a\lambda f_1+af_3+af_5)=U_+^{1,\lambda}$, we have $\lambda=\mu$.
\end{proof}

\begin{corollary} Following five triple flag varieties of ${\rm O}_{12}(\bbf)$ are of infinite type.
$$\mct_{(6),(4),(2,2,2)},\quad \mct_{(6),(4),(2,2,1)},\quad \mct_{(6),(4),(2,1,2)},\quad \mct_{(6),(4),(1,2,2)},\quad \mct_{(6),(4),(1,2,1,2)}.$$
\end{corollary}

Suppose $4\le \beta\le n-2$ (hence $n\ge 6$). Let $\varphi: \bbf^{12}\to \bbf^{2n}$ denote the linear inclusion defined by $\varphi(f_i)=e_{i+n-6}$ for $i=1,\ldots,12$. Write $W_0=U_{[n-6]}\oplus\varphi(V)$ and $W_1= U_{[\beta-4]}\oplus\varphi(U_-)$.

\subsubsection{Case of $r=3$ and $\gamma_1+\gamma_2+\gamma_3=n$}

\begin{proposition} Suppose $\gamma_1,\gamma_2,\gamma_3\ge 2$. Then $\mct_{(n),(\beta),(\gamma_1,\gamma_2,\gamma_3)}$ is of infinite type.
\label{prop3.13}
\end{proposition}

\begin{proof} Define subspaces $W_{2,\lambda},\ W_3$ and $W_4$ of $\bbf^{2n}$ by
\begin{align*}
W_{2,\lambda} & =U_{[\gamma_1-2]}\oplus \varphi(\bbf (\lambda f_1+f_3+f_5)\oplus \bbf(f_2+f_4+f_6)), \\
W_3 & =U_{[\gamma_1+\gamma_2-4]}\oplus \varphi(\bbf f_1\oplus\bbf f_2\oplus \bbf(f_3+f_5)\oplus \bbf(f_4+f_6)) \\
\mand W_4 & =U_{[\gamma_1+\gamma_2+\gamma_3-6]}\oplus \varphi(U_+).
\end{align*}
Then
$m_\lambda=(W_0,\ W_1,\ W_{2,\lambda}\subset W_3\subset W_4)$
is an element of $\mct_{(n),(\beta),(\gamma_1,\gamma_2,\gamma_3)}$. By Lemma \ref{lem2.7} and Lemma \ref{lem3.10} (i), we have $Gm_\lambda\ni m_\mu \Longrightarrow \lambda=\mu$.
\end{proof}

\begin{corollary} Suppose ${\bf c}=(\gamma_1,\gamma_2,\gamma_3)$ with $\gamma_1+\gamma_2+\gamma_3=n$. Then $|G\backslash \mct_{(n),(\beta),{\bf c}}|<\infty \Longrightarrow$
$${\bf c}=(1,k,n-k-1),\quad (k,1,n-k-1)\quad\mbox{or}\quad (k,n-k-1,1)$$
with some $k$.
\label{cor3.14}
\end{corollary}

\subsubsection{Case of $r=3$ and $\gamma_1+\gamma_2+\gamma_3<n$}

\begin{proposition} Suppose $\gamma_1+\gamma_2+\gamma_3<n$ and
$$\min(\gamma_1,\gamma_2)\ge 2,\quad \min(\gamma_1,\gamma_3)\ge 2 \quad\mbox{or}\quad \min(\gamma_2,\gamma_3)\ge 2.$$
Then $\mct_{(n),(\beta),(\gamma_1,\gamma_2,\gamma_3)}$ is of infinite type.
\label{prop3.15}
\end{proposition}

\begin{proof} (i) Case of $\min(\gamma_1,\gamma_2)\ge 2$. Define subspaces $W_{2,\lambda},\ W_3$ and $W_4$ of $\bbf^{2n}$ by
\begin{align*}
W_{2,\lambda} & =U_{[\gamma_1-2]}\oplus \varphi(\bbf (\lambda f_1+f_3+f_5)\oplus \bbf(f_2+f_4+f_6)), \\
W_3 & =U_{[\gamma_1+\gamma_2-4]}\oplus \varphi(\bbf f_1\oplus\bbf f_2\oplus \bbf(f_3+f_5)\oplus \bbf(f_4+f_6)) \\
\mand W_4 & =U_{[\gamma_1+\gamma_2+\gamma_3-5]}\oplus \varphi(U'_+).
\end{align*}
Then
$m_\lambda=(W_0,\ W_1,\ W_{2,\lambda}\subset W_3\subset W_4)$
is an element of $\mct_{(n),(\beta),(\gamma_1,\gamma_2,\gamma_3)}$. By Lemma \ref{lem2.7} and Lemma \ref{lem3.10} (ii), we have $Gm_\lambda\ni m_\mu \Longrightarrow \lambda=\mu$.

(ii) Case of $\min(\gamma_1,\gamma_3)\ge 2$. Define subspaces $W_{2,\lambda},\ W_3$ and $W_4$ of $\bbf^{2n}$ by
\begin{align*}
W_{2,\lambda} & =U_{[\gamma_1-2]}\oplus \varphi(\bbf (\lambda f_1+f_3+f_5)\oplus \bbf(f_2+f_4+f_6)), \\
W_3 & =U_{[\gamma_1+\gamma_2-3]}\oplus \varphi(\bbf f_1\oplus \bbf(f_3+f_5)\oplus \bbf(f_2+f_4+f_6)) \\
\mand W_4 & =U_{[\gamma_1+\gamma_2+\gamma_3-5]}\oplus \varphi(U'_+).
\end{align*}
Then
$m_\lambda=(W_0,\ W_1,\ W_{2,\lambda}\subset W_3\subset W_4)$
is an element of $\mct_{(n),(\beta),(\gamma_1,\gamma_2,\gamma_3)}$. By Lemma \ref{lem2.7} and Lemma \ref{lem3.10} (iii), we have $Gm_\lambda\ni m_\mu \Longrightarrow \lambda=\mu$.

(iii) Case of $\min(\gamma_2,\gamma_3)\ge 2$. Define subspaces $W_{2,\lambda},\ W_3$ and $W_4$ of $\bbf^{2n}$ by
\begin{align*}
W_{2,\lambda} & =U_{[\gamma_1-1]}\oplus \varphi(\bbf (\lambda f_1+f_3+f_5)), \\
W_3 & =U_{[\gamma_1+\gamma_2-3]}\oplus \varphi(\bbf f_1\oplus \bbf(f_3+f_5)\oplus \bbf(f_2+f_4+f_6)) \\
\mand W_4 & =U_{[\gamma_1+\gamma_2+\gamma_3-5]}\oplus \varphi(U'_+).
\end{align*}
Then
$m_\lambda=(W_0,\ W_1,\ W_{2,\lambda}\subset W_3\subset W_4)$
is an element of $\mct_{(n),(\beta),(\gamma_1,\gamma_2,\gamma_3)}$. By Lemma \ref{lem2.7} and Lemma \ref{lem3.10} (iv), we have $Gm_\lambda\ni m_\mu \Longrightarrow \lambda=\mu$.
\end{proof}

\begin{corollary} Suppose ${\bf c}=(\gamma_1,\gamma_2,\gamma_3)$ with $\gamma_1+\gamma_2+\gamma_3<n$. Then $|G\backslash \mct_{(n),(\beta),{\bf c}}|<\infty \Longrightarrow$
$${\bf c}=(1,1,k),\quad (1,k,1)\quad\mbox{or}\quad (k,1,1)$$
with some $k$.
\label{cor3.16}
\end{corollary}

\subsubsection{Case of $r=4$ and $\gamma_1+\gamma_2+\gamma_3+\gamma_4=n$}

\begin{proposition} Suppose $\gamma_1+\gamma_2+\gamma_3+\gamma_4=n$ and $\min(\gamma_i,\gamma_j)\ge 2$ with some $1\le i<j\le 4$. Then $\mct_{(n),(\beta),(\gamma_1,\gamma_2,\gamma_3,\gamma_4)}$ is of infinite type.
\label{prop3.17}
\end{proposition}

\begin{proof} The cases of $(i,j)=(1,2),\ (1,3)$ and $(2,3)$ are reduced to Proposition \ref{prop3.15}. The case of $(i,j)=(1,4)$ is reduced to Proposition \ref{prop3.13} if we consider ${\bf c}'=(\gamma_1,\gamma_2+\gamma_3,\gamma_4)$. The case of $(i,j)=(3,4)$ is also reduced to Proposition \ref{prop3.13} if we consider ${\bf c}'=(\gamma_1+\gamma_2,\gamma_3,\gamma_4)$.

So we have only to consider the case of $(i,j)=(2,4)$. Define subspaces $W_{2,\lambda},\ W_3,\ W_4$ and $W_5$ of $\bbf^{2n}$ by
\begin{align*}
W_{2,\lambda} & =U_{[\gamma_1-1]}\oplus \varphi(\bbf (\lambda f_1+f_3+f_5)), \\
W_3 & =U_{[\gamma_1+\gamma_2-3]}\oplus \varphi(\bbf f_1\oplus \bbf(f_3+f_5)\oplus \bbf(f_2+f_4+f_6)) \\
W_4 & =U_{[\gamma_1+\gamma_2-4]}\oplus \varphi(\bbf f_1\oplus\bbf f_2\oplus \bbf(f_3+f_5)\oplus \bbf(f_4+f_6)) \\
\mand W_5 & =U_{[\gamma_1+\gamma_2+\gamma_3+\gamma_4-6]}\oplus \varphi(U_+).
\end{align*}
Then
$m_\lambda=(W_0,\ W_1,\ W_{2,\lambda}\subset W_3\subset W_4\subset W_5)$
is an element of $\mct_{(n),(\beta),(\gamma_1,\gamma_2,\gamma_3,\gamma_4)}$. By Lemma \ref{lem2.7} and Lemma \ref{lem3.11}, we have $Gm_\lambda\ni m_\mu \Longrightarrow \lambda=\mu$.
\end{proof}

\begin{corollary} Suppose ${\bf c}=(\gamma_1,\gamma_2,\gamma_3,\gamma_4)$ with $\gamma_1+\gamma_2+\gamma_3+\gamma_4=n$. Then $|G\backslash \mct_{(n),(\beta),{\bf c}}|<\infty \Longrightarrow$
$${\bf c}=(1,1,1,n-3),\quad (1,1,n-3,1),\quad (1,n-3,1,1)\quad\mbox{or}\quad (n-3,1,1,1).$$
\label{cor3.18}
\end{corollary}

\subsubsection{Case of $r=4$ and $\gamma_1+\gamma_2+\gamma_3+\gamma_4<n$}

\begin{proposition} Suppose $\gamma_1+\gamma_2+\gamma_3+\gamma_4<n$ and $\gamma_i\ge 2$ with some $i=1,2,3,4$. Then $\mct_{(n),(\beta),(\gamma_1,\gamma_2,\gamma_3,\gamma_4)}$ is of infinite type.
\label{prop3.19}
\end{proposition}

\begin{proof} If $\gamma_i\ge 2$ for $i=1$ or $2$, then $|G\backslash \mct_{(n),(\beta),(\gamma_1,\gamma_2,\gamma_3+\gamma_4)}|=\infty$ by Proposition \ref{prop3.15} (ii) or (iii), respectively. If $\gamma_i\ge 2$ for $i=3$ or $4$, then $|G\backslash \mct_{(n),(\beta),(\gamma_1+\gamma_2,\gamma_3,\gamma_4)}|=\infty$ by Proposition \ref{prop3.15} (i) or (ii), respectively.
\end{proof}

\begin{corollary} Suppose ${\bf c}=(\gamma_1,\gamma_2,\gamma_3,\gamma_4)$ with $\gamma_1+\gamma_2+\gamma_3+\gamma_4<n$. Then
$$|G\backslash \mct_{(n),(\beta),{\bf c}}|<\infty \Longrightarrow {\bf c}=(1,1,1,1).$$
\label{cor3.20}
\end{corollary}

\subsubsection{Case of $r\ge 5$}

\begin{proposition} Suppose $r\ge 5$. Then $\mct_{(n),(\beta),{\bf c}}$ is of infinite type.
\end{proposition}

\begin{proof} (A) Case of $\gamma_i\ge 2$ for some $i=1,2,3,4$. We have $|G\backslash \mct_{(n),(\beta),(\gamma_1,\gamma_2,\gamma_3,\gamma_4)}|=\infty$ by Proposition \ref{prop3.19}. So we have $|G\backslash \mct_{(n),(\beta),{\bf c}}|=\infty$.

(B) Case of $\gamma_5\ge 2$. We have $|G\backslash \mct_{(n),(\beta),(\gamma_1+\gamma_2,\gamma_3,\gamma_4,\gamma_5)}|=\infty$ by Proposition \ref{prop3.17}. So we have $|G\backslash \mct_{(n),(\beta),{\bf c}}|=\infty$.

(C) Case of $\gamma_1=\gamma_2=\gamma_3=\gamma_4=\gamma_5=1$. Since $|G\backslash \mct_{(n),(\beta),(2,1,1,1)}|=\infty$ by Proposition \ref{prop3.19}, we have $|G\backslash \mct_{(n),(\beta),(1^5)}|=\infty$.
\end{proof}

\subsection{Proof of Proposition \ref{prop1.6} (ii)}

Consider $\bbf^6$ with the canonical basis $f_1,\ldots,f_6$ and the symmetric bilinear form $(\ ,\ )$ such that $(f_i,f_j)=\delta_{i,7-j}$. Write $G'={\rm O}_6(\bbf)$ and $G'_0={\rm SO}_6(\bbf)$. Define triple flags $m_\lambda=(V,\ U_-^1\subset U_-^2,\ U_+^1\subset U_+^{2,\lambda})\in \mct_{(3),(1,1),(1,1)}$ for $\lambda\in\bbf^\times$ by
\begin{align*}
U_-^1 & =\bbf f_4,\ U_-^2=\bbf f_4\oplus\bbf f_5,\ V=\bbf(f_2+f_4)\oplus \bbf(f_3-f_5)\oplus\bbf f_6, \\
U_+^1 & =\bbf (f_1+f_3)\mand U_+^{2,\lambda}=\bbf (f_1+f_3)\oplus \bbf (\lambda f_2+f_3).
\end{align*}

\begin{lemma} If $G'm_\lambda\ni m_\mu$, then $\lambda/\mu\in (\bbf^\times)^2$.
\label{lem3.22}
\end{lemma}

\begin{proof} Suppose $gm_\lambda=m_\mu$ for some $g\in G'$. Since $gV=V$, we have $g\in G'_0$ by Corollary \ref{cor2.4}. Since $U_+=\bbf f_1\oplus\bbf f_2\oplus\bbf f_3$ and $U_-=\bbf f_4\oplus\bbf f_5\oplus\bbf f_6$ are maximal isotropic subspaces of $\bbf^6$ containing $U_+^{2,\lambda}$ and $U_-^2$, respectively, we have
$$gU_+=U_+\mand gU_-=U_-$$
by Proposition \ref{prop2.6}. Since $\bbf f_1$ is the orthogonal space of $U_-^2$ in $U_+$, we have $g\bbf f_1=\bbf f_1$. On the other hand, $g$ preserves $U_+\cap (V+U_-)=\bbf f_2\oplus\bbf f_3$. So we can write
$$g=\ell\left(\bp c & 0 \\ 0 & A \ep\right)$$
with some $c\in\bbf^\times$ and $A\in{\rm GL}_2(\bbf)$. Since $gV=V$, we have $A\in {\rm SL}_2(\bbf)$. Since $gU_-^1=U_-^1$, we have
$$A=\bp a & x \\ 0 & a^{-1} \ep$$
with some $a\in\bbf^\times$ and $x\in\bbf$. Since $gU_+^1=U_+^1$, we have $c=a^{-1}$ and $x=0$. Hence we have
$$g=\ell\left(\bp c & 0 & 0 \\ 0 & c^{-1} & 0 \\ 0 & 0 & c \ep\right).$$
Since $gU_+^{2,\lambda}=U_+^{2,\mu}$, we have
$$g(\lambda f_2+f_3)=\lambda c^{-1}f_2+cf_3\in U_+^{2,\mu}.$$
Hence we have $\lambda/\mu=c^2\in (\bbf^\times)^2$.
\end{proof}

Now we prove Proposition \ref{prop1.6} (ii): Define a linear inclusion $\varphi: \bbf^6\to \bbf^{2n}$ by $\varphi(f_i)=e_{i+n-3}$. Define triple flags $n_\lambda=(W_0,\ W_1\subset W_2,\ W_3\subset W_{4,\lambda})\in \mct_{(n),(\beta_1,\beta_2),(\gamma_1,\gamma_2)}$ by
\begin{align*}
W_0 & =U_{[n-3]}\oplus\varphi(V),\ W_1=U_{[\beta_1-1]}\oplus\varphi(U_-^1),\ W_{2,\lambda}=U_{[\beta_1+\beta_2-2]} \oplus\varphi(U_-^2) \\
W_3 & =U_{[\gamma_1-1]}\oplus\varphi(U_+^1)\mand W_{4,\lambda}=U_{[\gamma_1+\gamma_2-2]}\oplus \varphi(U_+^{2,\lambda}).
\end{align*}
Then we have $Gn_\lambda\ni n_\mu\Longrightarrow \lambda/\mu\in (\bbf^\times)^2$ by Lemma \ref{lem2.7} and Lemma \ref{lem3.22}. \hfill$\square$

\subsection{Proof of Proposition \ref{prop1.6} (iii)}

Consider $\bbf^{10}$ with the canonical basis $f_1,\ldots,f_{10}$ and the symmetric bilinear form $(\ ,\ )$ such that $(f_i,f_j)=\delta_{i,11-j}$. Write $G'={\rm O}_{10}(\bbf)$ and $G'_0={\rm SO}_{10}(\bbf)$. Define triple flags $m_\lambda=(V,\ U_-,\ U_+^1\subset U_+^{2,\lambda} \subset U_+^3\subset U_+^4)\in \mct_{(5),(3),(1,1,1,1)}$ for $\lambda\in\bbf^\times$ by
\begin{align*}
V & =\bbf f_3\oplus\bbf(f_4+f_6)\oplus\bbf(f_5-f_7)\oplus \bbf f_9\oplus \bbf f_{10}, \quad 
U_- =\bbf f_6\oplus \bbf f_7\oplus \bbf f_8, \\
U_+^1 & =\bbf(f_1+f_2+f_3+f_5), \quad 
U_+^{2,\lambda} =\bbf(f_1+f_2+f_3+f_5)\oplus \bbf(\lambda f_4+f_1+f_5), \\
U_+^3 & =\bbf f_4\oplus \bbf(f_1+f_5)\oplus\bbf (f_2+f_3) \mand
U_+^4 =\bbf f_4\oplus \bbf f_1\oplus\bbf f_5\oplus\bbf (f_2+f_3).
\end{align*}

\begin{lemma} If $G'm_\lambda\ni m_\mu$, then $\lambda/\mu\in (\bbf^\times)^2$.
\label{lem3.23}
\end{lemma}

\begin{proof} Suppose $gm_\lambda=m_\mu$ for some $g\in G'$. Since $gV=V$, we have $g\in G'_0$ by Corollary \ref{cor2.4}. Since $U_+=\bbf f_1\oplus \bbf f_2\oplus \bbf f_3 \oplus \bbf f_4\oplus \bbf f_5$ is a maximally isotropic subspace of $\bbf^{10}$ containing $U_+^4$, we have $gU_+=U_+$ by Proposition \ref{prop2.6}.

Since $\bbf f_1\oplus\bbf f_2$ is the orthogonal space of $U_-$ in $U_+$, we have
\begin{equation}
g(\bbf f_1\oplus\bbf f_2)=\bbf f_1\oplus\bbf f_2. \label{eq3.4}
\end{equation}
Since $V\cap U_+=\bbf f_3$, we have
\begin{equation}
g(\bbf f_3)=\bbf f_3. \label{eq3.5}
\end{equation}
Since $(V\oplus U_-)\cap U_+=\bbf f_3\oplus\bbf f_4\oplus \bbf f_5$, we have
\begin{equation}
g(\bbf f_3\oplus\bbf f_4\oplus \bbf f_5)=\bbf f_3\oplus\bbf f_4\oplus \bbf f_5. \label{eq3.6}
\end{equation}
By (\ref{eq3.4}),\ (\ref{eq3.5}) and (\ref{eq3.6}), we have
$$g|_{U_+}=\bp A & 0 & 0 \\ 0 & c & * \\ 0 & 0 & B \ep$$
with some $c\in\bbf^\times$ and $A,B\in{\rm GL}_2(\bbf)$. Since $gV=V$, we have $B\in{\rm SL}_2(\bbf)$.

Since $gU_+^4=U_+^4$, we have
$$g|_{U_+}=\bp a & * & 0 & 0 \\ 0 & c & 0 & 0 \\ 0 & 0 & c & 0 \\ 0 & 0 & 0 & B \ep$$
with $a,c\in\bbf^\times$ and $B\in{\rm SL}_2(\bbf)$. Since $gU_+^3=U_+^3$, we have
$$g|_{U_+}=\bp a & 0 & 0 & 0 & 0 \\ 0 & c & 0 & 0 & 0 \\ 0 & 0 & c & 0 & 0 \\ 0 & 0 & 0 & a^{-1} & * \\ 0 & 0 & 0 & 0 & a \ep.$$
Since $gU_+^1=U_+^1$, we have $g|_{U_+}={\rm diag}(a,a,a,a^{-1},a)$. Since $gU_+^{2,\lambda}=U_+^{2,\mu}$, we have
$$g(\lambda f_4+f_1+f_5)=\lambda a^{-1}f_4+af_1+af_5\in U_+^{2,\mu}.$$
Hence $\lambda/\mu=a^2\in (\bbf^\times)^2$.
\end{proof}

Now we prove Proposition \ref{prop1.6} (iii): Define a linear inclusion $\varphi: \bbf^{10}\to \bbf^{2n}$ by $\varphi(f_i)=e_{i+n-5}$. Define triple flags $n_\lambda=(W_0,\ W_1,\ W_2\subset W_{3,\lambda}\subset W_4\subset W_5)\in \mct_{(n),(\beta),(\gamma_1,\gamma_2,\gamma_3,\gamma_4)}$ by
\begin{align*}
W_0 & =U_{[n-5]}\oplus\varphi(V),\quad W_1=U_{[\beta-3]}\oplus\varphi(U_-),\quad W_2=U_{[\gamma_1-1]} \oplus\varphi(U_+^1) \\
W_{3,\lambda} & =U_{[\gamma_1+\gamma_2-2]}\oplus\varphi(U_+^{2,\lambda}), \quad 
W_4 =U_{[\gamma_1+\gamma_2+\gamma_3-3]}\oplus\varphi(U_+^3) \\
\mand W_5 & =U_{[\gamma_1+\gamma_2+\gamma_3+\gamma_4-4]}\oplus \varphi(U_+^4).
\end{align*}
Then we have $Gn_\lambda\ni n_\mu\Longrightarrow \lambda/\mu\in (\bbf^\times)^2$ by Lemma \ref{lem2.7} and Lemma \ref{lem3.23}. \hfill$\square$

\section{Triple flag varieties of finite type}

We have only to prove the following three propositions to prove that triple flag varieties in (III-3) and (III-4) are of finite type because
$$gV'=V'\Longrightarrow gV=V$$
for $g\in G_0$ and a flag $V'\subset V$ of $n-1$ and $n$ dimensional isotropic subspaces $V'$ and $V$ by Corollary \ref{cor2.4}. (We exchange the roles of ${\bf a}$ and ${\bf c}$.)

For a triple ${\bf a}=(\alpha_1,\alpha_2,\alpha_3)$ of positive integers, write $|\bf a|=\alpha_1+\alpha_2+\alpha_3$.

\begin{proposition} If ${\bf a}=(\alpha_1,\alpha_2,\alpha_3)$ satisfies one of the following five conditions, then the triple flag variety $\mct_{{\bf a},(\beta),(n)}$ is of finite type.

{\rm (i)} $|{\bf a}|=n$ and $\alpha_1=1$.

{\rm (ii)} $|{\bf a}|=n$ and $\alpha_2=1$.

{\rm (iii)} $\alpha_1=\alpha_2=1$.

{\rm (iv)} $\alpha_2=\alpha_3=1$.

{\rm (v)} $\alpha_1=\alpha_3=1$.
\label{prop4.1}
\end{proposition}

\begin{proposition} The triple flag variety $\mct_{(1,1,1,n-3),(\beta),(n)}$ is of finite type.
\label{prop4.2'}
\end{proposition}

\begin{proposition} If $|\bbf^\times/(\bbf^\times)^2|<\infty$, then the triple flag variety $\mct_{(1,1,1,1),(\beta),(n)}$ is of finite type.
\label{prop4.3'}
\end{proposition}

We have only to prove the following four propositions to prove that triple flag varieties in (I-1),\ (I-2) and (II) are of finite type.

\begin{proposition} The triple flag varieties $\mct_{(1^n),(n-1),(n)}$ and $\mct_{(1^n),(1,n-1),(n)}$ are of finite type.
\label{prop4.2}
\end{proposition}

\begin{proposition} {\rm (i)} If $\beta\le 2$, then the triple flag variety $\mct_{(1^n),(\beta),(n)}$ is of finite type.

{\rm (ii)} If $|\bbf^\times/(\bbf^\times)^2|<\infty$, then the triple flag varieties $\mct_{(1^n),(3),(n)}$ is of finite type.

{\rm (iii)} If $r\ge 4$ and $\gamma_1+\gamma_2+\gamma_3+\gamma_4=n$, then the triple flag variety $\mct_{(\gamma_1,\gamma_2,\gamma_3,\gamma_4),(3),(n)}$ is of finite type.
\label{prop4.3}
\end{proposition}

\begin{proposition} {\rm (i)} If $|\bbf^\times/(\bbf^\times)^2|<\infty$, then the triple flag variety $\mct_{(1^n),(1,1),(n)}$ is of finite type.

{\rm (ii)} The triple flag variety $\mct_{(\gamma_1,n-\gamma_1),(1,1),(n)}$ is of finite type.
\label{prop4.3''}
\end{proposition}

\begin{proposition} If $|\bbf^\times/(\bbf^\times)^2|<\infty$, then the triple flag variety $\mct_{(\beta,n-\beta),(1),(1^n)}$ is of finite type.
\label{prop4.4}
\end{proposition}

In order to prove these propositions, we prepare in the next section technical tools which were introduced in \cite{M3} Section 3.

\section{Review of technical results in \cite{M3}}

In this section, we will prepare notations and technical results given in \cite{M3} which are also valid in the even-degree case with a slight modification.

\subsection{Normalization of $U_+$ and $U_-$}

For $i\in I=\{1,\ldots,2n\}$, write $\overline{i}=2n+1-i$. Let $U_+$ and $U_-$ be $\alpha$ and $\beta$-dimensional isotropic subspaces of $\bbf^{2n}$, respectively. Define subspaces
$$W_0=U_+\cap U_-,\quad W_+=U_+\cap U_-^\perp\mand W_-=U_-\cap U_+^\perp$$
of $\bbf^{2n}$. Write
$a_0 =\dim W_0, \ 
a_+  =\dim W_+ -a_0$ and $a_- =\dim W_- -a_0$.
Since the bilinear form $(\ ,\ )$ is nondegenerate on the pair $(U_+/W_+,U_-/W_-)$, we have
$$\alpha-a_0-a_+=\beta-a_0-a_-.$$
Put $a_1=\alpha-a_0-a_+=\beta-a_0-a_-$ and $d=a_0+a_++a_-$. Define subspaces
\begin{align*}
W_{(0)} & =\bbf e_1\oplus\cdots\oplus \bbf e_{a_0}, \quad 
W_{(+)} =\bbf e_{a_0+1}\oplus\cdots\oplus\bbf e_{a_0+a_+}, \\ 
W_{(-)} & =\bbf e_{a_0+a_++1}\oplus\cdots\oplus\bbf e_d, \quad
U_{(+)} =\bbf e_{d+1}\oplus\cdots\oplus\bbf e_{d+a_1} \\
\quad\mbox{and}\quad U_{(-)} & =\bbf e_{\overline{d+a_1}}\oplus\cdots\oplus\bbf e_{\overline{d+1}} =\bbf e_{d'+1}\oplus\cdots\oplus\bbf e_{d'+a_1}
\end{align*}
of $\bbf^{2n}$ where $d'=\overline{d+a_1}-1=2n-d-a_1$.

\begin{proposition} $($Proposition 3.6 in \cite{M3}$)$ There exists an element $g\in G$ such that
$$gU_+=W_{(0)}\oplus W_{(+)}\oplus U_{(+)}\quad\mbox{and that}\quad gU_-=W_{(0)}\oplus W_{(-)}\oplus U_{(-)}.$$
\label{prop5.1}
\end{proposition}

So we may assume 
$$U_+=W_{(0)}\oplus W_{(+)}\oplus U_{(+)}\mand U_-=W_{(0)}\oplus W_{(-)}\oplus U_{(-)}$$
in the following. Hence we have
$$W_0=W_{(0)},\quad W_+=W_{(0)}\oplus W_{(+)}\mand W_-=W_{(0)}\oplus W_{(-)}.$$
Write
$$R=\{g\in G\mid gU_+=U_+\mbox{ and }gU_-=U_-\}.$$

\subsection{Invariants of the $R$-orbit of $V\in M=M_{(n)}$} \label{sec5.2}

In order to describe $G$-orbits on the triple flag variety $\mct_{(\alpha),(\beta),(n)}$, we have only to describe $R$-orbits of maximally isotropic subspaces for each $(U_+,U_-)$ in the previous subsection.

Let $V$ be a maximal isotropic subspace in $\bbf^{2n}$. Then the following $b_1,\ldots,b_{11}$ are clearly invariants of the $R$-orbit of $V$.
\begin{align*}
b_1 & =\dim(W_0\cap V),\quad b_2=a_0-b_1,\\
b_3 & =\dim(W_+\cap V)-b_1,\quad b_4=\dim(W_-\cap V)-b_1, \\
b_5 & =\dim(U_+\cap V)-b_1-b_3,\quad b_6=\dim(U_-\cap V)-b_1-b_4, \\
b_7 & =\dim(W\cap V)-b_1-b_3-b_4, \\
b_8 & =\dim((W_++U_-)\cap V)-\dim(W\cap V)-b_6, \\
b_9 & =\dim((U_++W_-)\cap V)-\dim(W\cap V)-b_5, \\
b_{10} & =a_+-b_3-b_7-b_8,\quad b_{11}=a_--b_4-b_7-b_9.
\end{align*}

We will deduce the remaining essential invariants from
$$X=X(V)=(U_++U_-)\cap V.$$
(Note that we don't need $X'=X^\perp$ in the even-degree case because $X^\perp=X$.) Write $W=W_++W_-=W_{(0)}\oplus W_{(+)}\oplus W_{(-)}$. Consider the projection $\pi: W^\perp\to \pi(W^\perp)=W^\perp/W$. Then $\pi(W^\perp)$ is decomposed as
$$\pi(W^\perp)=\pi(U_+)\oplus\pi(U_-)\oplus\pi(Z)$$
where $Z=(U_++U_-)^\perp\subset W^\perp$. Put $a_2=n-d-a_1$. Then
$$\dim \pi(Z)=2a_2\mand \dim Z=2a_2+d.$$

Let $\pi_+,\pi_-$ and $\pi_Z$ denote the projections of $W^\perp$ onto
$\pi(U_+),\ \pi(U_-)$ and $\pi(Z)$,
respectively.

\begin{lemma} $($Lemma 3.8 in \cite{M3}$)$ $a_1-\dim\pi(X)=a_2-\dim\pi(Z\cap V)$.
\label{lem5.2'}
\end{lemma}

The bilinear form $(\ ,\ )$ naturally induces a nondegenerate bilinear form on $W^\perp/W$. It is nondegenerate on the pair $(\pi(U_+),\pi(U_-))\cong (U_+/W_+,U_-/W_-)\cong (U_{(+)},U_{(-)})$. Put
$$X_0=((U_++W_-)\cap V)+((W_++U_-)\cap V).$$
For $v\in X$, write $v=v_++v_-$ with $v_+\in U_++W_-$ and $v_-\in W_++U_-$. Then $\pi(v_+)$ is uniquely defined from $v$. Define a subspace
$$X_1=X_1(V)=\{v\in X\mid (v_+,X)=\{0\}\}$$
of $X$.

\begin{lemma}  $($Lemma 3.9 in \cite{M3}$)$ $X_0\subset X_1$.
\end{lemma}

Put
$b_{12}=\dim X_1-\dim X_0$ and $b_{15}=\dim X-\dim X_1$.

\begin{lemma}  $($Lemma 3.10 in \cite{M3}$)$ The projections $\pi_+$ and $\pi_-$ induce bijections
$$X/X_0\stackrel{\sim}{\to} \pi_+(X)/\pi_+(X_0)\mand X/X_0\stackrel{\sim}{\to} \pi_-(X)/\pi_-(X_0),$$
respectively.
\label{lem5.4}
\end{lemma}

\begin{corollary} $($Corollary 3.11 in \cite{M3}$)$ $b_{12}\le a_1-\dim \pi(X)$.
\end{corollary}

Put $b_{13}=a_1-\dim \pi(X)-b_{12}$ and $b_{14}=a_2-b_{12}-b_{13}$. Then $b_{14}=\dim \pi(Z\cap V)$ by Lemma \ref{lem5.2'}. By the definition of $X_1$, the bilinear form $(\ ,\ )$ is nondegenerate on the pair $(\pi_+(X)/\pi_+(X_1),\pi_-(X)/\pi_-(X_1))$. By Lemma \ref{lem5.4}, there exists a bijection
$$\xi: \pi_+(X)/\pi_+(X_1)\stackrel{\sim}{\to} \pi_-(X)/\pi_-(X_1)$$
induced by $\pi_-|_{X}\circ \pi_+|_{X}^{-1}$. Define a bilinear form $\langle\ ,\ \rangle$ on $\pi_+(X)/\pi_+(X_1)$ by
$$\langle u,v\rangle =(u,\xi(v))\quad \mbox{for }u,v\in \pi_+(X).$$
(This is well-defined since $(u,\xi(v))=0$ if $u$ or $v$ is contained in $\pi_+(X_1)$.)

\begin{lemma} $($Lemma 3.12 in \cite{M3}$)$ For $u,\ v\in X$, we have $\langle \pi_+(u),\pi_+(v)\rangle =-\langle  \pi_+(v), \pi_+(u)\rangle$.
\end{lemma}

\begin{corollary} $($Corollary 3.13 in \cite{M3}$)$ $b_{15}$ is even.
\end{corollary}

Summarizing the arguments in this subsection, we have:

\begin{proposition}  $($Proposition 3.14 in \cite{M3}$)$ The invariants $b_1,\ldots,b_{15}$ for $V$ satisfy following equalities. $($$b_{15}$ is even.$)$
\begin{align}
a_0 & =b_1+b_2, \label{eq5.1'} \\
a_+ & =b_3+b_7+b_8+b_{10}, \label{eq5.2} \\
a_- & =b_4+b_7+b_9+b_{11}, \label{eq5.3} \\
a_1 & =b_5+b_6+b_8+b_9+2b_{12}+b_{13}+b_{15}, \label{eq5.4} \\
a_2 & =b_{12}+b_{13}+b_{14} \label{eq5.5}
\end{align}
\label{prop5.8}
\end{proposition}

\subsection{Representative of the $R$-orbit of $V$}

Conversely suppose that nonnegative numbers $b_1,\ldots,b_{15}$ satisfy the equalities in Proposition \ref{prop5.8}. Then we define a maximally isotropic subspace
$$V(b_1,\ldots,b_{15})=\bigoplus_{j=1}^{15} V_{(j)}$$
as follows.

Define subsets $I_{(j)}$ of $I$ for $j=1,\ldots,15$ by 
$I_{(j)}=\{i(j,1),\ldots,i(j,b_j)\}$
where
\begin{align*}
i(1,k) & =k, \quad
i(2,k)=b_1+k, \quad
i(3,k)=a_0+k, \quad
i(4,k) =a_0+a_++k, \\
i(5,k) & =d+k, \quad
i(6,k) =d'+k, \quad
i(7,k) =a_0+b_3+k, \\
i(8,k) & =a_0+b_3+b_7+k, \quad
i(9,k) =a_0+a_++b_4+b_7+k, \\
i(10,k) & =a_0+a_+-b_{10}+k, \quad
i(11,k) =d-b_{11}+k \\
i(12,k) & =d+b_5+b_9+k, \quad 
i(13,k) =d+b_5+b_9+b_{12}+b_{15}+k, \\
i(14,k) & =d+a_1+k 
\mand i(15,k) =d+b_5+b_9+b_{12}+k.
\end{align*}
For $j\in J_1=\{1,2,3,4,5,6,10,11,14\}$, put $U_{(j)}=\bigoplus_{i\in \widetilde{I_{(j)}}} \bbf e_i$ where $\widetilde{I_{(j)}}=I_{(j)}\sqcup \overline{I_{(j)}}$ and define maximally isotropic subspaces $V_{(j)}$ of $U_{(j)}$ by
$$V_{(j)}=\begin{cases} \bigoplus_{i\in I_{(j)}} \bbf e_i & \text{if $j=1,3,4,5,6,14$}, \\
\bigoplus_{i\in I_{(j)}} \bbf e_{\overline{i}} & \text{if $j=2,10,11$}.
\end{cases}$$

Define maps $\eta_j: I_{(j)}\to I$ for $j\in J_2=\{7,8,9,13\}$ by
\begin{align*}
\eta_7(i(7,k)) & =a_0+a_++b_4+k, \quad
\eta_8(i(8,k)) =d'+b_6+k, \\
\eta_9(i(9,k)) & =d+b_5+k \quad
\mand \eta_{13}(i(13,k)) =d+a_1+b_{14}+b_{12}+k
\end{align*}
for $k=1,\ldots,b_j$. For $j\in J_2$, put $U_{(j)}=\bigoplus_{i\in \widetilde{I_{(j)}}} \bbf e_i$ where 
$$\widetilde{I_{(j)}}=I_{(j)}\sqcup \eta_j(I_{(j)})\sqcup \overline{I_{(j)}}\sqcup \overline{\eta_j(I_{(j)})}.$$
Define maximally isotropic subspaces $V_{(j)}=V_{(j)}^1\oplus V_{(j)}^2$ of $U_{(j)}$ for $j\in J_2$ by
$$V_{(j)}^1=\bigoplus_{i\in I_{(j)}} \bbf (e_i+e_{\eta_j(i)}) \mand  V_{(j)}^2=\bigoplus_{i\in I_{(j)}} \bbf (e_{\overline{i}}-e_{\overline{\eta_j(i)}}).$$

Define maps $\kappa,\lambda:I_{(12)}\to I$ by
$$\kappa(i(12,k))=d'+b_6+b_8+k\mand \lambda(i(12,k))=d+a_1+b_{14}+k$$
for $k=1,\ldots,b_{12}$. Put $U_{(12)}=\bigoplus_{i\in \widetilde{I_{(12)}}} \bbf e_i$ where 
$$\widetilde{I_{(12)}}=I_{(12)}\sqcup \kappa(I_{(12)})\sqcup \lambda(I_{(12)})\sqcup \overline{I_{(12)}}\sqcup \overline{\kappa(I_{(12)})}\sqcup \overline{\lambda(I_{(12)})}.$$
Define a maximally isotropic subspace $V_{(12)}=V_{(12)}^1\oplus V_{(12)}^2\oplus V_{(12)}^3$ of $U_{(12)}$ by
\begin{align*}
V_{(12)}^1 & =\bigoplus_{i\in I_{(12)}} \bbf (e_i+e_{\kappa(i)}),\quad V_{(12)}^2=\bigoplus_{i\in I_{(12)}} \bbf (e_i+e_{\lambda(i)}) \\
\mand V_{(12)}^3 & =\bigoplus_{i\in I_{(12)}} \bbf (e_{\overline{i}}-e_{\overline{\kappa(i)}}-e_{\overline{\lambda(i)}}).
\end{align*}

Put $U_{(15)}=(\bigoplus_{i\in I_{15}\sqcup \overline{I_{(15)}}} \bbf e_i)$ 
and define a map
$$\eta_{15}(i(15,k))=d'+b_6+b_8+b_{12}+b_{13}+k$$
for $k=1,\ldots,b_{15}$. Put
$$I_{(15)}^+ =\{i(15,1),\ldots,i(15,\frac{b_{15}}{2})\}.$$
Then we define a maximally isotropic subspace $V_{(15)}$ of $U_{(15)}$ by
$$V_{(15)}=\bigoplus_{i\in I_{(15)}^+} (\bbf (e_i+e_{\eta_{15}(i)})\oplus \bbf (e_{\overline{i}}-e_{\overline{\eta_{15}(i)}})).$$
(We see that the definition of $V_{(15)}$ becomes much simpler in the even-degree case than $V_{(15)}^\varepsilon$ in \cite{M3}.) We can prove the following theorem in the same way as in \cite{M3} Section 3.6.

\begin{theorem} $($Theorem 3.15 in \cite{M3}$)$ Let $V$ be a maximally isotropic subspace of $\bbf^{2n}$. Define the numbers $b_1,\ldots,b_{15}$ as in Section \ref{sec5.2}. Then the $R$-orbit of $V$ contains the representative
$$V(b_1,\ldots,b_{15}).$$
\label{th5.9}
\end{theorem}

\subsection{Construction of elements in $R_V$} \label{sec5.4}

Assume $V=V(b_1,\ldots,b_{15})$. As in \cite{M3}, we can construct elements in $R_V=\{g\in R \mid gV=V\}$ as follows.

\begin{lemma} $($Lemma 3.18 in \cite{M3}$)$ For $j=1,\ldots,14$, let $A=\{a_{i,k}\}|_{i,k\in I_{(j)}}$ be an invertible matrix with the inverse matrix $A^{-1}=\{b_{i,k}\}$. Then we can construct an element $h=h_{(j)}(A)$ of $R_V$ such that $:$

{\rm (i)} If $j\in J_1=\{1,2,3,4,5,6,10,11,14\}$, then
$$he_k=\sum_{i\in I_{(j)}} a_{i,k}e_i,\quad he_{\overline{k}}=\sum_{i\in I_{(j)}} b_{k,i}e_{\overline{i}}$$
for $k\in I_{(j)}$ and $he_\ell=e_\ell$ for $\ell\in I-\widetilde{I_{(j)}}$.

{\rm (ii)} If $j\in J_2=\{7,8,9,13\}$, then
$$he_k=\sum_{i\in I_{(j)}} a_{i,k}e_i,\ he_{\eta_j(k)}=\sum_{i\in I_{(j)}} a_{i,k}e_{\eta_j(i)},\ he_{\overline{k}}=\sum_{i\in I_{(j)}} b_{k,i}e_{\overline{i}},\ he_{\overline{\eta_j(k)}}=\sum_{i\in I_{(j)}} b_{k,i}e_{\overline{\eta_j(i)}}$$
for $k\in I_{(j)}$ and $he_\ell=e_\ell$ for $\ell\in I-\widetilde{I_{(j)}}$.

{\rm (iii)} If $j=12$, then
\begin{align*}
he_k & =\sum_{i\in I_{(12)}} a_{i,k}e_i,\quad he_{\kappa(k)}=\sum_{i\in I_{(12)}} a_{i,k}e_{\kappa(i)},\quad he_{\lambda(k)}=\sum_{i\in I_{(12)}} a_{i,k}e_{\lambda(i)}, \\
he_{\overline{k}} & =\sum_{i\in I_{(12)}} b_{k,i}e_{\overline{i}},\quad he_{\overline{\kappa(k)}}=\sum_{i\in I_{(12)}} b_{k,i}e_{\overline{\kappa(i)}},\quad he_{\overline{\lambda(k)}}=\sum_{i\in I_{(12)}} b_{k,i}e_{\overline{\lambda(i)}}
\end{align*}
for $k\in I_{(12)}$ and $he_\ell=e_\ell$ for $\ell\in I-\widetilde{I_{(12)}}$.
\label{lem5.10}
\end{lemma}

Next we prepare elements in $R_V$ constructed in \cite{M2} for $U_{(15)}$. Define an alternating form $\langle\ ,\ \rangle$ on $\bbf^{2m}$ with the canonical basis $f_1,\ldots,f_{2m}$ by
$$\langle f_k,f_\ell\rangle=\begin{cases} -\delta_{k,2m+1-\ell} & \text{if $k \le m$,} \\
\delta_{k,2m+1-\ell} & \text{if $k >m$.}
\end{cases}$$
Then we can define a subgroup ${\rm Sp}'_{2m}(\bbf)\ (\cong{\rm Sp}_{2m}(\bbf))$ of ${\rm GL}_{2m}(\bbf)$ by
$${\rm Sp}'_{2m}(\bbf)=\{g\in{\rm GL}_{2m}(\bbf)\mid \langle gu,gv \rangle=\langle u,v \rangle\mbox{ for }u,v\in\bbf^{2m}\}.$$

Write
$U_{(15)}=U^+_{(15)}\oplus U^-_{(15)}$
with
$$U^+=U_{(15)}\cap U_+=\bigoplus_{i\in I_{(15)}} \bbf e_i \mand U^-=U_{(15)}\cap U_-=\bigoplus_{i\in I_{(15)}} \bbf e_{\overline{i}}.$$
Then we can define a bijection
$\xi :U^+_{(15)}\to U^-_{(15)}$
by the condition
$$\xi (u)=v\Longleftrightarrow u+v\in V_{(15)}.$$
Define a bilinear form $\langle\ ,\ \rangle$ on $U^+_{(15)}$ by
$$\langle u,v\rangle=(u,\xi(v)).$$
Then the form $\langle\ ,\ \rangle$ is alternating and nondegenerate. By the definition of $V_{(15)}$, we have
$$\langle e_{i(15,k)},e_{i(15,\ell)}\rangle=\begin{cases} -\delta_{k,b_{15}+1-\ell} & \text{if $k \le b_{15}/2$,} \\
\delta_{k,b_{15}+1-\ell} & \text{if $k >b_{15}/2$.}
\end{cases}$$
Put $m=b_{15}/2$ and define an injective linear map $\varphi: \bbf^{2m}\to U_{(15)}^+$ by
$\varphi (f_k)=e_{i(15,k)}$.
Clearly we have the following lemma.

\begin{lemma} For $A\in {\rm Sp}'_{b_{15}}(\bbf)$, we can define an element $h=h_{(15)}(A)$ of $R_V$ such that
$$h\varphi(f_k)=\varphi(Af_k),\quad h\xi(\varphi(f_k))=\xi(\varphi(Af_k))$$
and that $hu=u$ for $u\in U_{(15)}^\perp=\bigoplus_{j=1}^{14} U_{(j)}$. 
\label{lem5.11'}
\end{lemma}

The index set $I_+=\{i\in I\mid e_i\in U_+\}$ is decomposed as
$$I_+=\bigsqcup_{j\in \mathcal{I}_+} I_{(j)}$$
where
$\mathcal{I}_+=\{1,2,3,7,8,10,5,9,12,13,15,\overline{12},\overline{8},\overline{6}\}$
and $I_{(\overline{6})}=\overline{I_{(6)}},\ I_{(\overline{8})}=\overline{\eta_8(I_{(8)})}$ and $I_{(\overline{12})}=\overline{\kappa(I_{(12)})}$. Correspondingly $U_+$ is decomposed as
$$U_+=\bigoplus_{j\in\mathcal{I}_+} U_{(j)}^+$$
where 
$U_{(j)}^+ =\bigoplus_{i\in I_{(j)}} \bbf e_i=\bbf e_{i(j,1)}\oplus\cdots\oplus \bbf e_{i(j,\dim U_{(j)}^+)},\ i(\overline{6},k)=\overline{i(6,b_6+1-k)},\ i(\overline{8},k)=\overline{\eta_8(i(8,b_8+1-k)}$ and $i(\overline{12},k)=\overline{\kappa (i(12,b_{12}+1-k)}$.

\setlength{\unitlength}{1mm}
\begin{figure}
\centering
\begin{picture}(140,70)(10,25)
\put(30,90){\makebox(0,0){$1$}}
\put(50,90){\makebox(0,0){$2$}}
\put(30,70){\makebox(0,0){$3$}}
\put(50,70){\makebox(0,0){$7$}}
\put(70,70){\makebox(0,0){$8$}}
\put(90,70){\makebox(0,0){$10$}}
\put(30,50){\makebox(0,0){$5$}}
\put(50,50){\makebox(0,0){$9$}}
\put(70,50){\makebox(0,0){$12$}}
\put(90,50){\makebox(0,0){$13$}}
\put(110,30){\makebox(0,0){$\overline{8}$}}
\put(70,30){\makebox(0,0){$15$}}
\put(90,30){\makebox(0,0){$\overline{12}$}}
\put(130,30){\makebox(0,0){$\overline{6}$}}
\put(45,90){\vector(-1,0){10}}
\put(45,70){\vector(-1,0){10}}
\put(65,70){\vector(-1,0){10}}
\put(85,70){\vector(-1,0){10}}
\put(45,50){\vector(-1,0){10}}
\put(65,50){\vector(-1,0){10}}
\put(85,50){\vector(-1,0){10}}
\put(30,75){\vector(0,1){10}}
\put(50,75){\vector(0,1){10}}
\put(30,55){\vector(0,1){10}}
\put(50,55){\vector(0,1){10}}
\put(70,55){\vector(0,1){10}}
\put(90,55){\vector(0,1){10}}
\put(90,35){\vector(0,1){10}}
\put(70,35){\vector(0,1){10}}
\put(85,30){\vector(-1,0){10}}
\put(105,30){\vector(-1,0){10}}
\put(125,30){\vector(-1,0){10}}
\end{picture}
\caption{Diagram of $\mathcal{I}_+$}
\label{fig5.1}
\end{figure}
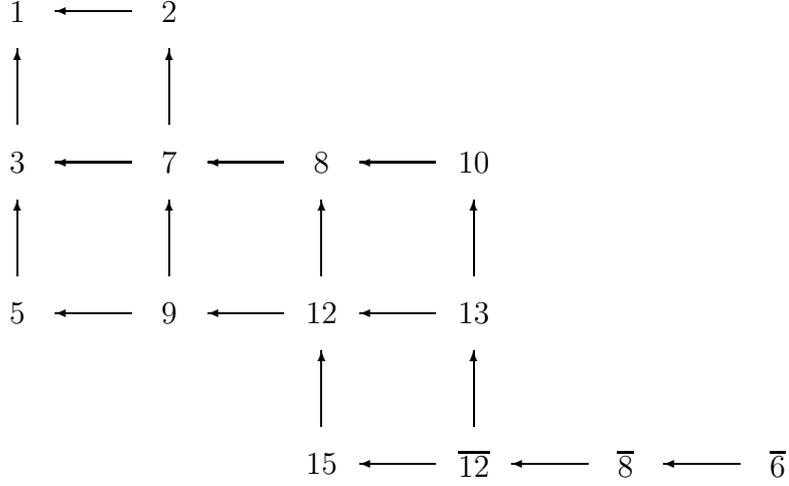

Consider the diagram of $\mathcal{I}_+$ in Figure \ref{fig5.1}.
We define a partial order $j\prec j'$ for $j,j'\in\mathcal{I}_+$ if there exists a sequence $j_0,j_1,\ldots,j_k$ in $\mathcal{I}_+$ such that
$$j=j_0\longleftarrow j_1\longleftarrow \cdots\longleftarrow j_k=j'.$$
For example, $1\prec j$ for all $j\in\mathcal{I}_+-\{1\}$ and $j\prec \overline{6}$ for all $j\in\mathcal{I}_+-\{\overline{6}\}$.

\begin{lemma} $($Lemma 3.21 in \cite{M3}$)$ {\rm (i)} Suppose that $j\prec j'$ for $j,j'\in\mathcal{I}_+$ and that
$$(j,j') \notin\{(8,12),\ (8,15),\ (12,15),\  (15,\overline{12}),\ (15,\overline{8}),\ (\overline{12},\overline{8})\}.$$
Then for $i\in I_{(j)},\ k\in I_{(j')}$ and $\mu\in\bbf$, there exists an element $g=g_{i,k}(\mu)\in R_V$ such that
$$ge_k=e_k+\mu e_i\quad\mbox{and that}\quad ge_\ell=e_\ell\mbox{ for }\ell\in I_+-\{k\}.$$

{\rm (ii)} For $i\in I_{(8)},\ k\in I_{(12)}$ and $\mu\in\bbf$, there exists an element $g=g_{i,k}(\mu)\in R_V$ such that
$$ge_k=e_k+\mu e_i,\quad ge_{\overline{\eta_8(i)}}=e_{\overline{\eta_8(i)}}-\mu e_{\overline{\kappa(k)}}$$
and that $ge_\ell=e_\ell$ for $\ell\in I_+-\{k,\overline{\eta_8(i)}\}$.

{\rm (iii)} For $i\in I_{(12)},\ k\in I_{(15)}$ and $\mu\in\bbf$, there exists an element $g=g_{i,k}(\mu)\in R_V$ such that
$$ge_k=e_k+\mu e_i,\quad ge_{\overline{\kappa(i)}}= e_{\overline{\kappa(i)}}-\mu e_{\overline{\eta_{15}(k)}}$$
and that $ge_\ell=e_\ell$ for $\ell\in I_+-\{k,\overline{\kappa(i)}\}$.

{\rm (iv)} For $i\in I_{(8)},\ k\in I_{(15)}$ and $\mu\in\bbf$, there exists an element $g=g_{i,k}(\mu)\in R_V$ such that
$$ge_k=e_k+\mu e_i,\quad ge_{\overline{\eta_8(i)}}= e_{\overline{\eta_8(i)}}-\mu e_{\overline{\eta_{15}(k)}}$$
and that $ge_\ell=e_\ell$ for $\ell\in I_+-\{k,\overline{\eta_8(i)}\}$.
\label{lem5.11}
\end{lemma}

\begin{lemma} For $k\in I_+-I_{(1)}$ and $u\in U_{(1)}^+$, there exists an element $g\in R_V$ such that
$$g(e_k+u)=e_k$$
and that $ge_\ell=e_\ell$ for $\ell\in I_+-\{k\}$.
\label{lem5.13'}
\end{lemma}

\begin{proof} Write $u=\lambda_1e_1+\cdots+\lambda_{b_1}e_{b_1}$. Then the product $g=\prod_{i=1}^{b_1} g_{i,k}(-\lambda_i)$ satisfies the desired condition.
\end{proof}

By using Lemma \ref{lem5.11}, we gave the following two lemmas in \cite{M3}.

\begin{lemma} $($Lemma 3.22 in \cite{M3}$)$ Let $k$ be an index in $I_{(\overline{6})}$. Then for any element $u$ in $\bigoplus_{i\in I_+-I_{(\overline{6})}} \bbf e_i$,
there exists a $g\in R_V$ such that
$g(e_k+u)=e_k$
and that $ge_\ell=e_\ell$ for $\ell\in I_+-\{k\}$. %$($Remark$:$ In \cite{M3}, we should add ``and that $ge_\ell=e_\ell$ for $\ell\in I_+-\{k\}$''.$)$
\label{lem5.12}
\end{lemma}

\begin{lemma} $($Lemma 3.23 in \cite{M3}$)$ Let $k$ be an index in $I_{(\overline{8})}$. Then for any element $u$ in $\bigoplus_{i\in I_+-I_{(\overline{8})}-I_{(\overline{6})}} \bbf e_i$,
there exists a $g\in R_V$ such that
$g(e_k+u)=e_k$
and that $ge_\ell=e_\ell$ for $\ell\in I_+-I_{(12)}-I_{(15)}-\{k\}$. %$($Remark$:$ It is written insufficiently in \cite{M3} ``and that $g$ acts trivially on $U_{(6)}^+$''.$)$
\label{lem5.13}
\end{lemma}

Similarly we also have:

\begin{lemma} Let $k$ be an index in $I_{(\overline{12})}$. Then for any element $u$ in

\noindent $\bigoplus_{i\in I_+-I_{(\overline{12})}-I_{(\overline{8})}-I_{(\overline{6})}} \bbf e_i$,
there exists a $g\in R_V$ such that
$g(e_k+u)=e_k$
and that $ge_\ell=e_\ell$ for $\ell\in I_+-I_{(15)}-\{k\}$.
\label{lem5.17}
\end{lemma}

We can prove the following by the same arguments as in Section 4 of \cite{M3}.

\begin{proposition} $|R_V\backslash M_{\gamma_1}(U_+)|<\infty$ where $M_{\gamma_1}(U_+)$ is the Grassmann variety consisting of $\gamma_1$-dimensional subspaces in $U_+$.
\label{prop5.18}
\end{proposition}

\begin{corollary} The triple flag variety $\mct_{(\gamma_1,\gamma_2),(\beta),(n)}$ is of finite type.
\label{cor5.19}
\end{corollary}

\section{Proof of Proposition \ref{prop4.1}}

\subsection{Reduction of flags in $U_+$} \label{sec6.1}

For any subset $J$ of $I_+$, let $p_J$ denote the canonical projection of $U_+$ onto $\bigoplus_{i\in J} \bbf e_i$ defined by
$$p_J(\sum_{i\in I_+} x_ie_i)=\sum_{i\in J} x_ie_i.$$
For a subset $\{j_1,\ldots,j_k\}$ of $\mathcal{I}_+$, write
$$U_{(j_1,\ldots,j_k)}^+=U_{(j_1)}^+\oplus\cdots\oplus U_{(j_k)}^+\mand I_{(j_1,\ldots,j_k)}=I_{(j_1)}\sqcup\cdots\sqcup I_{(j_k)}.$$
For $j\in\mathcal{I}_+$ and $\ell$ with $0\le \ell\le\dim U_{(j)}^+$, define a subset $I_{(j)}(\ell)=\{i(j,1),\ldots,i(j,\ell)\}$ of $I_{(j)}$ and a subspace
$$U_{(j)}^{+,\ell}=\bigoplus_{i\in I_{(j)}(\ell)} \bbf e_i=\bbf e_{i(j,1)}\oplus\cdots\oplus \bbf e_{i(j,\ell)}$$
of $U_{(j)}^+$. For a subset $\{j_1,\ldots,j_k\}$ of $\mathcal{I}_+$ and $\ell_1,\ldots,\ell_k$ with $0\le \ell_t\le \dim U^+_{(j_t)}\ (1\le t\le k)$, write
$$U_{(j_1,\ldots,j_k)}^{+,\ell_1,\ldots,\ell_k}=U^{+,\ell_1}_{(j_1)}\oplus\cdots \oplus U^{+,\ell_k}_{(j_k)}.$$

Let $U_+^1\subset\cdots\subset U_+^m$ be a flag in $U_+$. Put $s_{0,j}=\dim (U_+^j\cap U_{(1)}^+)$ for $j=1,\ldots,m$. Then we can take an $h_0=h_{(1)}(A)\in R_V$ with some $A\in{\rm GL}_{b_1}(\bbf)$ such that
$$h_0U_+^j\cap U_{(1)}^+=U_{(1)}^{+,s_{0,j}}$$
for $j=1,\ldots,m$ by Lemma \ref{lem5.10}. We can write
$$h_0U_+^j=U_{(1)}^{+,s_{0,j}}\oplus \bbf v_1\oplus\cdots\oplus \bbf v_{d(j)}$$
with some linearly independent vectors $v_1,\ldots, v_{d(m)}$ in $U_+^m$. Let $W$ be a subspace of $U_+$ such that
$$v_1,\ldots, v_{d(m)}\in W$$
and that $U_+=U_{(1)}^+\oplus W$. Write $W=\sum_{k\in I_+-I_{(1)}} \bbf(e_k+u_k)$ with $u_k\in U_{(1)}^+$. Then we have $g_{(k)}\in R_V$ such that
$$g_{(k)}(e_k+u_k)=e_k$$
and that $ge_\ell=e_\ell$ for $\ell\in I_+-\{k\}$ by Lemma \ref{lem5.13'}. Then the product $g_0=\prod_{k\in I_+-I_{(1)}} g_{(k)}$ satisfies
$$g_0W=U'$$
where $U'=\bigoplus_{i\in I_+-I_{(1)}} \bbf e_i$. Hence we have
$$g_0h_0U_+^j=S_{0,j}\oplus U_{(1)}^{+,s_{0,j}}$$
with $S_{0,j}=g_0 (\bbf v_1\oplus\cdots\oplus \bbf v_{d(j)})\subset U'$.

Put $s_{1,j}=\dim p_{I_{(\overline{6})}} S_{0,j}$ for $j=1,\ldots,m$. Then there exists an $h_1=h_{(6)}(A)\in R_V$ with some $A\in{\rm GL}_{b_6}(\bbf)$ such that
$$p_{I_{(\overline{6})}} h_1S_{0,j} =U_{(\overline{6})}^{+,s_{1,j}}=\bigoplus_{i\in I_{(\overline{6})}(s_{1,j})} \bbf e_i $$
by Lemma \ref{lem5.10}. Put $U''=\bigoplus_{i\in I_+-I_{(1)}-I_{(\overline{6})}} \bbf e_i$. Then we can write
$$h_1U_+^1 =(h_1S_{0,j}\cap U'')\oplus (\bigoplus_{i\in I_{(\overline{6})}(s_{1,j})} \bbf v_i)$$
with vectors $v_i\in e_i+U''$ for $i\in I_{(\overline{6})}(s_{1,m})$.

By Lemma \ref{lem5.12}, we can take a $g_1\in R_V$ such that
$$g_1v_i=e_i\quad\mbox{for all }i\in  I_{(\overline{6})}(s_{1,m}).$$
and that $g_1$ acts trivially on $U^+_{(1)}\oplus U''$. Thus we have
$$g_1h_1S_{0,j}=S_{1,j}\oplus U_{(\overline{6})}^{+,s_{1,j}}\mand g_1h_1g_0h_0U_+^j=S_{1,j}\oplus U_{(\overline{6},1)}^{+,s_{1,j},s_{0,j}}$$
where $S_{1,j}=S_{0,j}\cap U''$.

Put $s_{2,j}=\dim p_{I_{(\overline{8})}} S_{1,j}$. Then there exists an $h_2=h_{(8)}(A)\in R_V$ with some $A\in{\rm GL}_{b_8}(\bbf)$ such that
\begin{align*}
p_{I_{(\overline{8})}} h_2S_{1,j} & =U_{(\overline{8})}^{+,s_{2,j}} =\bigoplus_{i\in I_{((\overline{8})}(s_{2,j})} \bbf e_i
\end{align*}
by Lemma \ref{lem5.10}.

Put $U'''=\bigoplus_{i\in I-I_{(1)}-I_{(\overline{6})}-I_{(\overline{8})}} \bbf e_i$. Then we can write
$$h_2S_{1,1} =(h_2S_{1,1}\cap U''') \oplus (\bigoplus_{i\in I_{(\overline{8})}(s_{2,j})} \bbf v_i)$$
with vectors $v_i\in e_i+U'''$ for $i\in I_{(\overline{8})}(s_{2,j})$. By Lemma \ref{lem5.13}, we can take an element $g_2\in R_V$ such that
$$g_2v_i=e_i \quad \mbox{for all } i\in I_{(\overline{8})}(s_{2,j})$$
and that $g_2e_\ell=e_\ell$ for $\ell\in I_+-I_{(15)}-I_{(12)}-I_{(\overline{8})}(s_{2,j})$. Thus we have
$$g_2h_2S_{1,j}=S_{2,j}\oplus U_{(\overline{8})}^{+,s_{2,j}}\mand g_2h_2g_1h_1g_0h_0U_+^j=S_{2,j}\oplus U_{(\overline{8},\overline{6},1)}^{+,s_{2,j},s_{1,j},s_{0,j}}$$
where $S_{2,j}=g_2h_2S_{1,j}\cap U'''$.

Put $s_{3,j}=\dim p_{I_{(\overline{12})}} S_{2,j}$. Then there exists an $h_3=h_{(12)}(A)$ with some $A\in{\rm GL}_{b_{12}}(\bbf)$ such that
$$p_{I_{(\overline{12})}} h_3S_{2,j} =U_{(\overline{12})}^{+,s_{3,j}}=\bigoplus_{i\in I_{(\overline{12})}(s_{3,j})} \bbf e_i$$
by Lemma \ref{lem5.10}.

Put $\widetilde{U}=\bigoplus_{i\in I-I_{(1)}-I_{(\overline{6})}-I_{(\overline{8})}-I_{(\overline{12})}} \bbf e_i$. Then we can write
$$h_3S_{2,j}=(h_3S_{2,j}\cap \widetilde{U})\oplus (\bigoplus_{i\in I_{(\overline{12},s_{3,j})}} \bbf v_i)$$
with vectors $v_i\in e_i+\widetilde{U}$ for $i\in I_{(\overline{12})}(s_{3,j})$. By Lemma \ref{lem5.17}, we can take an element $g_3\in R_V$ such that
$$g_3v_i=e_i\quad\mbox{for all }i\in I_{(\overline{12})}(s_{3,j})$$
and that $g_3e_\ell=e_\ell$ for $\ell\in I-I_{(15)}-I_{(\overline{12})}(s_{3,j})$. Thus we have $g_3h_3S_{2,j} =\widetilde{S}_j\oplus U_{(\overline{12})}^{+,s_{3,j}}$ and
\begin{equation}
g_3h_3g_2h_2g_1h_1g_0h_0U_+^j  =\widetilde{S}_j\oplus U^j_\# \label{eq6.0'}
\end{equation}
where $\widetilde{S}_j=g_3h_3S_{2,j} \cap \widetilde{U}$ and $U^j_\#=U^{+,s_{3,j},s_{2,j},s_{1,j},s_{0,j}}_{(\overline{12},\overline{8},\overline{6},1)}$.

Clearly we have the following lemma.

\begin{lemma} {\rm (i)} Let $h=h_{(8)}(A)$ be the element of $R_V$ defined in Lemma \ref{lem5.10} for $A\in {\rm GL}_{b_8}(\bbf)$. Then
$$hU_{(\overline{8})}^{+,s_{2,j}}=U_{(\overline{8})}^{+,s_{2,j}}
\Longleftrightarrow  h(U_{(8)}^{+,b_8-s_{2,j}})=U_{(8)}^{+,b_8-s_{2,j}}$$

{\rm (ii)} Let $h=h_{(12)}(A)$ be the element of $R_V$ defined in Lemma \ref{lem5.10} for $A\in {\rm GL}_{b_{12}}(\bbf)$. Then
$$hU_{(\overline{12})}^{+,s_{3,j}}=U_{(\overline{12})}^{+,s_{3,j}}
\Longleftrightarrow  h(U_{(12)}^{+,b_{12}-s_{3,j}})=U_{(12)}^{+,b_{12}-s_{3,j}}$$
\label{lem4.2}
\end{lemma}

By (\ref{eq6.0'}) we may assume
$$U_+^j=\widetilde{S}_j\oplus U^j_\#$$
for $j=1,\ldots,m$. Let $Q$ denote the restriction of
\begin{align*}
R_V(\widetilde{U}) & =\{g\in R_V\mid g\widetilde{U}=\widetilde{U},\ gU_{(\overline{12})}^{+,s_{3,j}}=U_{(\overline{12})}^{+,s_{3,j}},\ gU_{(\overline{8})}^{+,s_{2,j}}=U_{(\overline{8})}^{+,s_{2,j}}\mbox{ for }j=1,\ldots,m \\
& \qquad\qquad\qquad \mbox{ and $g$ acts trivially on }U^+_{(\overline{6},1)}\}
\end{align*}
to $\widetilde{U}$. Consider the decomposition $\widetilde{U}=\widetilde{U}_1\oplus \widetilde{U}_2$ of $\widetilde{U}$ where
$$\widetilde{U}_1=U_{(2,7,8,10,13)}^+ \mand \widetilde{U}_2=U_{(3,5,9,12,15)}^+.$$
Consider ${\rm GL}(\widetilde{U}_1)\times {\rm GL}(\widetilde{U}_2)$ naturally as a subgroup of ${\rm GL}(\widetilde{U})$. 

Let $P_1$ denote the parabolic subgroup of ${\rm GL}(\widetilde{U}_1)$ consisting of elements represented by matrices
$$\bp A & * & * & * & * \\
0 & B & * & * & * \\
0 & 0 & C & * & * \\
0 & 0 & 0 & D & * \\
0 & 0 & 0 & 0 & E \ep$$
$(A\in{\rm GL}_{b_2}(\bbf),\ B\in{\rm GL}_{b_7}(\bbf),\ C\in \mcb_{b_8},\ D\in{\rm GL}_{b_{10}}(\bbf),\ E\in{\rm GL}_{b_{13}}(\bbf))$ with respect to the basis
\begin{align*}
& e_{i(2,1)},\ldots,e_{i(2,b_2)},e_{i(7,1)},\ldots,e_{i(7,b_7)},e_{i(8,1)},\ldots,e_{i(8,b_8)}, \\
& e_{i(10,1)},\ldots,e_{i(10,b_{10})},e_{i(13,1)},\ldots,e_{i(13,b_{13})}
\end{align*}
of $\widetilde{U}_1$. On the other hand, let $K_2$ denote the subgroup of ${\rm GL}(\widetilde{U}_2)$ represented by matrices
$$\bp A & * & * & * & * \\
0 & B & * & * & * \\
0 & 0 & C & * & * \\
0 & 0 & 0 & D & * \\
0 & 0 & 0 & 0 & E \ep$$
$(A\in{\rm GL}_{b_3}(\bbf),\ B\in{\rm GL}_{b_5}(\bbf),\ C\in{\rm GL}_{b_9}(\bbf),\ D\in \mcb_{b_{12}},\ E\in{\rm Sp}'_{b_{15}}(\bbf))$ with respect to the basis
\begin{align*}
& e_{i(3,1)},\ldots,e_{i(3,b_3)},e_{i(5,1)},\ldots,e_{i(5,b_5)},e_{i(9,1)},\ldots,e_{i(9,b_9)}, \\
& e_{i(12,1)},\ldots,e_{i(12,b_{12})},e_{i(15,1)},\ldots,e_{i(15,b_{15})}
\end{align*}
of $\widetilde{U}_2$. By Lemma \ref{lem5.10}, Lemma \ref{lem5.11'}, Lemma \ref{lem5.11} (with Figure \ref{fig5.1}) and Lemma \ref{lem4.2}, we have:

\begin{lemma} {\rm (i)} Let $g$ be an element of $P_1\times K_2$. Then there exists a $\widetilde{g}\in R_V$ such that
$$\widetilde{g}\widetilde{U}=\widetilde{U},\quad \widetilde{g}|_{\widetilde{U}}=g,\quad \widetilde{g}\mbox{ acts trivially on }U^+_{(\overline{6},1)},\quad  \widetilde{g}e_{i(\overline{8},k)}\in U_{(\overline{8})}^{+,k}\mbox{ for }k=1,\ldots,b_8
$$
and that
$\widetilde{g}e_{i(\overline{12},k)}\in U_{(\overline{12})}^{+,k}\oplus U^+_{(15)}$
for $k=1,\ldots,b_{12}$.

{\rm (ii)} If $s_{3,m}=0$, then we have
$P_1\times K_2\subset Q$.
\label{lem6.2}
\end{lemma}

\subsection{Proof of (i) and (ii)} \label{sec6.2}

Since $\dim U_+=a_0+a_++a_1=n$, we have
$a_-=a_2=0$.
Hence
$$b_k=0\quad\mbox{for }k=4,7,9,11,12,13,14$$
by (\ref{eq5.3}) and (\ref{eq5.5}). So the Figure \ref{fig5.1} for Lemma \ref{lem5.11} is reduced to much simpler Figure \ref{fig6.1}. Furthermore when we use Lemma \ref{lem5.11} in $\widetilde{U}$, we have only to consider the diagram in Figure \ref{fig6.2}.

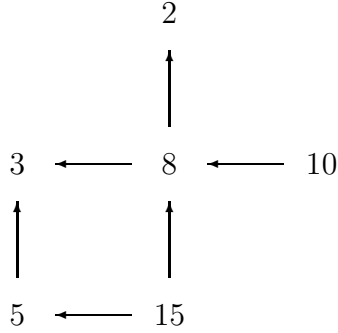
\begin{figure}
\centering
\begin{picture}(140,50)(-10,45)
\put(30,90){\makebox(0,0){$1$}}
\put(50,90){\makebox(0,0){$2$}}
\put(30,70){\makebox(0,0){$3$}}
\put(50,70){\makebox(0,0){$8$}}
\put(70,70){\makebox(0,0){$10$}}
\put(30,50){\makebox(0,0){$5$}}
\put(50,50){\makebox(0,0){$15$}}
\put(70,50){\makebox(0,0){$\overline{8}$}}
\put(90,50){\makebox(0,0){$\overline{6}$}}

\put(45,90){\vector(-1,0){10}}
\put(45,70){\vector(-1,0){10}}
\put(65,70){\vector(-1,0){10}}

\put(45,50){\vector(-1,0){10}}
\put(65,50){\vector(-1,0){10}}
\put(85,50){\vector(-1,0){10}}
\put(30,75){\vector(0,1){10}}
\put(50,75){\vector(0,1){10}}
\put(30,55){\vector(0,1){10}}
\put(50,55){\vector(0,1){10}}
\put(70,55){\vector(0,1){10}}

\end{picture}
\caption{Diagram of $\mathcal{I}_+$ when $\dim U_+=n$}
\label{fig6.1}
\end{figure}

\begin{figure}
\centering
\begin{picture}(140,50)(-10,45)
\put(50,90){\makebox(0,0){$2$}}
\put(30,70){\makebox(0,0){$3$}}
\put(50,70){\makebox(0,0){$8$}}
\put(70,70){\makebox(0,0){$10$}}
\put(30,50){\makebox(0,0){$5$}}
\put(50,50){\makebox(0,0){$15$}}

\put(45,70){\vector(-1,0){10}}
\put(65,70){\vector(-1,0){10}}

\put(45,50){\vector(-1,0){10}}

\put(50,75){\vector(0,1){10}}
\put(30,55){\vector(0,1){10}}
\put(50,55){\vector(0,1){10}}
\end{picture}
\caption{Diagram of $\mathcal{I}_+$ for $\widetilde{U}$ when $\dim U_+=n$}
\label{fig6.2}
\end{figure}

Let $U_+^1\subset U_+^2$ be a flag in $U_+$ such that $\dim U_+^1=1$ or $\dim U_+^2-\dim U_+^1=1$. As in Section \ref{sec6.1}, we may assume $U_+^1=\widetilde{S}_1\oplus U_\#^1$ and $U_+^2=\widetilde{S}_2\oplus U_\#^2$. Hence $\dim \widetilde{S}_1\le 1$ or $\dim \widetilde{S}_2-\dim \widetilde{S}_1\le 1$.

Since $b_{12}=0$, we have $P_1\times K_2\subset Q$ by Lemma \ref{lem6.2}. It follows from Proposition \ref{prop7.4} in the appendix that there are a finite number of $P_1\times K_2$-orbits of flags $\widetilde{S}_1\subset\widetilde{S}_2$ in $\widetilde{U}$. Thus we have proved (i) and (ii).

\subsection{Proof of (iii)} \label{sec6.4}

(A) Case of $\dim \widetilde{S}_2=2$. Since $s_{3,2}=0$, we have $P_1\times K_2\subset Q$ by Lemma \ref{lem6.2}. By Proposition \ref{prop7.4}, there are a finite number of $P_1\times K_2$-orbits of flags $\widetilde{S}_1\subset \widetilde{S}_2$ in $\widetilde{U}$. Hence there are a finite number of $Q$-orbits of flags $\widetilde{S}_1\subset \widetilde{S}_2$ in $\widetilde{U}$.

(B) Case of $\dim \widetilde{S}_2=1$. By Proposition \ref{prop5.18}, there are a finite number of $R_V$-orbits of $U_+^2$. So we may fix a $U_+^2=\widetilde{S}_2\oplus U^2_\#$. Since $U_+^1=\widetilde{S}_2$ or $U^2_\#$, the assertion is clear.

(C) The case of $\dim \widetilde{S}_2=0$ is clear because $U_+^j=U^j_\#$ for $j=1,2$.

\subsection{Proof of (iv)} \label{sec6.5}

Put
$$d_1=\dim U_0-\dim (U_0\cap \widetilde{S}_1)\mand d_2=\dim \widetilde{U}-\dim \widetilde{S}_1$$
where $U_0=U_{(2,7,8,10,13,3,5,9,12)}^+$. Then we have $0\le d_1\le d_2\le 2$.

(A) Case of $d_1=d_2=2$. By Proposition \ref{prop7.4}, there are a finite number of $P_1\times K_2$-orbits of flags $\widetilde{S}_1\subset \widetilde{S}_2$ in $\widetilde{U}$. For $g\in P_1\times K_2$, we can take a $\widetilde{g}\in R_V$ such that
\begin{align*}
\widetilde{g}|_{\widetilde{U}} & =g,\quad \widetilde{g}u =u\mbox{ for }u\in U_{(\overline{6},1)}^+,\quad \widetilde{g}U_{(\overline{8})}^+=U_{(\overline{8})}^+ \\
\mbox{and that}\quad \widetilde{g}U_{(\overline{12})}^+ & =\bbf(e_{i(\overline{12},1)}+v_1)\oplus\cdots\oplus \bbf(e_{i(\overline{12},b_{12})}+v_{b_{12}})
\end{align*}
with some $v_1,\ldots,v_{b_{12}}\in U_{(15)}^+$ by Lemma \ref{lem6.2}. Hence we may assume
$$U_+^j=\widetilde{S}_j\oplus \widetilde{g}U_{(\overline{12})}^+\oplus U_{(\overline{8},\overline{6},1)}^+$$
for $j=1,2$. We have only to show that there exists a $g'\in R_V$ such that
\begin{equation}
g'U_+^j=\widetilde{S}_j\oplus U^+_{(\overline{12},\overline{8},\overline{6},1)} \label{eq6.2'}
\end{equation}
for $j=1,2$.

Since $p_{I_{(15)}} \widetilde{S}_1=U_{(15)}^+$, we can take $u_1,\ldots,u_{b_{12}}\in U_0$ such that
$$v_k\in u_k+\widetilde{S}_1$$
for $k=1,\ldots,b_{12}$. So we have
$$\bbf(e_{i(\overline{12},1)}+u_1)\oplus\cdots\oplus \bbf(e_{i(\overline{12},b_{12})}+u_{b_{12}}) \subset U_+^j.$$
By Lemma \ref{lem5.11}, we can take a $g'\in R_V$ such that
$$g'(e_{i(\overline{12},k)}+u_k)=e_{i(\overline{12},k)}$$
for $k=1,\ldots,b_{12}$ and that $g'e_\ell=e_\ell$ for $\ell\in I_+-I_{(\overline{12})}$. Hence we have (\ref{eq6.2'}).

(B) Case of $d_1=1$ and $d_2=2$.

(B.1) Suppose that $\widetilde{S}_1\supset U'_0=U_{(2,7,8,3,5,9,12)}^+$. Then we have
$$\widetilde{S}_j=U'_0\oplus (\widetilde{S}_j\cap U_{(10,13,15)}^+)$$
for $j=1,2$. Let $P$ be the subgroup of ${\rm GL}(U_{(10,13)}^+)$ stabilizing the subspace $U_{(10)}^+$ and $H$ the subgroup of ${\rm GL}(U_{(15)}^+)$ consisting of elements $h_{(15)}(A)\in R_V$ with $A\in{\rm Sp}'_{b_{15}}(\bbf)$. Then there are a finite number of $P\times H$-orbits of flags
$$\widetilde{S}_1\cap U_{(10,13,15)}^+\subset \widetilde{S}_2\cap U_{(10,13,15)}^+$$
in $U_{(10,13,15)}^+$ by Proposition \ref{prop7.4}. Since $P\times H$ acts trivially on $U'_0\oplus U^+_{(\overline{12},\overline{8},\overline{6},1)}$, we have a finite number of $R_V$-orbits of the flags $U_+^1\subset U_+^2$ in $U_+$.

(B.2) Suppose that $\widetilde{S}_1\not\supset U'_0$. Then we have
$$U_0=(U_0\cap \widetilde{S}_1)\oplus \bbf e_k$$
with some $k\in I_{(2,7,8,3,5,9,12)}$. Since $\dim p_{I_{(15)}} \widetilde{S}_1=b_{15}-1$, we can take an $h_1=h_{(15)}(A)\in R_V$ with some $A\in{\rm Sp}'_{b_{15}}(\bbf)$ such that
$$p_{I_{(15)}} h_1\widetilde{S}_1=U_{(15),1}^+=\bbf e_{i(15,1)}\oplus\cdots\oplus \bbf e_{i(15,b_{15}-1)}.$$

If $\dim(U_{(15)}^+\cap h_1\widetilde{S}_1)=b_{15}-1$, then we have
\begin{equation}
h_1\widetilde{S}_1=(U_0\cap \widetilde{S}_1)\oplus U_{(15),1}^+. \label{eq6.3}
\end{equation}
On the other hand, if $\dim(U_{(15)}^+\cap h_1\widetilde{S}_1)=b_{15}-2$, then we can take an $h_2=h_{(15)}(B)\in R_V$ with some $B\in{\rm Sp}'_{b_{15}}(\bbf)$ such that
$$h_2(U_{(15)}^+\cap h_1\widetilde{S}_1)=U_{(15),2}^+\quad \mbox{or}\quad U_{(15),3}^+$$
and that $h_2U_{(15),1}^+=U_{(15),1}^+$ where
$$U_{(15),2}^+=\bbf e_{i(15),2}\oplus\cdots\oplus \bbf e_{i(15,b_{15}-1)} \mand U_{(15),3}^+=\bbf e_{i(15,1)}\oplus\cdots\oplus \bbf e_{i(15,b_{15}-2)}.$$
If $h_2(U_{(15)}^+\cap h_1\widetilde{S}_1)=U_{(15),2}^+$, then we have
$$h_2h_1\widetilde{S}_1 =(U_0\cap \widetilde{S}_1)\oplus U_{(15),2}^+\oplus \bbf(e_k+\lambda e_{i(15,1)})$$
with some $\lambda\in\bbf^\times$. Take an element $h_3\in R_V$ such that
$$h_3e_{i(15,1)}=\lambda^{-1}e_{i(15,1)},\quad h_3e_{i(15,b_{15})}=\lambda e_{i(15,b_{15})}$$
and that $h_3e_\ell=e_\ell\mbox{ for }\ell\in I-\{i(15,1),i(15,b_{15})\}$. Then we have
\begin{equation}
h_3h_2h_1\widetilde{S}_1 =(U_0\cap \widetilde{S}_1)\oplus U_{(15),2}^+\oplus \bbf(e_k+e_{i(15,1)}). \label{eq6.4}
\end{equation}
If $h_2(U_{(15)}^+\cap h_1\widetilde{S}_1)=U_{(15,3)}^+$, then we have
$$h_2h_1\widetilde{S}_1 =(U_0\cap \widetilde{S}_1)\oplus U_{(15),3}^+\oplus \bbf(e_k+\lambda e_{i(15,b_{15}-1)})$$
with some $\lambda\in\bbf^\times$. Take an element $h_3\in R_V$ such that
$$h_3e_{i(15,2)}=\lambda e_{i(15,2)},\quad h_3e_{i(15,b_{15}-1)}=\lambda^{-1} e_{i(15,b_{15}-1)}$$
and that $h_3e_\ell=e_\ell\mbox{ for }\ell\in I-\{i(15,2),i(15,b_{15}-1)\}$. Then we have
\begin{equation}
h_3h_2h_1\widetilde{S}_1 =(U_0\cap \widetilde{S}_1)\oplus U_{(15),3}^+\oplus \bbf(e_k+e_{i(15,b_{15}-1)}). \label{eq6.5}
\end{equation}

By Proposition \ref{prop5.18}, there are a finite number of $R_V$-orbits of subspaces $U_+^1$ in $U_+$. So we may assume $U_+^1=\widetilde{S}_1\oplus U^+_{(\overline{12},\overline{8},\overline{6},1)}$ with $\widetilde{S}_1$ of the form
\begin{align*}
& (U_0\cap \widetilde{S}_1)\oplus U_{(15),1}^+,\quad (U_0\cap \widetilde{S}_1)\oplus U_{(15),2}^+\oplus \bbf(e_k+e_{i(15,1)}) \\
\quad\mbox{or}\quad & (U_0\cap \widetilde{S}_1)\oplus U_{(15),3}^+\oplus \bbf(e_k+e_{i(15,b_{15}-1)})
\end{align*}
by (\ref{eq6.3}), (\ref{eq6.4}) and (\ref{eq6.5}).

We have only to show that there are a finite number of $Q(U_+^1)$-orbits of the spaces $U_+^2=\widetilde{S}_2\oplus U^+_{(\overline{12},\overline{8},\overline{6},1)}$ such that $U_+^1\subsetneqq U_+^2\subsetneqq U_+$ where $Q(U_+^1)=\{g\in R_V\mid gU_+^1=U_+^1\}$. Put $W_\mu=\bbf (e_{i(15,b_{15})}+\mu e_k)$ for $\mu\in\bbf$.

(B.2.1) Case of $\widetilde{S}_1=(U_0\cap \widetilde{S}_1)\oplus U_{(15),1}^+$. The space $U_+^2$ equals
$$U_0\oplus U_{(15),1}^+\oplus U^+_{(\overline{12},\overline{8},\overline{6},1)}\quad\mbox{or}\quad U_{+,\mu}^2=(U_0\cap \widetilde{S}_1)\oplus U_{(15),1}^+\oplus W_\mu\oplus U^+_{(\overline{12},\overline{8},\overline{6},1)}$$
with some $\mu\in\bbf$. Take a $g(\mu)=g_{k,i(15,b_{15})}(\mu)\in R_V$ such that
$$g(\mu)e_{i(15,b_{15})}=e_{i(15,b_{15})}+\mu e_k$$
by Lemma \ref{lem5.11}. If $k\notin I_{(8,12)}=I_{(8)}\sqcup I_{(12)}$, then $g(\mu)e_\ell=e_\ell$ for $\ell\in I-\{i(15,b_{15})\}$. If $k=i(8,m)$, then
$$g(\mu)e_{i(\overline{8},b_8+1-m)}=e_{i(\overline{8},b_8+1-m)}-\mu e_{i(15,1)}$$
and $g(\mu)e_\ell=e_\ell$ for $\ell\in I-\{i(15,b_{15}),i(\overline{8},b_8+1-m)\}$. If $k=i(12,m)$, then
$$g(\mu)e_{i(\overline{12},b_{12}+1-m)}=e_{i(\overline{12},b_{12}+1-m)}-\mu e_{i(15,1)}$$
and $g(\mu)e_\ell=e_\ell$ for $\ell\in I-\{i(15,b_{15}),i(\overline{12},b_{12}+1-m)\}$. Hence we have
$g(\mu)U_{+,0}^2=U_{+,\mu}^2$
and $g(\mu)U_+^1=U_+^1$.

(B.2.2) Case of $\widetilde{S}_1=(U_0\cap \widetilde{S}_1)\oplus U_{(15),2}^+\oplus \bbf(e_k+e_{i(15,1)})$. The space $U_+^2$ equals
\begin{align*}
& U_0\oplus U_{(15),2}^+\oplus \bbf(e_k+e_{i(15,1)})\oplus U^+_{(\overline{12},\overline{8},\overline{6},1)} \\
\mbox{or}\quad U_{+,\mu}^2 = & (U_0\cap \widetilde{S}_1)\oplus U_{(15),2}^+\oplus \bbf(e_k+e_{i(15,1)})\oplus W_\mu\oplus U^+_{(\overline{12},\overline{8},\overline{6},1)}
\end{align*}
with some $\mu\in\bbf$. If $k\notin I_{(8,12)}$, then we have $g(\mu)U_{+,0}^2=U_{+,\mu}^2$ and $g(\mu)U_+^1=U_+^1$. When $k=i(8,m)$, we take $g'(-\mu)=g_{k,i(\overline{8},b_8+1-m)}(-\mu)$ such that
$$g'(-\mu)e_{i(\overline{8},b_8+1-m)}=e_{i(\overline{8},b_8+1-m)}-\mu e_k$$
and that $g'(-\mu)e_\ell=e_\ell$ for $\ell\in I-\{i(\overline{8},b_8+1-m)\}$ by Lemma \ref{lem5.11}. Then we have
$$g'(-\mu)g(\mu)U_{+,0}^2=U_{+,\mu}^2\mand g'(-\mu)g(\mu)U_+^1=U_+^1.$$
When $k=i(12,m)$, we take $g'(-\mu)=g_{k,i(\overline{12},b_{12}+1-m)}(-\mu)$ such that
$$g'(-\mu)e_{i(\overline{12},b_{12}+1-m)}=e_{i(\overline{12},b_{12}+1-m)}-\mu e_k$$
and that $g'(-\mu)e_\ell=e_\ell$ for $\ell\in I-\{i(\overline{12},b_{12}+1-m)\}$ by Lemma \ref{lem5.11}. Then we have
$$g'(-\mu)g(\mu)U_{+,0}^2=U_{+,\mu}^2\mand g'(-\mu)g(\mu)U_+^1=U_+^1.$$

(B.2.3) Case of $\widetilde{S}_1=(U_0\cap \widetilde{S}_1)\oplus U_{(15),3}^+\oplus \bbf(e_k+e_{i(15,b_{15}-1)})$. The space $U_+^2$ equals
\begin{align*}
& U_0\oplus U_{(15),3}^+\oplus \bbf(e_k+e_{i(15,b_{15}-1)})\oplus U^+_{(\overline{12},\overline{8},\overline{6},1)} \\
\mbox{or}\quad U_{+,\mu}^2 = & (U_0\cap \widetilde{S}_1)\oplus U_{(15),3}^+\oplus \bbf(e_k+e_{i(15,b_{15}-1)})\oplus W_\mu\oplus U^+_{(\overline{12},\overline{8},\overline{6},1)}
\end{align*}
with some $\mu\in\bbf$. Since $e_{i(15,1)}\in U_{(15),3}^+$, we have $g(\mu)U_{+,0}^2=U_{+,\mu}^2$ and $g(\mu)U_+^1=U_+^1$ as in (B.2.1).

(C) Case of $d_1=0$ and $d_2=2$. Since $\widetilde{S}_1\supset U_0$, we have $\widetilde{S}_j=U_0\oplus (U_{(15)}^+\cap \widetilde{S}_j)$ and 
$U_+^j=\widetilde{S}_j\oplus U^+_{(\overline{12},\overline{8},\overline{6},1)}$.
As in (B.2), we can take an $h=h_{(15)}(A)\in R_V$ with some $A\in{\rm Sp}'_{b_{15}}(\bbf)$ such that $h(U_{(15)}^+\cap \widetilde{S}_2)=U_{(15),1}^+$ and that $h(U_{(15)}^+\cap \widetilde{S}_1)=U_{(15),2}^+$ or $U_{(15),3}^+$. So the assertion is proved.

(D) Case of $d_2=1$. By Proposition \ref{prop5.18}, there are a finite number of $R_V$-orbits of $U_+^1$. So we may fix a $U_+^1=\widetilde{S}_1\oplus U^1_\#$. 
Since $U_+^2=\widetilde{S}_1\oplus U^+_{(\overline{12},\overline{8},\overline{6},1)}$ or $\widetilde{U}\oplus U^1_\#$, the assertion is clear.

(E) The case of $d_2=0$ is clear because $U_+^j=\widetilde{U}\oplus U^j_\#$ for $j=1,2$.

\subsection{Proof of (v)}

Put $d_1=\dim U_0-\dim (U_0\cap \widetilde{S}_2)$ and $d_2=\dim \widetilde{U}-\widetilde{S}_2$. Then $0\le d_1\le d_2\le 1$.

(A) Case of $d_1=d_2=1$. 

(A.1) Suppose $U_+^1\subset \widetilde{U}$. By Proposition \ref{prop7.4}, there are a finite number of $P_1\times K_2$-orbits of flags $\widetilde{S}_1\subset \widetilde{S}_2$ in $\widetilde{U}$. For $g\in P_1\times K_2$, we can take a $\widetilde{g}\in R_V$ such that
\begin{align*}
\widetilde{g}|_{\widetilde{U}} & =g,\quad \widetilde{g}u =u\mbox{ for }u\in U_{(\overline{6},1)}^+,\quad \widetilde{g}U_{(\overline{8})}^+=U_{(\overline{8})}^+ \\
\mbox{and that}\quad \widetilde{g}U_{(\overline{12})}^+ & =\bbf(e_{i(\overline{12},1)}+v_1)\oplus\cdots\oplus \bbf(e_{i(\overline{12},b_{12})}+v_{b_{12}})
\end{align*}
with some $v_1,\ldots,v_{b_{12}}\in U_{(15)}^+$. Hence we may assume
$$U_+^2=\widetilde{S}_2\oplus \widetilde{g}U_{(\overline{12})}^+\oplus U_{(\overline{8},\overline{6},1)}^+.$$
We have only to show that there exists a $g'\in R_V$ such that
$$g'U_+^2=\widetilde{S}_2\oplus U^+_{(\overline{12},\overline{8},\overline{6},1)}$$
and that $g'U_+^1=U_+^1$.

Since $p_{I_{(15)}} \widetilde{S}_2=U_{(15)}^+$, we can take $u_1,\ldots,u_{b_{12}}\in U_0$ such that
$$v_k\in u_k+\widetilde{S}_2$$
for $k=1,\ldots,b_{12}$. So we have
$$\bbf(e_{i(\overline{12},1)}+u_1)\oplus\cdots\oplus \bbf(e_{i(\overline{12},b_{12})}+u_{b_{12}}) \subset U_+^2.$$
By Lemma \ref{lem5.11}, we can take a $g'\in R_V$ such that
$$g'(e_{i(\overline{12},k)}+u_k)=e_{i(\overline{12},k)}$$
for $k=1,\ldots,b_{12}$ and that $g'e_\ell=e_\ell$ for $\ell\in I_+-I_{(\overline{12})}$. Hence $g'U_+^2=\widetilde{S}_2\oplus U^+_{(\overline{12},\overline{8},\overline{6},1)}$ and $g'U_+^1=U_+^1$.

(A.2) Suppose $U_+^1\not\subset \widetilde{U}$. Then $U_+^1=U^1_\#$. We can take a nontrivial linear form $f$ on $\widetilde{U}$ such that $\widetilde{S}_2=\{v\in \widetilde{U}\mid f(v)=0\}$. By Lemma \ref{lem5.10} and Lemma \ref{lem5.11'}, we can take an $h\in R_V$ such that
$$f(he_k)\begin{cases} \in\{0,1\} & \text{for $k\in I_0\sqcup\{i(15,1)\}$,} \\
=0 & \text{for $k=i(15,2),\ldots,i(15,b_{15})$} \end{cases}$$
where $I_0=I_{(2,7,8,10,13,3,5,9,12)}$. Since $hU^1_\#=U^1_\#$, there are a finite number of $R_V$-orbits of flags $U_+^1\subset U_+^2$.

(B) Case of $d_1=0$ and $d_2=1$. Since $\widetilde{S}_2\supset U_0$, we have $\widetilde{S}_2=U_0\oplus (U_{(15)}^+\cap \widetilde{S}_2)$. We can take an $h=h_{(15)}(A)\in R_V$ with some $A\in{\rm Sp}'_{b_{15}}(\bbf)$ such that $h(U_{(15)}^+\cap \widetilde{S}_2)=U_{(15),1}^+$ as in Section \ref{sec6.5} (B). So we may assume
$$U_+^2=U_0\oplus U_{(15),1}^+\oplus U_{(\overline{12},\overline{8},\overline{6},1)}^+.$$
We have only to consider one-dimensional subspaces $U_+^1$ in $U_+^2$.

If $U_+^1\not\subset U_0\oplus U_{(15),1}^+$, then the assertion is clear since $U_+^1=U^1_\#$. So we may assume $U_+^1\subset U_0\oplus U_{(15),1}^+$. Write $U_+^1=\bbf v$ with $v=(\sum_{k\in I_0} \lambda_ke_k)+v_{(15)}$ and $v_{(15)}\in U^+_{(15),1}$. By Lemma \ref{lem5.10} and Lemma \ref{lem5.11'}, we can take an $h\in R_V$ such that $h\lambda e_k\in \{0,e_k\}$ for $k\in I_0$, $hv_{(15)}\in \{0,e_{i(15,1)},e_{i(15,2)}\}$ and that $hU_+^2=U_+^2$. So the assertion is proved.

(C) Case of $d_1=d_2=0$. we have $U_+^2=\widetilde{U}\oplus U^2_\#$. If $U_+^1\not\subset \widetilde{U}$, then the assertion is clear since $U_+^1=U^1_\#$. So we may assume $U_+^1\subset \widetilde{U}$. By the same argument as in (B), we can take an $h\in R_V$ such that $hU_+^1=\bbf u$ with some $u\in (\sum_{k\in I_0} \{0,e_k\})+\{0,e_{i(15,1)}\}$ and that $hU_+^2=U_+^2$.

Thus we have completed the proof of Proposition \ref{prop4.1}.

\section{Proof of Propositions \ref{prop4.2'} and \ref{prop4.3'}}

\subsection{Proof of Proposition \ref{prop4.2'}}

We have only to show that there are a finite number of $R_V$-orbits of flags $U_+^1\subset U_+^2\subset U_+^3$ with $\dim U_+^j=j$ in $U_+=\bbf e_1\oplus\cdots\oplus \bbf e_n$. Since $\dim U_+=n$, we have 
$$b_k=0\quad\mbox{for }k=4,7,9,11,12,13,14.$$
As in Section \ref{sec6.1}, we may assume
$$U_+^j=\widetilde{S}_j\oplus U_{(\overline{8},\overline{6},1)}^{+,s_{2,j},s_{1,j},s_{0,j}}$$
for $j=1,2,3$. We have only to show that there are a finite number of $Q$-orbits of flags $\widetilde{S}_1\subset \widetilde{S}_2\subset \widetilde{S}_3$ where $Q$ is the subgroup of ${\rm GL}(\widetilde{U})$ defined in Section \ref{sec6.1} for $m=3$.

(A) Case of $\dim \widetilde{S}_3=0$. The assertion is clear.

(B) Case of $\dim \widetilde{S}_3=1$. Since $P_1\times K_2\subset Q$ and since there are a finite number of $P_1\times K_2$-orbits of one-dimensional subspaces in $\widetilde{U}$ by Proposition 6.3 in \cite{M3}, the assertion is clear.

(C) Case of $\dim \widetilde{S}_3=2$. We have
$$\widetilde{S}_1=\{0\},\quad \widetilde{S}_2=\widetilde{S}_1\quad\mbox{or}\quad \widetilde{S}_3=\widetilde{S}_2.$$
So we have only to show that there are a finite number of $Q$-orbits of flags $S'\subset \widetilde{S}_3$ with $\dim S'=1$.

Since $P_1\times K_2\subset Q$ and since there are a finite number of $P_1\times K_2$-orbits of flags $S'\subset \widetilde{S}_3$ in $\widetilde{U}$ by Proposition \ref{prop7.4}, the assertion is clear.

(D) Case of $\dim \widetilde{S}_3=3$. Write $W=U_+^3=\widetilde{S}_3$. Let $Q(W')$ denote the restriction of $\{g\in R_V\mid gW'=W'\}$ to $W'=g'W$ for some $g'\in R_V$. Then we have only to show $|Q(W')\backslash M(W')|<\infty$.

(D.1) Suppose that $\dim p_{I_{(15)}} W=3$ and that $p_{I_{(15)}} W$ is not isotropic with respect to the alternating form $\langle\ ,\ \rangle$ on $U_{(15)}^+$. By Lemma \ref{lem5.11'}, we can take an $h_1=h_{(15)}(B)\in R_V$ with some $B\in {\rm Sp}'_{b_{15}}(\bbf)$ such that
$$h_1p_{I_{(15)}} W=U_{(15),1}^+=\bbf e_{i(15,1)}\oplus\bbf e_{i(15,2)}\oplus\bbf e_{i(15,b_{15}-1)}.$$
By Lemma \ref{lem5.11} (with Figure \ref{fig6.2}), we can take a $g_1\in R_V$ such that
$$W'=g_1h_1W\subset U_{(15),1}^+\oplus U_{(10)}^+.$$

(D.1.1) Case of $W'= U_{(15),1}^+$. With respect to the basis $e_{i(15,1)},\  e_{i(15,2)},\ e_{i(15,b_{15}-1)}$ of $W'$, elements of $Q(W')$ are represented by matrices
$$\bp \lambda & * \\ 0 & A \ep$$
with $\lambda\in\bbf^\times$ and $A\in{\rm SL}_2(\bbf)$. Hence $|Q(W')\backslash M(W')|<\infty$.

(D.1.2) Case of $W'\cap U_{(15),1}^+=\{0\}$. By Lemma \ref{lem5.10}, we may assume
$$W'=\bbf(e_{i(15,1)}+e_{i(10,1)})\oplus\bbf (e_{i(15,2)}+e_{i(10,2)})\oplus\bbf (e_{i(15,b_{15}-1)}+e_{i(10,3)}).$$
With respect to the basis $e_{i(15,1)}+e_{i(10,1)},\ e_{i(15,2)}+e_{i(10,2)},\ e_{i(15,b_{15}-1)}+e_{i(10,3)}$ of $W'$, elements of $Q(W')$ are represented by
$$\bp \lambda & * \\ 0 & A \ep$$
with $\lambda\in\bbf^\times$ and $A\in{\rm SL}_2(\bbf)$. Hence $|Q(W')\backslash M(W')|<\infty$.

(D.1.3) Case of $W'\cap U_{(15),1}^+=\bbf e_{i(15,1)}$. By Lemma \ref{lem5.10}, we may assume
$$W'=\bbf e_{i(15,1)}\oplus\bbf (e_{i(15,2)}+e_{i(10,1)})\oplus\bbf (e_{i(15,b_{15}-1)}+e_{i(10,2)}).$$
With respect to the basis $e_{i(15,1)},\ e_{i(15,2)}+e_{i(10,1)},\ e_{i(15,b_{15}-1)}+e_{i(10,2)}$ of $W'$, elements of $Q(W')$ are represented by
$$\bp \lambda & * \\ 0 & A \ep$$
with $\lambda\in\bbf^\times$ and $A\in{\rm SL}_2(\bbf)$. Hence $|Q(W')\backslash M(W')|<\infty$.

(D.1.4) Case of $\dim(W'\cap U_{(15),1}^+)=1$ and $W'\cap U_{(15),1}^+\ne \bbf e_{i(15,1)}$. We can take an $h_2=h_{(15)}(B)\in R_V$ with some $B\in{\rm Sp}'_{b_{15}}(\bbf)$ such that
$$h_2U_{(15),1}=U_{(15),1}\quad\mbox{and that}\quad h_2(W'\cap U_{(15)}^+) =\bbf e_{i(15,2)}.$$
So we may assume that
$$W'=\bbf (e_{i(15,1)}+e_{i(10,1)})\oplus\bbf e_{i(15,2)}\oplus\bbf (e_{i(15,b_{15}-1)}+e_{i(10,2)}).$$
With respect to the basis $e_{i(15,1)}+e_{i(10,1)},\ e_{i(15,2)},\ e_{i(15,b_{15}-1)}+e_{i(10,2)}$ of $W'$, elements of $Q(W')$ are represented by
$$\bp \lambda & 0 & * \\ 0 & a & * \\ 0 & 0 & a^{-1} \ep$$
with $\lambda,a\in\bbf^\times$. Since the subgroup of ${\rm GL}_2(\bbf)\times\bbf^\times$ consisting of elements
$$\left(\bp \lambda & 0 \\ 0 & a \ep,\ a^{-1}\right)$$
has a finite number of orbits on $M(\bbf^2)\times M(\bbf)$, we have $|Q(W')\backslash M(W')|<\infty$.

(D.1.5) Case of $\dim(W'\cap U_{(15),1}^+)=2$ and $W'\ni e_{i(15,1)}$. We may assume
$$W'=\bbf e_{i(15,1)}\oplus\bbf e_{i(15,2)}\oplus\bbf (e_{i(15,b_{15}-1)}+e_{i(10,1)}).$$
With respect to the basis $e_{i(15,1)},\ e_{i(15,2)},\ e_{i(15,b_{15}-1)}+e_{i(10,1)}$ of $W'$, elements of $Q(W')$ are represented by
$$\bp \lambda & * & * \\ 0 & a & * \\ 0 & 0 & a^{-1} \ep$$
with $\lambda,a\in\bbf^\times$. So we have $|Q(W')\backslash M(W')|<\infty$.

(D.1.6) Case of $\dim(W'\cap U_{(15),1}^+)=2$ and $W'\not\ni e_{i(15,1)}$. We may assume
$$W'=\bbf (e_{i(15,1)}+e_{i(10,1)})\oplus\bbf e_{i(15,2)}\oplus\bbf e_{i(15,b_{15}-1)}.$$
With respect to the basis $e_{i(15,1)}+e_{i(10,1)},\ e_{i(15,2)},\ e_{i(15,b_{15}-1)}$ of $W'$, elements of $Q(W')$ are represented by
$$\bp \lambda & 0 \\ 0 & A \ep$$
with $\lambda\in\bbf^\times$ and $A\in{\rm SL}_2(\bbf)$. So we have $|Q(W')\backslash M(W')|<\infty$ by Lemma \ref{lem12.13} in the appendix.

(D.2) Suppose that $\dim p_{I_{(15)}} W=3$ and that $p_{I_{(15)}} W$ is  isotropic with respect to the alternating form $\langle\ ,\ \rangle$ on $U_{(15)}^+$. (Hence $b_{15}\ge 6$.) Then we can take an $h_1\in R_V$ such that
$$h_1p_{I_{(15)}} W=U_{(15),1}^+=\bbf e_{i(15,1)}\oplus\bbf e_{i(15,2)}\oplus\bbf e_{i(15,3)}.$$
By Lemma \ref{lem5.11}, we can take $g_1\in R_V$ such that
$$g_1h_1W\subset U_{(15),1}^+\oplus U_{(10)}^+.$$
Put $d=\dim (W\cap U_{(15),1}^+)$. Then we can take an $h_2\in R_V$ such that
$$W'=h_2g_1h_1W=(\bigoplus_{1\le k\le d} \bbf e_{i(15,k)})\oplus (\bigoplus_{1\le\ell\le 3-d} \bbf (e_{i(15,d+\ell)}+e_{i(10,\ell)}))$$
by Lemma \ref{lem5.10} and Lemma \ref{lem5.11'}. Let $P$ denote the parabolic subgroup of ${\rm GL}(W')$ stabilizing the subspace $W'\cap U_{(15),1}^+$. Then we have $Q(W')=P$. Hence $|Q(W')\backslash M(W')|<\infty$ by the Bruhat decomposition.

(D.3) Suppose that $\dim p_{I_{(15)}} W=2$ and that $p_{I_{(15)}} W$ is not isotropic with respect to $\langle\ ,\ \rangle$. By Lemma \ref{lem5.10}. we can take an $h_1=h_{(15)}(B)\in R_V$ with some $B\in {\rm Sp}'_{b_{15}}(\bbf)$ such that
$$h_1p_{I_{(15)}} W=U_{(15),1}^+=\bbf e_{i(15,1)}\oplus\bbf e_{i(15,b_{15})}.$$
Put $W_1=W\cap U_0$ where $U_0=U_{(2,3,5,8,10)}^+$. Then $\dim W_1=1$. First suppose $W_1=\bbf v\subset U'_0=U_{(2,3,5,8)}^+$. Take a complementary subspace $W_2$ of $W_1$ in $W$. Then by Lemma \ref{lem5.11}, we can take a $g_1\in R_V$ such that
$$W'_2=g_1h_1W_2\subset U_{(15),1}^+\oplus U_{(10)}^+.$$
Put $W'=g_1W=W_1\oplus W'_2$.

(D.3.1) Case of $W'_2= U_{(15),1}^+$. With respect to the basis $v,\ e_{i(15,1)},\ e_{i(15,b_{15})}$ of $W'$, elements of $Q(W')$ are represented by
$$\bp \lambda & * \\ 0 & A \ep$$
with $\lambda\in\bbf^\times$ and $A\in{\rm SL}_2(\bbf)$. Hence $|Q(W')\backslash M(W')|<\infty$.

(D.3.2) Case of $W'_2\cap U_{(15),1}^+=\{0\}$. We may assume
$$W'=\bbf v\oplus \bbf(e_{i(15,1)}+e_{i(10,1)})\oplus \bbf(e_{i(15,b_{15})}+e_{i(10,2)})$$
by Lemma \ref{lem5.10}. With respect to the basis $v,\ e_{i(15,1)}+e_{i(10,1)},\ e_{i(15,b_{15})}+e_{i(10,2)}$ of $W'$, elements of $Q(W')$ are represented by
$$\bp \lambda & * \\ 0 & A \ep$$
with $\lambda\in\bbf^\times$ and $A\in{\rm SL}_2(\bbf)$. Hence $|Q(W')\backslash M(W')|<\infty$.

(D.3.3) Case of $\dim(W'_2\cap U_{(15),1}^+)=1$. We may assume
$$W'=\bbf v\oplus \bbf e_{i(15,1)}\oplus \bbf(e_{i(15,b_{15})}+e_{i(10,1)}).$$
With respect to the basis $v,\ e_{i(15,1)}\ e_{i(15,b_{15})}+e_{i(10,1)}$ of $W'$, elements of $Q(W')$ are represented by
$$\bp \lambda & * & * \\ 0 & a & * \\ 0 & 0 & a^{-1} \ep$$
with $\lambda,a\in\bbf^\times$. Hence $|Q(W')\backslash M(W')|<\infty$.

\bigskip
Next suppose $W_1=\bbf v\not\subset U_0$. By lemma \ref{lem5.10}, we may assume $p_{I_{(10)}} W_1=\bbf e_{i(10,1)}$. Put $U_{(10),2}^+=\bbf e_{i(10,2)}\oplus\cdots\oplus \bbf e_{i(10,b_{10})}$. Take a complementary subspace $W_2$ of $W_1$ in $W$ such that $p_{I_{(10)}} W_2\subset U_{(10),2}^+$. Then by Lemma \ref{lem5.11}, we can take a $g_1\in R_V$ such that
$$W'_2=g_1W_2\subset U_{(15)}^+\oplus U_{(10),2}^+.$$

(D.3.4) Case of $W'_2= U_{(15),1}^+$. With respect to the basis $v,\ e_{i(15,1)},\ e_{i(15,b_{15})}$ of $W'=W_1\oplus W'_2$, elements of $Q(W')$ are represented by
$$\bp \lambda & 0 \\ 0 & A \ep$$
with $\lambda\in\bbf^\times$ and $A\in{\rm SL}_2(\bbf)$. Hence we have $|Q(W')\backslash M(W')|<\infty$ by Lemma \ref{lem12.13}.

(D.3.5) Case of $W'_2\cap U_{(15),1}^+=\{0\}$. We may assume
$$W'=\bbf v\oplus \bbf(e_{i(15,1)}+e_{i(10,2)})\oplus \bbf(e_{i(15,b_{15})}+e_{i(10,3)}).$$
With respect to the basis $v,\ e_{i(15,1)}+e_{i(10,2)},\ e_{i(15,b_{15})}+e_{i(10,3)}$ of $W'$, elements of $Q(W')$ are represented by
$$\bp \lambda & * \\ 0 & A \ep$$
with $\lambda\in\bbf^\times$ and $A\in{\rm SL}_2(\bbf)$. Hence $|Q(W')\backslash M(W')|<\infty$.

(D.3.6) Case of $\dim(W'_2\cap U_{(15),1}^+)=1$. We may assume
$$W'=\bbf v\oplus \bbf e_{i(15,1)}\oplus \bbf(e_{i(15,b_{15})}+e_{i(10,2)}).$$
With respect to the basis $v,\ e_{i(15,1)}\ e_{i(15,b_{15})}+e_{i(10,2)}$ of $W'$, elements of $Q(W')$ are represented by
$$\bp \lambda & 0 & * \\ 0 & a & * \\ 0 & 0 & a^{-1} \ep$$
with $\lambda,a\in\bbf^\times$. Hence we have $|Q(W')\backslash M(W')|<\infty$ as in (D.1.4).

(D.4) Suppose that $\dim p_{I_{(15)}} W=2$ and that $p_{I_{(15)}} W$ is isotropic with respect to $\langle\ ,\ \rangle$. Then we may assume
$$p_{I_{(15)}} W=\bbf e_{i(15,1)}\oplus\bbf e_{i(15,2)}.$$
We can prove $|Q(W')\backslash M(W')|<\infty$ by the same arguments as in (D.3).

(D.5) Suppose that $\dim p_{I_{(15)}} W=1$. Then we may assume
$p_{I_{(15)}} W=\bbf e_{i(15,1)}$.
Put $W_1=W\cap U_0$.

(D.5.1) Case of $W_1\subset U'_0=U_{(2,8,3,5)}^+$. By Lemma \ref{lem11.5} in the appendix, we can take a basis $v_1,v_2$ of $W_1$ such that $R_V$ contains elements $g_{\lambda,\mu}$ satisfying
$$g_{\lambda,\mu}v_1=\lambda v_1,\quad g_{\lambda,\mu}v_2=\mu v_2\mand g_{\lambda,\mu}\mbox{ acts trivially on }U_{(10)}^+\oplus U_{(15)}^+.$$
Let $W_2$ be a complementary subspace of $W_1$ in $W$. By Lemma \ref{lem5.10} and Lemma \ref{lem5.11}, we can take a $g_1\in R_V$ such that $W'_2=g_1W_2=\bbf v_3$ and that $g_1W_1=W_1$ where $v_3=e_{i(15,1)}+\varepsilon e_{i(10,1)}$ with $\varepsilon\in\{0,1\}$. Put $W'=g_1W=W_1\oplus W'_2$.

Let $Q'$ be the subgroup of ${\rm GL}(W')$ consisting of elements represented as
$$\bp \lambda & 0 & * \\ 0 & \mu & * \\ 0 & 0 & \nu \ep$$
($\lambda,\mu,\nu\in\bbf^\times$) with respect to the basis $v_1,v_2,v_3$. Then $Q'$ is contained in $Q(W')$. Hence $|Q(W')\backslash M(W')|<\infty$.

(D.5.2) Case of $\dim(W_1\cap U'_0)=1$. Put $W_{1,1}=W_1\cap U'_0=\bbf v_1$.

First suppose that $W_{1,1}\subset U_{(2,3,8)}^+$. Let $W_{1,2}$ be a complementary subspace of $W_{1,1}$ in $W_1$. Then by Lemma \ref{lem5.10} and Lemma \ref{lem5.11}, we can take a $g_1\in R_V$ such that $g_1W_{1,2}=\bbf v_2$ with $v_2=e_{i(10,1)}+\varepsilon_1 e_{i(5,1)}\ (\varepsilon_1\in\{0,1\})$ and that $g_1W_{1,1}=W_{1,1}$. Take a complementary subspace $W_2$ of $g_1W_1$ in $g_1W$ such that $p_{I_{(10)}} W_2\subset U_{(10),2}^+$ where $U_{(10),2}^+=\bbf e_{i(10,2)}\oplus\cdots\oplus \bbf e_{i(10,b_{10})}$. Then we can take a $g_2\in R_V$ such that $W'_2=g_2W_2=\bbf v_3$ with $v_3=e_{i(15,1)}+\varepsilon_2 e_{i(10,2)}\ (\varepsilon_2\in\{0,1\})$ and that $g_2g_1W_1=g_1W_1$ by Lemma \ref{lem5.10} and Lemma \ref{lem5.11}. Put $W'=g_2g_1W$. Let $Q'$ be the subgroup of ${\rm GL}(W')$ consisting of elements represented by matrices
$$\bp \lambda & 0 & * \\ 0 & \mu & 0 \\ 0 & 0 & \nu \ep$$
$(\lambda,\mu,\nu\in\bbf^\times)$ with respect to the basis $v_1,v_2,v_3$. Then $Q'\subset Q(W')$ and hence $|Q(W')\backslash M(W')|<\infty$ by Lemma \ref{lem5.1}.

Next suppose that  $W_{1,1}\not\subset U_{(2,3,8)}^+$. Then we may assume $p_{I_{(5)}} W_{1,1}=\bbf e_{i(5,1)}$. Let $W_{1,2}$ be a complementary subspace of $W_{1,1}$ in $W_1$ such that $p_{I_{(5)}}(W_{1,2})\subset \bbf e_{i(5,2)}\oplus\cdots\oplus \bbf e_{i(5,b_5)}$. Then by Lemma \ref{lem5.10} and Lemma \ref{lem5.11}, we can take a $g_1\in R_V$ such that $g_1W_{1,2}=\bbf v_2$ with $v_2=e_{i(10,1)}+\varepsilon_1 e_{i(5,2)}\ (\varepsilon_1\in\{0,1\})$ and that $g_1W_{1,1}=W_{1,1}$. Take a complementary subspace $W_2$ of $g_1W_1$ in $g_1W$ such that $p_{I_{(10)}} W_2\subset U_{(10),2}^+$. Then we can take a $g_2\in R_V$ such that $W'_2=g_2W_2=\bbf v_3$ with $v_3=e_{i(15,1)}+\varepsilon_2 e_{i(10,2)}\ (\varepsilon_2\in\{0,1\})$ and that $g_2g_1W_1=g_1W_1$. Put $W'=g_2g_1W$. Let $Q'$ be the subgroup of ${\rm GL}(W')$ consisting of elements represented by matrices
$$\bp \lambda & 0 & * \\ 0 & \mu & 0 \\ 0 & 0 & \nu \ep$$
$(\lambda,\mu,\nu\in\bbf^\times)$ with respect to the basis $v_1,v_2,v_3$. Then $Q'\subset Q(W')$ and hence $|Q(W')\backslash M(W')|<\infty$.

(D.5.3) Case of $W_1\cap U'_0=\{0\}$. By Lemma \ref{lem5.10} and Lemma \ref{lem5.11}, we can take a $g_1\in R_V$ such that $g_1W_1\subset U_{(5)}^+\oplus U_{(10)}^+$. Let $W_2$ be a complementary subspace of $g_1W_1$ in $g_1W$. We can also take a $g_2\in R_V$ such that $g_2W_2\subset \bbf e_{i(15,1)}\oplus U_{(10)}^+$ and that $g_2g_1W_1=g_1W_1$. Hence we have $W'=g_2g_1W\subset U_\#=U_{(5)}^+\oplus \bbf e_{i(15,1)}\oplus U_{(10)}^+$.

Let $Q_\#$ be the restriction of $\{g\in R_V\mid gU_\#=U_\#\}$ to $U_\#$. Then with respect to the basis $e_{i(5,1)},\ldots,e_{i(5,b_5)},e_{i(15,1)},e_{i(10,1)},\ldots,e_{i(10,b_{10})}$ of $U_\#$, elements of $Q_\#$ are represented by matrices
$$\bp A & * & 0 \\ 0 & \lambda & 0 \\ 0 & 0 & B \ep$$
with $\lambda\in\bbf^\times,\ A\in{\rm GL}_{b_5}(\bbf)$ and $B\in{\rm GL}_{b_{10}}(\bbf)$ by Lemma \ref{lem5.10} and Lemma \ref{lem5.11}. By Lemma \ref{lem9.8} in the appendix (case of $\alpha_1=b_{10},\ \alpha_2=b_5$ and $\alpha_3=\alpha_4=\alpha_5=0$), we have $$|Q_\#\backslash M(U_\#)|<\infty.$$
Hence we have $|Q(W')\backslash M(W')|<\infty$.

(D.6) Finally suppose that $\dim p_{I_{(15)}} W=0$. Then
$$W\subset U_0=U_{(2,8,10,3,5)}^+.$$

(D.6.1) Case of $\dim p_{I_{(5)}}W=3$. By Lemma \ref{lem5.11}, we may assume
$W\subset U_{(2,8,10,5)}^+$.
Put
\begin{align*}
d_0 & =\dim (W\cap U^+_{(5)}),\quad d_1=\dim (W\cap U^+_{(2,5)})-d_0, \\
d_2 & =\dim (W\cap U^+_{(2,8,5)})-d_0-d_1,\quad d_3=3-d_0-d_1-d_2
\end{align*}
and
\begin{align*}
W' & =(\bigoplus_{1\le k_0\le d_0} \bbf e_{i(5,k_0)}) \oplus(\bigoplus_{1\le k_1\le d_1} \bbf (e_{i(5,d_0+k_1)}+e_{i(2,k_1)}) \\
& \quad \oplus(\bigoplus_{1\le k_2\le d_2} \bbf (e_{i(5,d_0+d_1+k_2)}+e_{i(8,k_2)}) \oplus(\bigoplus_{1\le k_3\le d_3} \bbf (e_{i(5,d_0+d_1+d_2+k_3)}+e_{i(10,k_3)}).
\end{align*}
Then we have $gW=W'$ for some $g\in R_V$ by Lemma \ref{lem5.10} and Lemma \ref{lem5.11}. Let $P$ denote the parabolic subgroup of ${\rm GL}(W')$ stabilizing the flag
$$W'\cap U^+_{(5)}\subset W'\cap U^+_{(2,5)}\subset W'\cap U^+_{(2,8,5)}.$$
Then we have $Q(W')=P$ by Lemma \ref{lem5.10} and Lemma \ref{lem5.11}. Hence $|Q(W')\backslash M(W')|<\infty$ by the Bruhat decomposition.

(D.6.2) Case of $\dim p_{I_{(5)}}W=2$. Put $W_1=W\cap U_{(2,8,10,3)}^+$. Then $\dim W_1=1$. By Lemma \ref{lem5.10} and Lemma \ref{lem5.11}, we can take a $g_1\in R_V$ such that $g_1W_1=\bbf v_1$ where $v_1$ is one of the five vectors
$$e_{i(3,1)},\quad e_{i(3,1)}+e_{i(2,1)},\quad e_{i(2,1)},\quad  e_{i(8,1)}\quad\mbox{or}\quad e_{i(10,1)}.$$

(D.6.2.1) Case of $v_1=e_{i(10,1)}$. We can take a subspace $W_2$ of $g_1W$ such that $g_1W=\bbf v_1\oplus W_2$ and that $W_2\subset U_{(2,8,3,5)}^+\oplus U^+_{(10),1}$ where $U^+_{(10),1}=\bbf e_{i(10,2)}\oplus\cdots\oplus \bbf e_{i(10,b_{10})}$. Since $\dim p_{I_{(5)}}W_2=2$, we can take a $g_2\in R'_V=\{g\in R_V\mid gv_1=v_1\}$ such that
$$g_2W_2\subset U_{(2,8,5)}^+\oplus U^+_{(10),1}$$
by Lemma \ref{lem5.11}. By Lemma \ref{lem5.10} and Lemma \ref{lem5.11}, we can take a $g_3\in R'_V$ such that $g_3g_2W_2=\bbf (e_{i(5,1)}+v_2)\oplus\bbf (e_{i(5,2)}+v_3)$ where $(v_2,v_3)$ is one of
\begin{align*}
& (0,0),\quad (0,e_{i(2,1)}),\quad (0,e_{i(8,1)}),\quad (0,e_{i(10,2)}),\quad (e_{i(2,1)},e_{i(2,2)}),\quad (e_{i(2,1)},e_{i(8,1)}), \\
& (e_{i(2,1)},e_{i(10,2)}),\quad (e_{i(8,1)},e_{i(8,2)}),\quad (e_{i(8,1)},e_{i(10,2)})\quad\mbox{or}\quad (e_{i(10,2)},e_{i(10,3)}).
\end{align*}
Put $W'=g_3g_2g_1W=\bbf v_1\oplus \bbf (e_{i(5,1)}+v_2)\oplus\bbf (e_{i(5,2)}+v_3)$. For $(\lambda_1,\lambda_2,\lambda_3)\in(\bbf^\times)^3$, we can take an $h\in R_V$ such that
$$hv_1=\lambda_1v_1,\quad h(e_{i(5,1)}+v_2)=\lambda_2(e_{i(5,1)}+v_2)\quad\mbox{and that}\quad h(e_{i(5,2)}+v_3)=\lambda_3(e_{i(5,2)}+v_3)$$
by Lemma \ref{lem5.10}. On the other hand, we can take a $g\in R_V$ such that $gv_1=v_1,\ g(e_{i(5,1)}+v_2)=e_{i(5,1)}+v_2$ and that
$$g(e_{i(5,2)}+v_3)=e_{i(5,2)}+v_3+\lambda(e_{i(5,1)}+v_2)$$
for $\lambda\in\bbf$ by Lemma \ref{lem5.10} and Lemma \ref{lem5.11}. Hence $Q(W')$ contains a subgroup $H$ represented by matrices
$$\bp \lambda_1 & 0 & 0 \\ 0 & \lambda_2 & \lambda \\ 0 & 0 & \lambda_3 \ep \quad(\lambda_1,\lambda_2,\lambda_3\in\bbf^\times,\ \lambda\in\bbf)$$
with respect to the basis $v_1,e_{i(5,1)}+v_2,e_{i(5,2)}+v_3$. Since $|H\backslash M(W')|<\infty$ by Lemma \ref{lem5.1}, we have $|Q(W')\backslash M(W')|<\infty$.

(D.6.2.2) Case of $v_1=e_{i(8,1)}$. In the same way as in (D.6.2.1), we can take $g_2,g_3\in R_V$ such that $g_3g_2W_2=\bbf (e_{i(5,1)}+v_2)\oplus\bbf (e_{i(5,2)}+v_3)$ where $(v_2,v_3)$ is one of
\begin{align*}
& (0,0),\quad (0,e_{i(2,1)}),\quad (0,e_{i(8,2)}),\quad (0,e_{i(10,1)}),\quad (e_{i(2,1)},e_{i(2,2)}),\quad (e_{i(2,1)},e_{i(8,2)}), \\
& (e_{i(2,1)},e_{i(10,1)}),\quad (e_{i(8,2)},e_{i(8,3)}),\quad (e_{i(8,2)},e_{i(10,1)})\quad\mbox{or}\quad (e_{i(10,1)},e_{i(10,2)}).
\end{align*}
Put $W'=g_3g_2g_1W=\bbf v_1\oplus \bbf (e_{i(5,1)}+v_2)\oplus\bbf (e_{i(5,2)}+v_3)$. Then $Q(W')$ contains a subgroup $H$ of the form in (D.6.2.1) and hence $|Q(W')\backslash M(W')|<\infty$.

(D.6.2.3) Case of $v_1=e_{i(2,1)}$. In the same way as in (D.6.2.1), we can take $g_2,g_3\in R_V$ such that $g_3g_2W_2=\bbf (e_{i(5,1)}+v_2)\oplus\bbf (e_{i(5,2)}+v_3)$ where $(v_2,v_3)$ is one of
\begin{align*}
& (0,0),\quad (0,e_{i(2,2)}),\quad (0,e_{i(8,1)}),\quad (0,e_{i(10,1)}),\quad (e_{i(2,2)},e_{i(2,3)}),\quad (e_{i(2,2)},e_{i(8,1)}), \\
& (e_{i(2,2)},e_{i(10,1)}),\quad (e_{i(8,1)},e_{i(8,2)}),\quad (e_{i(8,1)},e_{i(10,1)})\quad\mbox{or}\quad (e_{i(10,1)},e_{i(10,2)}).
\end{align*}
Put $W'=g_3g_2g_1W=\bbf v_1\oplus \bbf (e_{i(5,1)}+v_2)\oplus\bbf (e_{i(5,2)}+v_3)$. Then $Q(W')$ contains a subgroup $H$ of the form in (D.6.2.1) and hence $|Q(W')\backslash M(W')|<\infty$.

(D.6.2.4) Case of $v_1=e_{i(3,1)}$. In the same way as in (D.6.2.1), we can take $g_2,g_3\in R_V$ such that $g_3g_2W_2=\bbf (e_{i(5,1)}+v_2)\oplus\bbf (e_{i(5,2)}+v_3)$ where $(v_2,v_3)$ is one of
\begin{align*}
& (0,0),\quad (0,e_{i(2,1)}),\quad (0,e_{i(8,1)}),\quad (0,e_{i(10,1)}),\quad (e_{i(2,1)},e_{i(2,2)}),\quad (e_{i(2,1)},e_{i(8,1)}), \\
& (e_{i(2,1)},e_{i(10,1)}),\quad (e_{i(8,1)},e_{i(8,2)}),\quad (e_{i(8,1)},e_{i(10,1)})\quad\mbox{or}\quad (e_{i(10,1)},e_{i(10,2)}).
\end{align*}
Put $W'=g_3g_2g_1W=\bbf v_1\oplus \bbf (e_{i(5,1)}+v_2)\oplus\bbf (e_{i(5,2)}+v_3)$. Then $Q(W')$ contains a subgroup $H$ of the form in (D.6.2.1) and hence $|Q(W')\backslash M(W')|<\infty$.

(D.6.2.5) Case of $v_1=e_{i(3,1)}+e_{i(2,1)}$. We can take a subspace $W_2$ of $g_1W$ such that $g_1W=\bbf v_1\oplus W_2$ and that $W_2\subset U_{(8,3,10,5)}^+\oplus U^+_{(2),1}$ where $U^+_{(2),1}=\bbf e_{i(2,2)}\oplus\cdots\oplus \bbf e_{i(2,b_2)}$. Since $\dim p_{I_{(5)}}W_2=2$, we can take a $g_2\in R'_V=\{g\in R_V\mid gv_1=v_1\}$ such that
$$g_2W_2\subset U_{(8,10,5)}^+\oplus U^+_{(2),1}$$
by Lemma \ref{lem5.11}. By Lemma \ref{lem5.10} and Lemma \ref{lem5.11}, we can take a $g_3\in R'_V$ such that $g_3g_2W_2=\bbf (e_{i(5,1)}+v_2)\oplus\bbf (e_{i(5,2)}+v_3)$ where $(v_2,v_3)$ is one of
\begin{align*}
& (0,0),\quad (0,e_{i(2,2)}),\quad (0,e_{i(8,1)}),\quad (0,e_{i(10,1)}),\quad (e_{i(2,2)},e_{i(2,3)}),\quad (e_{i(2,2)},e_{i(8,1)}), \\
& (e_{i(2,2)},e_{i(10,1)}),\quad (e_{i(8,1)},e_{i(8,2)}),\quad (e_{i(8,1)},e_{i(10,1)})\quad\mbox{or}\quad (e_{i(10,1)},e_{i(10,2)}).
\end{align*}
Put $W'=g_3g_2g_1W=\bbf v_1\oplus \bbf (e_{i(5,1)}+v_2)\oplus\bbf (e_{i(5,2)}+v_3)$.  Then $Q(W')$ contains a subgroup $H$ of the form in (D.6.2.1) and hence $|Q(W')\backslash M(W')|<\infty$.

(D.6.3) Case of $\dim p_{I_{(5)}}W\le 1$. We can take an $h=h_{(5)}(A)$ with some $A\in{\rm GL}_{b_5}(\bbf)$ such that $p_{I_{(5)}}(hW)\subset \bbf e_{i(5,1)}$. Hence
$$hW\subset U_\#=U^+_{(2,3,8,10)}\oplus \bbf e_{i(5,1)}.$$
Let $Q_\#$ be the restriction of $\{g\in R_V\mid gU_\#=U_\#\}$ to $U_\#$. With respect to the basis
$$e_{i(2,1)},\ldots,e_{i(2,b_2)},e_{i(3,1)},\ldots,e_{i(3,b_3)},e_{i(8,1)},\ldots,e_{i(8,b_8)},e_{i(10,1)},\ldots,e_{i(10,b_{10})},e_{i(5,1)}$$
of $U_\#$, elements of $Q_\#$ are represented by matrices
$$\bp A & 0 & * & * & 0 \\
0 & B & * & * & * \\
0 & 0 & C & * & 0 \\
0 & 0 & 0 & D & 0 \\
0 & 0 & 0 & 0 & \lambda \ep$$
with $A\in{\rm GL}_{b_2}(\bbf),\ B\in{\rm GL}_{b_3}(\bbf),\ C\in{\rm GL}_{b_8}(\bbf),\ D\in{\rm GL}_{b_{10}}(\bbf)$ and $\lambda\in\bbf^\times$ by Lemma \ref{lem5.10} and Lemma \ref{lem5.11}. Hence it follows from Lemma \ref{lem9.8} in the appendix (case of $\alpha_1=b_2,\ \alpha_2=b_3,\ \alpha_3=b_8+b_{10}$ and $\alpha_4=\alpha_5=0$) that $|Q_\#\backslash M(U_\#)|<\infty$. Since $hW\subset U_\#$, we have $|Q(hW)\backslash M(hW)|<\infty$.

Thus we have completed the proof of Proposition \ref{prop4.2'}.

\subsection{Proof of Proposition \ref{prop4.3'}}

Suppose $\dim U_+=4$ and $V=V(b_1,\ldots,b_{15})$ as in Theorem \ref{th5.9}. We have only to show $|R_V\backslash M(U_+)|<\infty$. Since $\dim U_+=4$, we have
\begin{equation*}
b_1+b_2+b_3+b_7+2b_8+b_{10}+b_5+b_9+2b_{12}+b_{13}+b_{15}+b_6=4. %\label{eq7.2'}
\end{equation*}
Put $\widetilde{\mathcal{I}}=\{i=2,3,7,8,10,5,9,12,13,15\mid b_i\ne 0\}$. By the reduction of full flags $U_+^1\subset U_+^2\subset U_+^3\subset U_+$ into $\widetilde{S}_1\subset \widetilde{S}_2\subset \widetilde{S}_3\subset \widetilde{U}$ in Section \ref{sec6.1}, we have only to show that $|Q\backslash M(\widetilde{U})|<\infty$.

(A) Case of $b_{15}=4$ ($\widetilde{U}=U_+$). Since $Q\cong {\rm Sp}_4(\bbf)$ by Lemma \ref{lem5.11'}, we have $|Q\backslash M(U_+)|<\infty$ by Corollary 1.11 in \cite{M2}.

(B) Case of $b_{15}=2$. If $\widetilde{\mathcal{I}}=\{15\}$, then $Q\cong {\rm SL}_2(\bbf)$ and hence $|Q\backslash M(\widetilde{U})|<\infty$.

If $\widetilde{\mathcal{I}}=\{j,15\}$ with some $j=2,3,7,8,10,5,9,12,13$, then $Q$ contains a subgroup isomorphic to ${\rm GL}_{b_j}(\bbf)\times {\rm SL}_2(\bbf)$ by Lemma \ref{lem5.10} and Lemma \ref{lem5.11'}. Hence we have $|Q\backslash M(\widetilde{U})|<\infty$ by Lemma \ref{lem12.13} in the appendix.

So we have only to consider the case of $\widetilde{\mathcal{I}}=\{j,k,15\}$ with some $j,k=2,3,7,10,5,$ $9,13\ (j<k)$. Note that $b_j=b_k=1$ and $\widetilde{U}=U_+$. Let us identify $U_+$ with $\bbf^4$ with respect to the basis $e_{i(j,1)},e_{i(k,1)},e_{i(15,1)},e_{i(15,2)}$ of $U_+$. 

Suppose $(j,k)\ne (10,13)$. Since $j=2,3,5,7$ or $9$, $Q$ contains a subgroup consisting of matrices
$$\bp \lambda & 0 & * \\ 0 & \mu & 0 \\ 0 & 0 & B \ep$$
with $\lambda,\mu\in\bbf^\times$ and $B\in{\rm SL}_2(\bbf)$ by Lemma \ref{lem5.10}, Lemma \ref{lem5.11'} and Lemma \ref{lem5.11}. By Lemma \ref{lem9.8} (case of $\alpha_3=1,\ \alpha_4=2$ and $\alpha_1=\alpha_2=\alpha_5=0$), we have $|Q\backslash M(U_+)|<\infty$.

Suppose $\widetilde{\mathcal{I}}=\{10,13,15\}$. Then $Q$ consists of matrices
\begin{equation}
\bp \lambda & * & 0 \\ 0 & \mu & 0 \\ 0 & 0 & B \ep \label{eq7.3''}
\end{equation}
with $\lambda,\mu\in\bbf^\times$ and $B\in{\rm SL}_2(\bbf)$. Let $Q'$ be the subgroup of ${\rm GL}(U_+)$ consisting of matrices (\ref{eq7.3''}) with $\lambda,\mu\in\bbf^\times$ and $B\in{\rm GL}_2(\bbf)$. Since we assume $|\bbf^\times/(\bbf^\times)^2|<\infty$, we have $|Q'/ZQ|=|\bbf^\times/(\bbf^\times)^2|<\infty$ where $Z=\{\lambda I_4\mid \lambda\in\bbf^\times\}$ is the center of ${\rm GL}(U_+)$. By Corollary \ref{cor10.9} (case of $\alpha_1=2,\ \alpha_2=1$ and $\alpha_3=\alpha_4=\alpha_5=0$), we have $|Q'\backslash M(U_+)|<\infty$. Hence $|Q\backslash M(U_+)|<\infty$.

(C) Case of $b_{15}=0$. We have $0\le |\widetilde{\mathcal{I}}|\le 4$.

(C.0) Case of $|\widetilde{\mathcal{I}}|=0$ is clear since $\widetilde{U}=\{0\}$.

(C.1) Case of $|\widetilde{\mathcal{I}}|=1$ is also clear since $Q\cong {\rm GL}(\widetilde{U})$.

(C.2) Case of $|\widetilde{\mathcal{I}}|=2$. If $\widetilde{\mathcal{I}}=\{j,k\}$, then $Q$ contains a subgroup of ${\rm GL}(\widetilde{U})$ isomorphic to ${\rm GL}_{b_j}(\bbf) \times {\rm GL}_{b_k}(\bbf)$. So we have $|Q\backslash M(\widetilde{U})|<\infty$.

(C.3) Case of $|\widetilde{\mathcal{I}}|=3$. Suppose $\widetilde{\mathcal{I}}=\{j,k,\ell\}$. Consider the decomposition $\widetilde{\mathcal{I}}=\mathcal{I}_1\sqcup \mathcal{I}_2$ with
$$\mathcal{I}_1\subset \{2,7,8,10,13\}\mand \mathcal{I}_2\subset \{3,5,9,12\}.$$

Suppose $\widetilde{\mathcal{I}}= \mathcal{I}_s$ with $s=1$ or $2$ and $j<k<\ell$. Then $Q$ is the parabolic subgroup of ${\rm GL}(\widetilde{U})$ stabilizing the flag
$$U_{(j)}^+\subset U_{(j,k)}^+$$
by Lemma \ref{lem5.10} and Lemma \ref{lem5.11}. Hence we have $|Q\backslash M(\widetilde{U})|<\infty$.

Suppose $\mathcal{I}_1=\{j,k\}$ and $\mathcal{I}_2=\{\ell\}$ with $j<k$. Then $Q$ contains a subgroup $Q'$ of ${\rm GL}(\widetilde{U})$ consisting of elements represented by
$$\bp A & * & 0 \\ 0 & B & 0 \\ 0 & 0 & C \ep$$
with $A\in{\rm GL}_{b_j}(\bbf),\ B\in {\rm GL}_{b_k}(\bbf)$ and $C\in {\rm GL}_{b_\ell}(\bbf)$ by Lemma \ref{lem5.10} and Lemma \ref{lem5.11}. If $b_\ell=1$, then we have $|Q'\backslash M(\widetilde{U})|<\infty$ by Lemma \ref{lem5.1}. If $b_j=b_k=1$ and $b_\ell=2$, then we have $|Q'\backslash M(\widetilde{U})|<\infty$ by Corollary \ref{cor10.9} (case of $\alpha_1=2,\ \alpha_2=1$ and $\alpha_3=\alpha_4=\alpha_5=0$).

We can also prove the case of $\mathcal{I}_1=\{j\}$ and $\mathcal{I}_2=\{k,\ell\}$ in the same way.

(C.4) Case of $|\widetilde{\mathcal{I}}|=4$. Put $\widetilde{\mathcal{I}}=\{j,k,\ell,m\}$. Then $j,k,\ell,m\ne 8,12,\ b_j=b_k=b_\ell=b_m=1$ and $\widetilde{U}=U_+$.

Suppose $\{j,k,\ell\}\subset \{2,7,10,13\}$ with $j<k<\ell$. Then $Q$ contains a subgroup $Q'$ of ${\rm GL}(U_+)$ consisting of elements represented by
$$\bp A & 0 \\ 0 & \lambda \ep$$
with $A\in \mcb_3$ and $\lambda\in\bbf^\times$ by Lemma \ref{lem5.10} and Lemma \ref{lem5.11}. Hence we have $|Q\backslash M(U_+)|<\infty$ by Lemma \ref{lem5.1}.

We can also prove the case of $\{j,k,\ell\}\subset \{3,5,9,13\}$ in the same way.

So we have only to consider the case of $\{j,k\}\subset \{2,7,10\}$ and $\{\ell,m\}\subset \{3,5,9\}$ with $j<k$ and $\ell<m$. If $\ell=3$, then $Q$ contains a subgroup $Q'$ consisting of elements represented by matrices
\begin{equation}
\bp a & * & 0 & 0 \\ 0 & b & 0 & 0 \\ 0 & * & c & * \\ 0 & 0 & 0 & d \ep \label{eq7.4}
\end{equation}
($a,b,c,d\in\bbf^\times$) with respect to the basis $e_{i(j,1)},e_{i(k,1)},e_{i(\ell,1)},e_{i(m,1)}$ of $U_+$ by Lemma \ref{lem5.10} and Lemma \ref{lem5.11}. So we have $|Q\backslash M(U_+)|<\infty$ by Lemma \ref{lem9.8} (case of $\alpha_1=\alpha_2=\alpha_3=1$ and $\alpha_4=\alpha_5=0$).

So we may assume $\{\ell,m\}=\{5,9\}$. In this case, $Q$ contains a subgroup $Q'$ consisting of elements represented by matrices (\ref{eq7.4}) with respect to the basis $e_{i(\ell,1)},e_{i(m,1)},$ $e_{i(j,1)},e_{i(k,1)}$ of $U_+$ by Lemma \ref{lem5.10} and Lemma \ref{lem5.11}. So we have $|Q\backslash M(U_+)|<\infty$ by Lemma \ref{lem9.8}.

Thus we have completed the proof of Proposition \ref{prop4.3'}.

\section{Proof of Proposition \ref{prop4.2}} \label{sec7}

In each $G$-orbit of the triple flag variety $\mct_{(n),(\beta),(n)}$, we can take a triple flag $(U_+,U_-,V)$ such that
\begin{align*}
U_+ =\bbf e_1\oplus\cdots\oplus \bbf e_n,\quad U_- & =\bbf e_1\oplus\cdots\oplus e_{a_0}\oplus \bbf e_{n+1}\oplus\cdots\oplus \bbf e_{n+\beta-a_0} \\
\mbox{and that }V & =\bigoplus_{j=1}^{15} V_{(j)}
\end{align*}
by Proposition \ref{prop5.1} and Theorem \ref{th5.9}. Since $\dim U_+=n$, we have $a_-=a_2=0$ and hence
\begin{equation}
b_k=0\quad\mbox{for }k=4,7,9,11,12,13,14 \label{eq7.1}
\end{equation}
by (\ref{eq5.3}) and (\ref{eq5.5}).

\subsection{Proof of the case of ${\bf b}=(\beta)=(n-1)$} \label{sec7.1}

Since $\dim U_-=n-1$, we have $a_+=1$ and hence
\begin{equation}
b_3+b_8+b_{10}=1 \label{eq7.2}
\end{equation}
by (\ref{eq5.2}). By the reduction in Section \ref{sec6.1}, we have only to show $|Q\backslash M(\widetilde{U})|<\infty$. Let us represent elements in $Q$ by matrices in ${\rm GL}_{\dim \widetilde{U}}(\bbf)$ with respect to the basis
$e_{b_1+1},\ldots,e_{n-b_6-b_8}$ of $\widetilde{U}$.

(A) Case of $b_3=1$. $\widetilde{U}$ is decomposed as
$$\widetilde{U}=U_{(2)}^+\oplus U_{(3)}^+\oplus U_{(5)}^+\oplus U_{(15)}^+$$
by (\ref{eq7.1}) and (\ref{eq7.2}). % where
%\begin{align*}
%U_{(2)}^+ & =\bbf e_{b_1+1}\oplus\cdots\oplus \bbf e_{a_0}, \quad
%U_{(3)}^+ =\bbf e_{a_0+1}, \\
%U_{(5)}^+ & =\bbf e_{a_0+2}\oplus\cdots\oplus \bbf e_{a_0+b_5+1}
%\mand U_{(15)}^+ =\bbf e_{a_0+b_5+2}\oplus\cdots\oplus \bbf e_{n-b_6}.
%\end{align*}
By Lemma \ref{lem5.10}, Lemma \ref{lem5.11'} and Lemma \ref{lem5.11} (with Figure \ref{fig6.2}), the group $Q$ consists of matrices of the form
$$\bp A & 0 & 0 & * \\
0 & \lambda & * & * \\
0 & 0 & B & * \\
0 & 0 & 0 & C \ep$$
with $\lambda\in\bbf^\times,\ A\in {\rm GL}_{b_2}(\bbf),\ B\in {\rm GL}_{b_5}(\bbf)$ and $C\in {\rm Sp}'_{b_{15}}(\bbf)$. By Corollary \ref{cor10.9} (case of $\alpha_3=\alpha_4=\alpha_5=0$), the subgroup of ${\rm GL}(U_{(2,3,5)}^+)$ consisting of matrices of the form
$$\bp A & 0 & 0 \\ 0 & \lambda & * \\ 0 & 0 & B \ep$$
has a finite number of orbits on $M(\bbf^{b_2+b_3+b_5})$. On the other hand,
$$|{\rm Sp}'_{b_{15}}(\bbf)\backslash M(\bbf^{b_{15}})|<\infty$$
by Corollary 1.11 in \cite{M2}. Hence we have $|Q\backslash M(\widetilde{U})|<\infty$ by Corollary \ref{cor8.7}.

(B) Case of $b_{10}=1$. $\widetilde{U}$ is decomposed as
$$\widetilde{U}=U_{(2)}^+\oplus U_{(10)}^+\oplus U_{(5)}^+\oplus U_{(15)}^+$$
by (\ref{eq7.1}) and (\ref{eq7.2}).
By Lemma \ref{lem5.10}, Lemma \ref{lem5.11'} and Lemma \ref{lem5.11}, the group $Q$ consists of matrices of the form
\begin{equation}
\bp A & * & 0 & * & \\
0 & \lambda & 0 & 0 \\
0 & 0 & B & * \\
0 & 0 & 0 & C \ep
\label{eq8.3''}
\end{equation}
with $A\in {\rm GL}_{b_2}(\bbf),\ \lambda\in\bbf^\times,\ B\in {\rm GL}_{b_5}(\bbf)$ and $C\in {\rm Sp}'_{b_{15}}(\bbf)$. By Lemma \ref{lem9.8} (case of $\alpha_3=\alpha_5=0$), we have $|Q\backslash M(\widetilde{U})|<\infty$.

(C) Case of $b_8=1$. $\widetilde{U}$ is decomposed as
$$\widetilde{U}=U_{(2)}^+\oplus U_{(8)}^+\oplus U_{(5)}^+\oplus U_{(15)}^+$$
by (\ref{eq7.1}) and (\ref{eq7.2}).
By Lemma \ref{lem5.10}, Lemma \ref{lem5.11'} and Lemma \ref{lem5.11}, the group $Q$ consists of matrices of the form (\ref{eq8.3''}). Hence $|Q\backslash M(\widetilde{U})|<\infty$.

Thus the triple flag variety $\mct_{(1^n),(n-1),(n)}$ is of finite type.

\subsection{Proof of the case of ${\bf b}=(1,n-1)$}

In each $G$-orbit of the triple flag variety $\mct_{(n),(n),(n)}$, we can take a triple flag $(U_+,U_-,V)$ such that
\begin{align*}
U_+ =\bbf e_1\oplus\cdots\oplus \bbf e_n,\quad U_- & =\bbf e_1\oplus\cdots\oplus e_{a_0}\oplus \bbf e_{n+1}\oplus\cdots\oplus \bbf e_{2n-a_0} \\
\mbox{and that }V & =\bigoplus_{j=1}^{15} V_{(j)}
\end{align*}
by Proposition \ref{prop5.1} and Theorem \ref{th5.9}. Since $\dim U_+=\dim U_-=n$, we have $a_+=a_-=a_2=0$ and hence
$$b_j=0\quad\mbox{for }j=3,4,7,8,9,10,11,12,13,14.$$
(See also \cite{M2} Theorem 1.20.) The space $U_-$ is decomposed as
$$U_-=U_{(1)}^-\oplus U_{(2)}^-\oplus U_{(6)}^-\oplus U_{(15)}^-\oplus U_{(5)}^-$$
where
\begin{align*}
U_{(1)}^- & =U_{(1)}^+= \bbf e_1\oplus\cdots\oplus \bbf e_{b_1},\quad U_{(2)}^- =U_{(2)}^+= \bbf e_{b_1+1}\oplus\cdots\oplus \bbf e_{a_0}, \\
U_{(6)}^- & = \bbf e_{n+1}\oplus\cdots\oplus \bbf e_{n+b_6},\quad U_{(15)}^- = \bbf e_{n+b_6+1}\oplus\cdots\oplus \bbf e_{n+b_6+b_{15}} \\
\mand U_{(5)}^- & = \bbf e_{n+b_6+b_{15}+1}\oplus\cdots\oplus \bbf e_{2n-a_0}.
\end{align*}

First we describe $R_V$-orbits of one-dimensional subspaces of $U_-$ as follows.

\begin{lemma} Let $U_-^1$ be a one-dimensional subspace of $U_-$. Then we can take a $g\in R_V$ such that $gU_-^1$ is one of the following six spaces.
$$\bbf e_1,\quad \bbf e_{b_1+1},\quad \bbf e_{n+1},\quad \bbf (e_{b_1+1}+e_{n+1}),\quad \bbf e_{n+b_6+1},\quad \bbf e_{n+b_6+b_{15}+1}$$
\label{lem7.1}
\end{lemma}

\begin{proof} (A) Suppose that $U_-^1\subset U_{(1)}^-$. Then we can take an $h=h_{(1)}(A)\in R_V$ with some $A\in {\rm GL}_{b_1}(\bbf)$ such that $hU_-^1=\bbf e_1$ by Lemma \ref{lem5.10}.

(B) Suppose that $U_-^1=\bbf(v_1+v_2)$ with $v_1=\sum_{i=1}^{b_1} \lambda_ie_i\in U_{(1)}^-$ and $v_2\in U_{(2)}^--\{0\}$. Then we can take an $h=h_{(2)}(B)\in R_V$ with some $B\in {\rm GL}_{b_2}(\bbf)$ such that $hv_2=e_{b_1+1}$ by Lemma \ref{lem5.10}. Hence
$$h(v_1+v_2)=v_1+e_{b_1+1}.$$
Define an element $g\in R_V$ by
$$g=\prod_{i=1}^{b_1} g_{i,b_1+1}(-\lambda_i)$$
with the elements $g_{i,b_1+1}(-\lambda_i)\in R_V$ given in Lemma \ref{lem5.11}. Then we have
$$gh(v_1+v_2)=g(v_1+e_{b_1+1})=e_{b_1+1}.$$

(C) Suppose that $U_-^1=\bbf(v_1+v_3)$ with $v_1=\sum_{i=1}^{b_1} \lambda_ie_i\in U_{(1)}^-$ and $v_3\in U_{(6)}^--\{0\}$. Then we can take an $h=h_{(6)}(C)\in R_V$ with some $C\in {\rm GL}_{b_6}(\bbf)$ such that $hv_3=e_{n+1}$ by Lemma \ref{lem5.10}. Hence
$$h(v_1+v_3)=v_1+e_{n+1}.$$

For $i=1,\ldots,b_1$ and $\lambda\in\bbf$, we can define an element $g=g_{i,n+1}(\lambda)\in R_V$ by
$$ge_{n+1}=e_{n+1}+\lambda e_i,\quad ge_{2n+1-i}=e_{2n+1-i}-\lambda e_n$$
and $ge_\ell=e_\ell$ for $\ell\in \{1,\ldots,2n\}-\{n+1,2n+1-i\}$. Put $g_1=\prod_{i=1}^{b_1} g_{i,n+1}(-\lambda_i)$. Then we have
$$g_1h(v_1+v_3)=g_1(v_1+e_{n+1})=e_{n+1}.$$

(D) Suppose that $U_-^1=\bbf(v_1+v_2+v_3)$ with $v_1=\sum_{i=1}^{b_1} \lambda_ie_i\in U_{(1)}^-,\ v_2\in U_{(2)}^--\{0\}$ and $v_3\in U_{(6)}^--\{0\}$. Then we can take $h_1=h_{(2)}(B)\ (B\in {\rm GL}_{b_2}(\bbf))$ and $h_2=h_{(6)}(C)\ (C\in {\rm GL}_{b_6}(\bbf))$ in $R_V$ such that $h_1v_2=e_{b_1+1}$ and that $h_2v_3=e_{n+1}$ by Lemma \ref{lem5.10}. Hence
$$h_1h_2(v_1+v_2+v_3)=v_1+e_{b_1+1}+e_{n+1}.$$
Take the element $g=\prod_{i=1}^{b_1} g_{i,b_1+1}(-\lambda_i)$ given in (B). Then we have
$$gh_1h_2(v_1+v_2+v_3)=g(v_1+e_{b_1+1}+e_{n+1})=e_{b_1+1}+e_{n+1}.$$

(E) Suppose that $U_-^1=\bbf(v_1+v_2+v_3+v_4)$ with
\begin{align*}
v_1 & =\sum_{i=1}^{b_1} \lambda_ie_i\in U_{(1)}^-,\quad v_2=\sum_{i=b_1+1}^{a_0} \lambda_ie_i\in U_{(2)}^-, \\
v_3 & =\sum_{i=n+1}^{n+b_6} \lambda_ie_i\in U_{(6)}^-\mand v_4\in U_{(15)}^--\{0\}.
\end{align*}
Then we can take an $h=h_{(15)}(D)\in R_V$ with some $D\in{\rm Sp}_{b_{15}}(\bbf)$ such that $hv_4=e_{n+b_6+1}$ by Lemma \ref{lem5.11'}. Hence
$$h(v_1+v_2+v_3+v_4)=v_1+v_2+v_3+e_{n+b_6+1}.$$

For $i\in \{1,\ldots,a_0\}\sqcup\{n+1,\ldots,n+b_6\}$ and $\lambda\in\bbf$, we can define an element $g=g_{i,n+b_6+1}(\lambda)\in R_V$ by
$$ge_{n+b_6+1}=e_{n+b_6+1}+\lambda e_i,\quad ge_{2n+1-i}=e_{2n+1-i}-\lambda e_{n-b_6}$$
and $ge_\ell=e_\ell$ for $\ell\in \{1,\ldots,2n\}-\{n+b_6+1,2n+1-i\}$. Put
$$g_1=\prod_{i\in \{1,\ldots,a_0\}\sqcup\{n+1,\ldots,n+b_6\}} g_{i,n+b_6+1}(-\lambda_i).$$
Then we have
$$g_1h(v_1+v_2+v_3+v_4)=g_1(v_1+v_2+v_3+e_{n+b_6+1})=e_{n+b_6+1}.$$

(F) Suppose that $U_-^1=\bbf(v_1+v_2+v_3+v_4+v_5)$ with
\begin{align*}
v_1 & =\sum_{i=1}^{b_1} \lambda_ie_i\in U_{(1)}^-,\quad v_2=\sum_{i=b_1+1}^{a_0} \lambda_ie_i\in U_{(2)}^-, \quad
v_3 =\sum_{i=n+1}^{n+b_6} \lambda_ie_i\in U_{(6)}^-, \\
v_4 & =\sum_{i=n+b_6+1}^{n+b_6+b_{15}} \lambda_ie_i\in U_{(15)}^- \mand v_5\in U_{(5)}^--\{0\}.
\end{align*}
Then we can take an $h=h_{(5)}(E)\in R_V$ with some $E\in {\rm GL}_{b_5}(\bbf)$ such that $hv_5=e_m$ by Lemma \ref{lem5.10} where $m=n+b_6+b_{15}+1$. Hence
$$h(v_1+v_2+v_3+v_4+v_5)=v_1+v_2+v_3+v_4+e_m.$$
For $i\in \{1,\ldots,a_0\}\sqcup\{n+1,\ldots,n+b_6+b_{15}\}$ and $\lambda\in\bbf$, we can define an element $g=g_{i,m}(\lambda)\in R_V$ by
$$ge_m=e_m+\lambda e_i,\quad ge_{2n+1-i}=e_{2n+1-i}-\lambda e_{2n+1-m}$$
and $ge_\ell=e_\ell$ for $\ell\in \{1,\ldots,2n\}-\{m,2n+1-i\}$. Put
$$g_1=\prod_{i\in \{1,\ldots,a_0\}\sqcup\{n+1,\ldots,n+b_6+b_{15}\}} g_{i,m}(-\lambda_i).$$
Then we have
$$g_1h(v_1+v_2+v_3+v_4+v_5)=g_1(v_1+v_2+v_3+v_4+e_m)=e_m.$$
\end{proof}

By Lemma \ref{lem7.1}, we may assume $U_-^1$ is one of the six spaces given in this lemma. We have only to consider the restriction $Q'$ of
$$R'_V=\{g\in R_V\mid gU_-^1=U_-^1\}$$
to $U_+$ and show that $Q'$ has a finite number of orbits on the full flag variety $M$ of ${\rm GL}(U_+)$.

(A) Case of $U_-^1=\bbf e_1$. The group $Q'$ consists of matrices of the form
$$\bp \lambda & * & * & * & * & * \\
0 & A & * & * & * & * \\
0 & 0 & B & 0 & * & * \\
0 & 0 & 0 & C & * & * \\
0 & 0 & 0 & 0 & D & * \\
0 & 0 & 0 & 0 & 0 & E\ep$$
with $\lambda\in\bbf^\times,\ A\in{\rm GL}_{b_1-1}(\bbf),\ B\in{\rm GL}_{b_2}(\bbf),\ C\in{\rm GL}_{b_5}(\bbf),\ D\in{\rm Sp}'_{b_{15}}(\bbf)$ and $E\in{\rm GL}_{b_6}(\bbf)$. In the appendix of \cite{M2}, it is shown that the subgroup of ${\rm GL}_{b_2+b_5}(\bbf)$ consisting of matrices
$$\bp B & 0 \\ 0 & C \ep$$
has a finite number of orbits on $M(\bbf^{b_2+b_5})$. It is also shown that $|{\rm Sp}'_{b_{15}}(\bbf)\backslash M(\bbf^{b_{15}})|<\infty$. Hence $|Q'\backslash M|<\infty$ by Corollary \ref{cor8.7}.

(B) Case of $U_-^1=\bbf e_{b_1+1}$. The group $Q'$ consists of matrices of the form
$$\bp A & * & * & * & * & * \\
0 & \lambda & * & 0 & * & * \\
0 & 0 & B & 0 & * & * \\
0 & 0 & 0 & C & * & * \\
0 & 0 & 0 & 0 & D & * \\
0 & 0 & 0 & 0 & 0 & E\ep$$
with $\lambda\in\bbf^\times,\ A\in{\rm GL}_{b_1}(\bbf),\ B\in{\rm GL}_{b_2-1}(\bbf),\ C\in{\rm GL}_{b_5}(\bbf),\ D\in{\rm Sp}'_{b_{15}}(\bbf)$ and $E\in{\rm GL}_{b_6}(\bbf)$. By Corollary \ref{cor10.9} (case of $\alpha_1=b_5,\ \alpha_2=b_2-1$ and $\alpha_3=\alpha_4=\alpha_5=0$), the subgroup of ${\rm GL}_{b_2+b_6}(\bbf)$ consisting of matrices
$$\bp \lambda & * & 0 \\ 0 & B & 0 \\ 0 & 0 & C \ep$$
has a finite number of orbits on $M(\bbf^{b_2+b_5})$. On the other hand, $|{\rm Sp}'_{b_{15}}(\bbf)\backslash M(\bbf^{b_{15}})|<\infty$. Hence $|Q'\backslash M|<\infty$ by Corollary \ref{cor8.7}.

(C) Case of $U_-^1=\bbf e_{n+1}$. The group $Q'$ consists of matrices of the form
$$\bp A & * & * & * & * & * \\
0 & B & 0 & * & * & * \\
0 & 0 & C & * & * & * \\
0 & 0 & 0 & D & * & * \\
0 & 0 & 0 & 0 & E & * \\
0 & 0 & 0 & 0 & 0 & \lambda \ep$$
with $\lambda\in\bbf^\times,\ A\in{\rm GL}_{b_1}(\bbf),\ B\in{\rm GL}_{b_2}(\bbf),\ C\in{\rm GL}_{b_5}(\bbf),\ D\in{\rm Sp}'_{b_{15}}(\bbf)$ and $E\in{\rm GL}_{b_6-1}(\bbf)$. Since the subgroup of ${\rm GL}_{b_2+b_5}(\bbf)$ consisting of matrices
$$\bp B & 0 \\ 0 & C \ep$$
has a finite number of orbits on $M(\bbf^{b_2+b_5})$, and since $|{\rm Sp}'_{b_{15}}(\bbf)\backslash M(\bbf^{b_{15}})|<\infty$, we have $|Q'\backslash M|<\infty$ by Corollary \ref{cor8.7}.

(D) Case of $U_-^1=\bbf (e_{b_1+1}+e_{n+1})$. The group $Q'$ consists of matrices of the form
$$\bp A & * & * & * & * & * & * \\
0 & \lambda & * & 0 & * & * & * \\
0 & 0 & B & 0 & * & * & * \\
0 & 0 & 0 & C & * & * & * \\
0 & 0 & 0 & 0 & D & * & * \\
0 & 0 & 0 & 0 & 0 & E & * \\
0 & 0 & 0 & 0 & 0 & 0 & \lambda^{-1} \ep$$
with $\lambda\in\bbf^\times,\ A\in{\rm GL}_{b_1}(\bbf),\ B\in{\rm GL}_{b_2-1}(\bbf),\ C\in{\rm GL}_{b_5}(\bbf),\ D\in{\rm Sp}'_{b_{15}}(\bbf)$ and $E\in{\rm GL}_{b_6-1}(\bbf)$.

The subgroup of ${\rm GL}_{b_2+b_5}(\bbf)\times\bbf^\times$ consisting of elements
$$\left(\bp \lambda & * & 0 \\ 0 & B & 0 \\ 0 & 0 & C \ep,\ \lambda^{-1}\right)$$
has a finite number of orbits on $M(\bbf^{b_2+b_5})\times M(\bbf)$ by Corollary \ref{cor10.9}. (Note that $M(\bbf^{b_2+b_5})\times M(\bbf)\cong M(\bbf^{b_2+b_5})$ since $M(\bbf)$ consists of one point.) We also have $|{\rm Sp}'_{b_{15}}\backslash M(\bbf^{b_{15}})|<\infty$. Hence $|Q'\backslash M|<\infty$ by Corollary \ref{cor8.7}.

(E) Case of $U_-^1=\bbf e_{n+b_6+1}$. The group $Q'$ consists of matrices of the form
$$\bp A & * & * & * & * \\
0 & B & 0 & * & * \\
0 & 0 & C & * & * \\
0 & 0 & 0 & D & X \\
0 & 0 & 0 & 0 & E \ep$$
with $A\in{\rm GL}_{b_1}(\bbf),\ B\in{\rm GL}_{b_2}(\bbf),\ C\in{\rm GL}_{b_5}(\bbf),\ D\in Q_{b_{15}},\ E\in{\rm GL}_{b_6}(\bbf)$ and a $b_{15}\times b_6$ matrix $X=\{x_{i,j}\}$ such that
$$x_{b_{15},1}=\cdots=x_{b_{15},b_6}=0$$
where
$Q_{b_{15}}=\{g\in {\rm Sp}'_{b_{15}}(\bbf)\mid g\bbf e_1=\bbf e_1\}$.
The subgroup of ${\rm GL}_{b_2+b_5}(\bbf)$ consisting of matrices
$$\bp B & 0 \\ 0 & C \ep$$
has a finite number of orbits on $M(\bbf^{b_2+b_5})$. On the other hand, the subgroup of ${\rm GL}_{b_{15}+b_6}(\bbf)$ consisting of the matrices
$$\bp D & X \\ 0 & E \ep$$
has a finite number of orbits on $M(\bbf^{b_{15}+b_6})$ by Lemma \ref{lem9.12} in the appendix. So we have $|Q'\backslash M|<\infty$ by Corollary \ref{cor8.7}.

(F) Case of $U_-^1=\bbf e_{n+b_6+b_{15}+1}$. The group $Q'$ consists of matrices of the form
$$\bp A & * & * & * & * & * \\
0 & B & 0 & 0 & * & * \\
0 & 0 & C & * & * & * \\
0 & 0 & 0 & \lambda & 0 & 0 \\
0 & 0 & 0 & 0 & D & * \\
0 & 0 & 0 & 0 & 0 & E \ep$$
with $A\in{\rm GL}_{b_1}(\bbf),\ B\in{\rm GL}_{b_2}(\bbf),\ C\in{\rm GL}_{b_5-1}(\bbf),\ D\in {\rm Sp}'_{b_{15}}(\bbf)$ and $E\in{\rm GL}_{b_6}(\bbf)$. By Lemma \ref{lem9.8} (case of $\alpha_1=b_2,\ \alpha_2=b_5-1,\ \alpha_3=0,\ \alpha_4=b_{15}$ and $\alpha_5=b_6$), the subgroup of ${\rm GL}_{n-b_1}(\bbf)$ consisting of matrices
$$\bp B & 0 & 0 & * & * \\
0 & C & * & * & * \\
0 & 0 & \lambda & 0 & 0 \\
0 & 0 & 0 & D & * \\
0 & 0 & 0 & 0 & E \ep$$
has a finite number of orbits on $M(\bbf^{n-b_1})$. So we have $|Q'\backslash M|<\infty$ by Corollary \ref{cor8.7}.

Thus the triple flag variety $\mct_{(1^n),(1,n-1),(n)}$ is of finite type.

\section{Proof of Proposition \ref{prop4.3}}

Take a triple flag $(U_+,U_-,V)\in \mct_{(n),(\beta),(n)}$ as in Section \ref{sec7}. Since $\dim U_+=n$, we have $a_-=a_2=0$ and hence
$$b_k=0\quad\mbox{for }k=4,7,9,11,12,13,14.$$
Moreover we have
$$\beta=\dim U_-=a_0+a_-+a_1=b_1+b_2+b_5+b_6+b_8+b_{15}.$$

(i) When $\beta\le 2$ or $\beta=3$ and $|\bbf^\times/(\bbf^\times)^2|<\infty$, we will prove $|Q\backslash M(\widetilde{U})|<\infty$ using the reduction in Section \ref{sec6.1}.

(ii) We will also prove $|R_V\backslash M_{\gamma_1,\gamma_2,\gamma_3}(U_+)|<\infty$ without assumption on $\bbf$.

Let us represent $Q$ by matrices with respect to the basis
$$e_{b_1+1},\ldots,e_{n-b_6-b_8}$$
of $\widetilde{U}$. Here we note that the group $Q$ consists of restrictions of elements in
$$\{g\in R_V\mid g\widetilde{U}=\widetilde{U},\ gU_{(\overline{8})}^{+,j}=U_{(\overline{8})}^{+,j}\mbox{ for }j=1,\ldots,b_8\mbox{ and $g$ acts trivially on }U_{\overline{6},1}^+\}$$
to $\widetilde{U}$ since we consider full flags in $U_+$.

(A) Case of $b_5=b_{15}=0$. The space $\widetilde{U}$ is decomposed as
$$\widetilde{U}=U_{(2)}^+\oplus U_{(3)}^+\oplus U_{(8)}^+\oplus U_{(10)}^+.$$
(Note that $0\le b_2+b_8\le \beta$.) The group $Q$ consists of matrices
$$\bp A & 0 & * & * \\
0 & B & * & * \\
0 & 0 & C & * \\
0 & 0 & 0 & D \ep$$
with $A\in{\rm GL}_{b_2}(\bbf),\ B\in{\rm GL}_{b_3}(\bbf),\ C\in \mcb_{b_8}$ and $D\in{\rm GL}_{b_{10}}(\bbf)$ by Lemma \ref{lem5.10} and Lemma \ref{lem5.11} (with Figure \ref{fig6.2}).
Since the subgroup of ${\rm GL}_{b_2+b_3}(\bbf)$ consisting of matrices
$$\bp A & 0 \\ 0 & B \ep$$
has a finite number of orbits on $M(\bbf^{b_2+b_3})$, we have $|Q\backslash M(\widetilde{U})|<\infty$ by Corollary \ref{cor8.7}.

(B) Case of $b_5=1$ and $b_{15}=0$. The space $\widetilde{U}$ is decomposed as
$$\widetilde{U}=U_{(2)}^+\oplus U_{(3)}^+\oplus U_{(8)}^+\oplus U_{(10)}^+\oplus U_{(5)}^+.$$
(Note that $0\le b_2+b_8\le \beta-1$.) The group $Q$ consists of matrices
$$\bp A & 0 & * & * & 0 \\
0 & B & * & * & * \\
0 & 0 & C & * & 0 \\
0 & 0 & 0 & D & 0 \\
0 & 0 & 0 & 0 & \lambda\ep.$$
with $\lambda\in\bbf^\times,\ A\in{\rm GL}_{b_2}(\bbf),\ B\in{\rm GL}_{b_3}(\bbf),\ C\in\mcb_{b_8}$ and $D\in{\rm GL}_{b_{10}}(\bbf)$ by Lemma \ref{lem5.10} and Lemma \ref{lem5.11}. By Lemma \ref{lem9.8} (case of $\alpha_1=b_2,\ \alpha_2=b_3,\ \alpha_3=b_8+b_{10}$ and $\alpha_4=\alpha_5=0$), we have $|Q\backslash M(\widetilde{U})|<\infty$.

(C) Case of $b_5=2$. We have $b_1+b_2+b_6+b_8=\beta-2$ and $b_{15}=0$. If $b_2=b_8=0$, then $\widetilde{U}$ is decomposed as
$$\widetilde{U}=U_{(3)}^+\oplus U_{(10)}^+\oplus U_{(5)}^+$$
and the group $Q$ consists of matrices
$$\bp A & * & * \\ 0 & B & 0 \\ 0 & 0 & C \ep$$
with $A\in{\rm GL}_{b_3}(\bbf),\ B\in{\rm GL}_{b_{10}}(\bbf)$ and $C\in{\rm GL}_2(\bbf)$. Since the subgroup of ${\rm GL}_{b_{10}+2}(\bbf)$ consisting of matrices
$$\bp B & 0 \\ 0 & C \ep$$
has a finite number of orbits on $M(\bbf^{b_{10}+2})$, we have $|Q\backslash M(\widetilde{U})|<\infty$ by Corollary \ref{cor8.7}.

If $b_2=1$, then $\widetilde{U}=U_+$ is decomposed as
$$U_+=U_{(2)}^+\oplus U_{(3)}^+\oplus U_{(10)}^+\oplus U_{(5)}^+$$
and the group $Q$ consists of matrices
$$\bp \lambda & 0 & * & 0 \\
0 & A & * & * \\
0 & 0 & B & 0 \\
0 & 0 & 0 & C \ep$$
with $\lambda\in\bbf^\times,\ A\in{\rm GL}_{b_3}(\bbf),\ B\in{\rm GL}_{b_{10}}(\bbf)$ and $C\in{\rm GL}_2(\bbf)$ by Lemma \ref{lem5.10} and Lemma \ref{lem5.11}. By Corollary \ref{cor10.9} (case of $\alpha_1=2,\ \alpha_2=b_{10},\ \alpha_3=b_3$ and $\alpha_4=\alpha_5=0$), we have $|Q\backslash M(U_+)|<\infty$.

If $b_8=1$, then $\widetilde{U}$ is decomposed as
$$\widetilde{U}=U_{(3)}^+\oplus U_{(8)}^+\oplus U_{(10)}^+\oplus U_{(5)}^+.$$
The group $Q$ consists of matrices
$$\bp A & * & * & * \\
0 & \lambda & * & 0 \\
0 & 0 & B & 0 \\
0 & 0 & 0 & C \ep$$
with $\lambda\in\bbf^\times,\ A\in{\rm GL}_{b_3}(\bbf),\ B\in{\rm GL}_{b_{10}}(\bbf)$ and $C\in{\rm GL}_2(\bbf)$ by Lemma \ref{lem5.10} and Lemma \ref{lem5.11}. The subgroup of ${\rm GL}_{b_{10}+3}(\bbf)$ consisting of matrices
$$\bp \lambda & * & 0 \\ 0 & B & 0 \\ 0 & 0 & C \ep$$
has a finite number of orbits on $M(\bbf^{b_{10}+3})$ by Corollary \ref{cor10.9} (case of $\alpha_1=2,\ \alpha_2=b_{10}$ and $\alpha_3=\alpha_4=\alpha_5=0$). So we have $|Q\backslash M(\widetilde{U})|<\infty$ by Corollary \ref{cor8.7}.

(D) Case of $b_5=3$. The space $\widetilde{U}=U_+$ is decomposed as
$$U_+=U_{(3)}^+\oplus U_{(10)}^+\oplus U_{(5)}^+$$
and the group $Q$ consists of matrices
$$\bp A & * & * \\ 0 & B & 0 \\ 0 & 0 & C \ep$$
with $A\in{\rm GL}_{b_3}(\bbf),\ B\in{\rm GL}_{b_{10}}(\bbf)$ and $C\in{\rm GL}_3(\bbf)$. Since the subgroup of ${\rm GL}_{b_{10}+3}(\bbf)$ consisting of matrices
$$\bp B & 0 \\ 0 & C \ep$$
has a finite number of orbits on $M(\bbf^{b_{10}+3})$, we have $|Q\backslash M(U_+)|<\infty$ by Corollary \ref{cor8.7}.

(E) Case of $b_5=b_8=0$ and $b_{15}=2$. The space $\widetilde{U}$ is decomposed as
$$\widetilde{U}=U_{(2)}^+\oplus U_{(3)}^+\oplus U_{(10)}^+\oplus U_{(15)}^+.$$
(Since $b_1+b_2+b_6=\beta-2$, we have $b_2=0$ or $1$.) The group $Q$ consists of matrices
$$\bp A & 0 & * & * \\
0 & B & * & * \\
0 & 0 & C & 0 \\
0 & 0 & 0 & D \ep$$
with $A\in{\rm GL}_{b_2}(\bbf),\ B\in{\rm GL}_{b_3}(\bbf),\ C\in{\rm GL}_{b_{10}}(\bbf)$ and $D\in{\rm Sp}'_2(\bbf)={\rm SL}_2(\bbf)$. The subgroups consisting of the matrices
$$\bp A & 0 \\ 0 & B \ep \mand \bp C & 0 \\ 0 & D  \ep$$
of ${\rm GL}_{b_2+b_3}(\bbf)$ and ${\rm GL}_{b_{10}+2}(\bbf)$ have finite numbers of orbits on $M(\bbf^{b_2+b_3})$ and $M(\bbf^{b_{10}+2})$ by \cite{M2} and Lemma \ref{lem12.13}, respectively. Hence we have $|Q\backslash M(\widetilde{U})|<\infty$ by Corollary \ref{cor8.7}.

(F) Case of $b_5=1$ and $b_{15}=2$. $\widetilde{U}=U_+$ is decomposed as
$$U_+=U_{(3)}^+\oplus U_{(10)}^+\oplus U_{(5)}^+\oplus U_{(15)}^+$$
and the group $Q$ consists of matrices
$$\bp A & * & * & * \\
0 & B & 0 & 0 \\
0 & 0 & \lambda & * \\
0 & 0 & 0 & C\ep$$
with $\lambda\in\bbf^\times,\ A\in{\rm GL}_{b_3}(\bbf),\ B\in{\rm GL}_{b_{10}}(\bbf)$ and $C\in{\rm SL}_2(\bbf)$ by Lemma \ref{lem5.10} and Lemma \ref{lem5.11}. Let $H$ denote the subgroup of ${\rm GL}_{b_{10}+3}(\bbf)$ consisting of matrices
$$\bp B & 0 & 0 \\ 0 & \lambda & * \\ 0 & 0 & C\ep$$
with $\lambda\in\bbf^\times,\ B\in{\rm GL}_{b_{10}}(\bbf)$ and $C\in{\rm SL}_2(\bbf)$.

(i) First assume $|\bbf^\times/(\bbf^\times)^2|<\infty$. By Corollary \ref{cor8.7}, we have only to show that $|H\backslash M(\bbf^{b_{10}+3})|<\infty$. Let $Z=\{\lambda I_{b_{10}+3}\mid \lambda\in\bbf^\times\}$ be the center of ${\rm GL}_{b_{10}+3}(\bbf)$. Then $ZH$ is a subgroup of
$$H'=\left\{\bp B & 0 & 0 \\ 0 & \lambda & * \\ 0 & 0 & C\ep \Bigm| \lambda\in\bbf^\times,\ B\in{\rm GL}_{b_{10}}(\bbf),\ C\in{\rm GL}_2(\bbf)\right\}$$
of index $|\bbf^\times/(\bbf^\times)^2|$. By Corollary \ref{cor10.9}, we have $|H'\backslash M(\bbf^{b_{10}+3})|<\infty$. Since $Z$ acts trivially on $M(\bbf^{b_{10}+3})$, we have $|H\backslash M(\bbf^{b_{10}+3})|<\infty$.

(ii) We have only to show that $|Q\backslash M_{\gamma_1,\gamma_2,\gamma_3}(\bbf^n)|<\infty$ when $\gamma_1+\gamma_2+\gamma_3<n$. Let $V_1\subset V_2\subset V_3$ be a flag in $\bbf^n$ such that $\dim V_i=\gamma_1+\cdots+\gamma_i$. Then we can take an element $g\in Q$ of the form
$$g=\bp A & * \\ 0 & I_{b_{10}+3} \ep$$
such that $gV_i=\bbf e_1\oplus\cdots\oplus \bbf e_{\ell_i} \oplus V'_i$ with some flag $V'_1\subset V'_2\subset V'_3$ in $U=U_{(10)}^+\oplus U_{(5)}^+\oplus U_{(15)}^+$ where $\ell_i=\dim(V_i\cap U^+_{(3)})$. By Lemma \ref{lem11.15} in the appendix, we have $|H\backslash M_{k_1,k_2,k_3}|<\infty$ for any $k_1,k_2,k_3$. Hence we have $|Q\backslash M_{\gamma_1,\gamma_2,\gamma_3}(\bbf^n)|<\infty$ for any $\gamma_1,\gamma_2,\gamma_3$.

(G) Case of $b_8=1$ and $b_{15}=2$. The space $\widetilde{U}$ is decomposed as
$$\widetilde{U}=U_{(3)}^+\oplus U_{(8)}^+\oplus U_{(10)}^+\oplus U_{(15)}^+.$$
With respect to the basis $e_1,\ldots,e_{n-1}$ of $\widetilde{U}$, elements of $Q$ are represented by
$$\bp A & * & * & * \\
0 & \lambda & * & 0 \\
0 & 0 & B & 0 \\
0 & 0 & 0 & C \ep$$
with $\lambda\in\bbf^\times,\ A\in{\rm GL}_{b_3}(\bbf),\ B\in{\rm GL}_{b_{10}}(\bbf)$ and $C\in{\rm SL}_2(\bbf)$.

(i) Suppose $|\bbf^\times/(\bbf^\times)^2|<\infty$. Then by the same reason as in (F), we have $|Q\backslash M(\widetilde{U})|<\infty$.

(ii) Let $V_1\subset V_2\subset V_3$ be a flag in $\mct_{\gamma_1,\gamma_2,\gamma_3}(U_+)$ with $\gamma_1+\gamma_2+\gamma_3<n$. Then we have only to show that $|R_V\backslash \mct_{\gamma_1,\gamma_2,\gamma_3}(U_+)|<\infty$. (Note that we cannot use $Q$ in this case.) By the same argument as in (F), we have only to prove the following lemma for the case of $b_3=0$.

\begin{lemma} Suppose that $n=2b_8+b_{10}+b_{15}$ with $b_8=1$ and $b_{15}=2$. Then we have $|R_V\backslash \mct_{\gamma_1,\gamma_2,\gamma_3}(U_+)|<\infty$ for any $\gamma_1,\gamma_2,\gamma_3$ with $\gamma_1+\gamma_2+\gamma_3<n$.
\end{lemma}

\begin{proof} The space $U_+$ is decomposed as
$$U_+=U_{(8)}^+\oplus U_{(15)}^+\oplus U_{(10)}^+\oplus U_{(\overline{8})}^+$$
with $U_{(8)}^+=\bbf e_1,\ U_{(15)}^+=\bbf e_2\oplus\bbf e_3,\ U_{(10)}^+=\bbf e_4\oplus\cdots\oplus\bbf e_{n-1}$ and $U_{(\overline{8})}^+=\bbf e_n$. (Here we exchanged the order of $U_{(10)}^+$ and $U_{(15)}^+$ for the sake of convenience.)

Write $U'_+=\bbf e_2\oplus\cdots\oplus \bbf e_n=U_{(15)}^+\oplus U_{(10)}^+\oplus U_{(\overline{8})}^+$. Let $V_1\subset V_2\subset V_3$ be a flag in $\mct_{\gamma_1,\gamma_2,\gamma_3}(U_+)$. Then there exists an $i\in\{0,1,2,3\}$ such that
$$V_i\not\ni e_1\mand V_{i+1}\ni e_1.$$
(We write $V_0=\{0\}$ and $V_4=U_+$.) Take an $n-1$ dimensional subspace $U$ of $U_+$ such that $U\supset V_i$ and $U\not\ni e_1$. Then we have
$$V_j=\bbf e_1\oplus (V_j\cap U)$$
for $j>i$. Take a $g\in R_V$ such that $gU=U'_+$. Then we have
$$gV_j=\begin{cases} gV_j\cap U'_+ & \text{ for $j\le i$,} \\ \bbf e_1\oplus (gV_j\cap U'_+) & \text{for $j>i$.} \end{cases}$$
So we have only to consider the flag $V'_1\subset V'_2\subset V'_3$ in $M_{k_1,k_2,k_3}(U'_+)$ where
$$V'_j=gV_j\cap U'_+.$$
We can write the restriction $H$ of $\{g\in R_V\mid gU'_+=U'_+\}$ to $U'_+$ as
$$H=\left\{\bp A & 0 & 0 \\ 0 & B & * \\ 0 & 0 & \lambda \ep \Bigm| \lambda\in\bbf^\times,\ A\in{\rm SL}_2(\bbf),\ B\in{\rm GL}_{n-4}(\bbf)\right\}.$$
Let $H'=H'(hV'_3)$ denote the restriction of $\{g\in H\mid ghV'_3=hV'_3\}$ to $hV'_3$ for $h\in H$. Write $V'_{3,1}=V'_3\cap U_{(15)}^+$ and $V'_{3,2}=V'_3\cap (U_{(10)}^+\oplus U_{(\overline{8})}^+)$. Let $\pi_2$ denote the projection $U'_+\to U_{(10)}^+\oplus U_{(\overline{8})}^+$.

(A) First suppose $\dim V'_{3,1}=2$. Then we have $V'_3=U_{(15)}^+\oplus V'_{3,2}$.

(A.1) If $V'_{3,2}\subset U_{(10)}^+$, then we can take an $h\in H$ such that
$$hV'_3=U_{(15)}^+\oplus \bbf e_4\oplus\cdots\oplus \bbf e_{k+1}$$
where $k=\dim V'_3=k_1+k_2+k_3$. Elements of $H'=H'(hV'_3)$ are represented by matrices
$$\bp A & 0 \\ 0 & B \ep$$
with $A\in{\rm SL}_2(\bbf)$ and $B\in{\rm GL}_{k-2}(\bbf)$. By Lemma \ref{lem12.13}, we have $|H'\backslash M(hV'_3)|<\infty$ and hence $|H'\backslash M_{k_1,k_2}(hV'_3)|<\infty$.

(A.2) If $V'_{3,2}\not\subset U_{(10)}^+$, then we can take an $h\in H$ such that
$$hV'_3=U_{(15)}^+\oplus \bbf e_4\oplus\cdots\oplus \bbf e_k\oplus\bbf e_n.$$
Elements of $H'=H'(hV'_3)$ are represented by matrices
$$\bp A & 0 & 0 \\ 0 & B & * \\ 0 & 0 & \lambda \ep$$
with $\lambda\in\bbf^\times,\ A\in{\rm SL}_2(\bbf)$ and $B\in{\rm GL}_{k-3}(\bbf)$. So we have $|H'\backslash M_{k_1,k_2}(hV'_3)|<\infty$ by Corollary \ref{cor11.15}.

(B) Next suppose $\dim V'_{3,1}=1$. By the action of $H$, we may assume
$V'_{3,1}=\bbf e_2$.

(B.1) Case of $\dim V'_{3,2}=k-1$ and $\pi_2(V'_3)\subset U_{(10)}^+$. We can take an $h\in H$ such that
$$hV'_3=\bbf e_2\oplus \bbf e_4\oplus\cdots\oplus \bbf e_{k+2}.$$
Elements of $H'$ are represented by matrices
$$\bp a & 0 \\ 0 & B \ep$$
with $a\in\bbf^\times$ and $B\in{\rm GL}_{k-1}(\bbf)$. So we have $|H'\backslash M(hV'_3)|<\infty$.

(B.2) Case of $\dim V'_{3,2}=k-1$ and $\pi_2(V'_3)\not\subset U_{(10)}^+$. We can take an $h\in H$ such that
$$hV'_3=\bbf e_2\oplus \bbf e_4\oplus\cdots\oplus \bbf e_{k+1}\oplus \bbf e_n.$$
Elements of $H'$ are represented by matrices
$$\bp a & 0 & 0 \\ 0 & B & * \\ 0 & 0 & \lambda \ep$$
with $a,\lambda\in\bbf^\times$ and $B\in{\rm GL}_{k-2}(\bbf)$. So we have $|H'\backslash M(hV'_3)|<\infty$ by Lemma \ref{lem5.1}.

(B.3) Case of $\dim V'_{3,2}=k-2$ and $\pi_2(V'_3)\subset U_{(10)}^+$. We can take an $h\in H$ such that
$$hV'_3=\bbf e_2\oplus \bbf e_4\oplus\cdots\oplus \bbf e_{k+1}\oplus \bbf(e_3+e_{k+2}).$$
With respect to the basis $e_2,e_4,\ldots,e_{k+1},e_3+e_{k+2}$ of $hV'_3$, elements of $H'$ are represented by matrices
$$\bp a & 0 & * \\ 0 & B & * \\ 0 & 0 & a^{-1} \ep$$
with $a\in\bbf^\times$ and $B\in{\rm GL}_{k-2}(\bbf)$. Since the subgroup of ${\rm GL}_{k-1}(\bbf)\times \bbf^\times$ consisting of elements
$$\left(\bp a & 0 \\ 0 & B \ep,\ a^{-1}\right)$$
has a finite number of orbits on $M(\bbf^{k-1})\times M(\bbf)$, we have $|H'\backslash M(hV'_3)|<\infty$ by Corollary \ref{cor8.7}.

(B.4) Case of $\dim V'_{3,2}=k-2,\ V'_{3,2}\subset U_{(10)}^+$ and $\pi_2(V'_3)\not\subset U_{(10)}^+$. We can take an $h\in H$ such that
$$hV'_3=\bbf e_2\oplus \bbf e_4\oplus\cdots\oplus \bbf e_{k+1}\oplus \bbf(e_3+e_n).$$
With respect to the basis $e_2,e_4,\ldots,e_{k+1},e_3+e_n$ of $hV'_3$, elements of $H'$ are represented by matrices
$$\bp a & 0 & * \\ 0 & B & * \\ 0 & 0 & a^{-1} \ep$$
with $a\in\bbf^\times$ and $B\in{\rm GL}_{k-2}(\bbf)$. As in (B.3), we have $|H'\backslash M(hV'_3)|<\infty$.

(B.5) Case of $\dim V'_{3,2}=k-2$ and $V'_{3,2}\not\subset U_{(10)}^+$. We can take an $h\in H$ such that
$$hV'_3=\bbf e_2\oplus \bbf e_4\oplus\cdots\oplus \bbf e_k\oplus \bbf e_n\oplus \bbf(e_3+e_{k+1}).$$
With respect to the basis
\begin{equation}
e_2,e_4,\ldots,e_k,e_n,e_3+e_{k+1} \label{eq8.1'}
\end{equation}
of $hV'_3$, elements of $H'$ are represented by matrices
$$\bp a & 0 & 0 & * \\ 0 & B & * & * \\ 0 & 0 & \lambda & 0 \\ 0 & 0 & 0 & a^{-1} \ep$$
with $a,\lambda\in\bbf^\times$ and $B\in{\rm GL}_{k-3}(\bbf)$. By Lemma \ref{lem5.11} (iv), we can take an element $g_1\in R_V$ such that
$$g_1e_n=e_n+\mu e_2,\quad g_1e_3=e_3-\mu e_1$$
with $\mu\in\bbf$ and that $g_1e_j=e_j$ for $j\ne 3,n$. By Lemma \ref{lem5.11} (i), we can take an element $g_2\in R_V$ such that
$$g_2e_{k+1}=e_{k+1}+\mu e_1$$
and that $g_2e_j=e_j$ for $j\ne k+1$. The product $g=g_2g_1\in R_V$ stabilizes $hV'_3$ and $g|_{hV'_3}$ is represented by the matrix
$$\bp 1 & 0 & \mu & 0 \\ 0 & I_{k-3} & 0 & 0 \\ 0 & 0 & 1 & 0 \\ 0 & 0 & 0 & 1 \ep$$
with respect to the basis (\ref{eq8.1'}). (Here we note that $g$ does not stabilize $U'_+$ if $\mu\ne 0$.) Hence the restriction of $\{g\in R_V\mid ghV'_3=hV'_3\}$ to $hV'_3$ contains the subgroup $H''$ consisting of matrices
$$\bp a & 0 & * & * \\ 0 & B & * & * \\ 0 & 0 & \lambda & 0 \\ 0 & 0 & 0 & a^{-1} \ep.$$
We can see that the subgroup of ${\rm GL}_{k-2}(\bbf)\times{\rm GL}_2(\bbf)$ consisting of elements
$$\left(\bp a & 0 \\ 0 & B \ep, \bp \lambda & 0 \\ 0 & a^{-1} \ep\right)$$
has a finite number of orbits on $M(\bbf^{k-2})\times M(\bbf^2)$. So we have $|H''\backslash M(hV'_3)|<\infty$ by Corollary \ref{cor8.7}.

(C) Finally suppose $\dim V'_{3,1}=0$. Put $V'_4=V'_3\oplus \bbf e_2$ and let $Q'$ be the restriction of $\{g\in R_V\mid gV'_4=V'_4\}$ to $V'_4$. Then we have $|Q'\backslash M(V'_4)|<\infty$ by the arguments in (B). So we have $|Q''\backslash M(V'_3)|<\infty$ where $Q''$ is the restriction of $Q'$ to $V'_3$.
\end{proof}

\section{Proof of Proposition \ref{prop4.3''}}

In each $G$-orbit of the triple flag variety $\mct_{(n),(2),(n)}$, we can take a triple flag $(U_+,U_-,V)$ such that
\begin{align*}
U_+ =\bbf e_1\oplus\cdots\oplus \bbf e_n,\quad U_- & =\begin{cases} \bbf e_1\oplus\bbf e_2 & \text{if $a_0=2$,} \\
\bbf e_1\oplus \bbf e_{n+1} & \text{if $a_0=1$,} \\
\bbf e_{n+1}\oplus \bbf e_{n+2} & \text{if $a_0=0$} \end{cases} 
\end{align*}
and that $V=\bigoplus_{j=1}^{15} V_{(j)}$ by Proposition \ref{prop5.1} and Theorem \ref{th5.9}. Since $\dim U_+=n$, we have $a_-=a_2=0$ and hence
$$b_j=0\quad\mbox{for }j=4,7,9,11,12,13,14.$$
On the other hand, since $\dim U_-=2$, we have
$$b_1+b_2+b_5+b_6+b_8+b_{15}=2.$$
For each case, we will first describe $R_V$-orbits of one-dimensional subspaces of $U_-$. Then for each representative $U_-^1$ of these orbits, we will compute the restriction $Q'$ of $R'_V=\{g\in R_V\mid gU_-^1=U_-^1\}$ to $U_+$. We have only to show that $Q'$ has a finite number of orbits on the full flag variety $M=M(U_+)$. Note that the restriction $R_V(U_+)$ of $R_V$ to $U_+$ has a finite number of orbits on $M$ by Proposition \ref{prop4.3} (i). So we have only to consider the cases of $Q'\subsetneqq R_V(U_+)$.

(A) Case of $b_1+b_2=2$. The space $U_+$ is decomposed as
$$U_+=U_{(1)}^+\oplus U_{(2)}^+\oplus U_{(3)}^+\oplus U_{(10)}^+$$
with $U_{(1)}^+\oplus U_{(2)}^+=U_-=\bbf e_1\oplus\bbf e_2,\ U_{(3)}^+=\bbf e_3\oplus\cdots\oplus \bbf e_{b_3+2}$ and $U_{(10)}^+=\bbf e_{b_3+3}\oplus\cdots\oplus \bbf e_n$.

(A.1) Case of $b_1=2$. By Lemma \ref{lem5.10}, we may assume $U_-^1=\bbf e_1$. By Lemma \ref{lem5.10} and Lemma \ref{lem5.11} (with Figure \ref{fig6.1}), $Q'$ consists of matrices
$$\bp \lambda & * & * & * \\
0 & \mu & * & * \\
0 & 0 & A & * \\
0 & 0 & 0 & B \ep$$
with $\lambda,\mu\in\bbf^\times,\ A\in{\rm GL}_{b_3}(\bbf)$ and $B\in{\rm GL}_{b_{10}}(\bbf)$. So we have $|Q'\backslash M|<\infty$ by the Bruhat decomposition.

(A.2) Case of $b_2=2$. By Lemma \ref{lem5.10}, we may assume $U_-^1=\bbf e_1$. By Lemma \ref{lem5.10} and Lemma \ref{lem5.11}, $Q'$ consists of matrices
$$\bp \lambda & * & 0 & * \\
0 & \mu & 0 & * \\
0 & 0 & A & * \\
0 & 0 & 0 & B \ep$$
with $\lambda,\mu\in\bbf^\times,\ A\in{\rm GL}_{b_3}(\bbf)$ and $B\in{\rm GL}_{b_{10}}(\bbf)$. Since the subgroup of ${\rm GL}_{b_3+2}(\bbf)$ consisting of matrices
$$\bp \lambda & * & 0 \\ 0 & \mu & 0 \\ 0 & 0 & A \ep$$
has a finite number of orbits on the full flag variety of ${\rm GL}_{b_3+2}(\bbf)$ by Corollary \ref{cor10.9} (case of $\alpha_1=b_{10},\ \alpha_2=1$ and $\alpha_3=\alpha_4=\alpha_5=0$), we have $|Q'\backslash M|<\infty$ by Corollary \ref{cor8.7}.

(A.3) Case of $b_1=b_2=1$. We have $U_{(1)}^+=\bbf e_1$ and $U_{(2)}^+=\bbf e_2$. If $U_-^1=\bbf e_1$, then we have $Q'=R_V(U_+)$. So we may assume $U_-^1\ne \bbf e_1$ and hence
$$U_-^1=\bbf(e_2+\lambda e_1)$$
with some $\lambda\in\bbf$. Define an element $g\in R_V$ by
$$ge_2=e_2+\lambda e_1,\quad ge_{2n}=e_{2n}-\lambda e_{2n-1}$$
and $ge_k=e_k$ for all $k\ne 2,2n$. Then we have $g^{-1}U_-^1=\bbf e_2$. So we may assume
$$U_-^1=\bbf e_2.$$
For this $U_-^1$, the group $Q'$ consists of matrices
$$\bp \lambda & 0 & * & * \\
0 & \mu & 0 & * \\
0 & 0 & A & * \\
0 & 0 & 0 & B \ep$$
with $\lambda,\mu\in\bbf^\times,\ A\in{\rm GL}_{b_3}(\bbf)$ and $B\in{\rm GL}_{b_{10}}(\bbf)$. By Lemma \ref{lem5.1}, the subgroup of ${\rm GL}_{b_3+2}(\bbf)$ consisting of matrices
$$\bp \lambda & 0 & * \\ 0 & \mu & 0 \\ 0 & 0 & A \ep$$
has a finite number of orbits on $M(\bbf^{b_3+2})$. So we have $|Q'\backslash M|<\infty$ by Corollary \ref{cor8.7}.

(B) Case of $b_5+b_6=2$. $U_+=U_{(3)}^+\oplus U_{(10)}^+\oplus U_{(5)}^+\oplus U_{(\overline{6})}^+$
with $U_{(3)}^+=\bbf e_1\oplus\cdots\oplus \bbf e_{b_3},\ U_{(10)}^+=\bbf e_{b_3+1}\oplus\cdots\oplus \bbf e_{n-2}$ and $U_{(5)}^+\oplus U_{(\overline{6})}^+=\bbf e_{n-1}\oplus\bbf e_n$. $U_-=U_{(\overline{6})}^-\oplus U_{(5)}^-=\bbf e_{n+1}\oplus\bbf e_{n+2}$.

(B.1) Case of $b_6=2$. Since $U_-=U_{(\overline{6})}^-=\bbf e_{n+1}\oplus\bbf e_{n+2}$, we may assume $U_-^1=\bbf e_{n+1}$ by Lemma \ref{lem5.10}. The group $Q'$ consists of matrices
$$\bp A & * & * & * \\
0 & B & * & * \\
0 & 0 & \lambda & * \\
0 & 0 & 0 & \mu \ep$$
with $\lambda,\mu\in\bbf^\times,\ A\in{\rm GL}_{b_3}(\bbf)$ and $B\in{\rm GL}_{b_{10}}(\bbf)$ by Lemma \ref{lem5.10} and Lemma \ref{lem5.11} . So we have $|Q'\backslash M|<\infty$ by the Bruhat decomposition.

(B.2) Case of $b_5=2$. Since $U_-=U_{(5)}^-=\bbf e_{n+1}\oplus\bbf e_{n+2}$, we may assume $U_-^1=\bbf e_{n+1}$ by Lemma \ref{lem5.10}. The group $Q'$ consists of matrices
$$\bp A & * & * & * \\
0 & B & 0 & 0 \\
0 & 0 & \lambda & * \\
0 & 0 & 0 & \mu \ep$$
with $\lambda,\mu\in\bbf^\times,\ A\in{\rm GL}_{b_3}(\bbf)$ and $B\in{\rm GL}_{b_{10}}(\bbf)$ by Lemma \ref{lem5.10} and Lemma \ref{lem5.11}. Since the subgroup of ${\rm GL}_{b_{10}+2}(\bbf)$ consisting of matrices
$$\bp B & 0 & 0 \\ 0 & \lambda & * \\ 0 & 0 & \mu \ep$$
has a finite number of orbits on the full flag variety of ${\rm GL}_{b_{10}+2}(\bbf)$ by Lemma \ref{lem9.8}, we have $|Q'\backslash M|<\infty$ by Corollary \ref{cor8.7}.

(B.3) Case of $b_5=b_6=1$. We have $U_-=U_{(\overline{6})}^-\oplus U_{(5)}^-$ with $U_{(\overline{6})}^-=\bbf e_{n+1}$ and $U_{(5)}^-=\bbf e_{n+2}$. If $U_-^1=\bbf e_{n+1}$, then we have $Q'=R_V(U_+)$. So we may assume $U_-^1\ne \bbf e_{n+1}$ and hence
$$U_-^1=\bbf(e_{n+2}+\lambda e_{n+1})$$
with some $\lambda\in\bbf$. Define an element $g\in R_V$ by
$$ge_{n+2}=e_{n+2}+\lambda e_{n+1},\quad ge_n=e_n-\lambda e_{n-1}$$
and $ge_k=e_k$ for all $k\ne n,n+2$. Then we have $g^{-1}U_-^1=\bbf e_{n+2}$. So we may assume
$$U_-^1=\bbf e_{n+2}.$$
For this $U_-^1$, the group $Q'$ consists of matrices
$$\bp A & * & * & * \\
0 & B & 0 & * \\
0 & 0 & \lambda & 0 \\
0 & 0 & 0 & \nu \ep$$
with $\lambda,\mu\in\bbf^\times,\ A\in{\rm GL}_{b_3}(\bbf)$ and $B\in{\rm GL}_{b_{10}}(\bbf)$. Since the subgroup of ${\rm GL}_{b_3+2}(\bbf)$ consisting of matrices
$$\bp B & 0 & * \\ 0 & \lambda & 0 \\ 0 & 0 & \mu\ep$$
has a finite number of orbits on $M(\bbf^{b_3+2})$ by Lemma \ref{lem5.1}, we have $|Q'\backslash M|<\infty$ by Corollary \ref{cor8.7}.

(C) Case of $b_8=2$. $U_+=U_{(3)}^+\oplus U_{(8)}^+\oplus U_{(10)}^+\oplus U_{(\overline{8})}^+$
with $U_{(3)}^+=\bbf e_1\oplus\cdots\oplus \bbf e_{b_3},\ U_{(8)}^+=\bbf e_{b_3+1}\oplus\bbf e_{b_3+2},\ U_{(10)}^+=\bbf e_{b_3+3}\oplus\cdots\oplus \bbf e_{n-2}$ and $U_{(\overline{8})}^+=\bbf e_{n-1}\oplus\bbf e_n$. By Lemma \ref{lem5.10}, we may assume $U_-^1=\bbf e_{n+1}$. The group $Q'$ consists of matrices
$$\bp A & * & * & * \\
0 & C & * & * \\
0 & 0 & B & * \\
0 & 0 & 0 & C^* \ep$$
with $A\in{\rm GL}_{b_3}(\bbf),\ B\in{\rm GL}_{b_{10}}(\bbf)$ and $C\in \mcb_2$
and $C^*=J_2{}^tC^{-1}J_2$. Consider the subgroup
$$H=\{(C,C^*)\mid C\in \mcb_2\}$$
of ${\rm GL}_2(\bbf)\times {\rm GL}_2(\bbf)$. Then the full flag variety $M(\bbf^2)\times M(\bbf^2)\cong P^1(\bbf)\times P^1(\bbf)$ is decomposed into five $H$-orbits.
Hence we have $|Q'\backslash M|<\infty$ by Corollary \ref{cor8.7}.

(D) Case of $b_{15}=2$. $U_+=U_{(3)}^+\oplus U_{(10)}^+\oplus U_{(15)}^+$
with $U_{(3)}^+=\bbf e_1\oplus\cdots\oplus \bbf e_{b_3},\ U_{(10)}^+=\bbf e_{b_3+1}\oplus\cdots\oplus \bbf e_{n-2}$ and $U_{(15)}^+=\bbf e_{n-1}\oplus\bbf e_n$. By Lemma \ref{lem5.11'}, we may assume $U_-^1=\bbf e_{n+1}$. The group $Q'$ consists of matrices
$$\bp A & * & * & * \\ 0 & B & 0 & 0 \\ 0 & 0 & \lambda & * \\ 0 & 0 & 0 & \lambda^{-1} \ep$$
with $A\in{\rm GL}_{b_3}(\bbf),\ B\in{\rm GL}_{b_{10}}(\bbf)$ and $\lambda\in\bbf^\times$.

(i) First we will show that
$$|\bbf^\times/(\bbf^\times)^2|<\infty \Longrightarrow |Q'\backslash M|<\infty.$$
Let $Q''$ denote the subgroup of ${\rm GL}_{b_{10}+2}(\bbf)$ consisting of matrices
$$\bp B & 0 & 0 \\ 0 & \lambda & * \\ 0 & 0 & \lambda^{-1} \ep$$
with $B\in{\rm GL}_{b_{10}}(\bbf)$ and $\lambda\in\bbf^\times$. Then we have only to show that $|Q''\backslash M(\bbf^{b_{10}+2})|<\infty$ by Corollary \ref{cor8.7}.

Let $Q'''$ denote the subgroup of ${\rm GL}_{b_{10}+2}(\bbf)$ consisting of matrices
$$\bp B & 0 & 0 \\ 0 & \mu_1 & * \\ 0 & 0 & \mu_2 \ep$$
with $B\in{\rm GL}_{b_{10}}(\bbf)$ and $\mu_1,\mu_2\in\bbf^\times$. Then we have $|Q'''\backslash M(\bbf^{b_{10}+2})|<\infty$ by Lemma \ref{lem9.8} (case of $\alpha_1=b_3,\ \alpha_2=1$ and $\alpha_3=\alpha_4=\alpha_5=0$).

Let $Z=\{\mu I_{b_{10}+2}\mid \mu\in\bbf^\times\}$ denote the center of ${\rm GL}_{b_{10}+2}(\bbf)$. Since $Z$ acts trivially on $M(\bbf^{b_{10}+2})$, we have
$$Q''\backslash M(\bbf^{b_{10}+2})\cong ZQ''\backslash M(\bbf^{b_{10}+2}).$$
So we may consider the subgroup $ZQ''$ of $Q'''$ consisting of matrices
$$\bp B & 0 & 0 \\ 0 & \mu_1 & * \\ 0 & 0 & \mu_2 \ep$$
with $B\in{\rm GL}_{b_{10}}(\bbf)$ and
$\mu_1,\mu_2\in\bbf^\times$ such that $\mu_1/\mu_2\in (\bbf^\times)^2$. Since $|Q'''/ZQ''|=|\bbf^\times/(\bbf^\times)^2|<\infty$, we have $|ZQ''\backslash M(\bbf^{b_{10}+2})|<\infty$.

(ii) Next we will show that $|Q'\backslash M_{\gamma_1}(\bbf^n)|<\infty$ without the assumption on $\bbf$ where $M_{\gamma_1}(\bbf^n)$ denote the Grassmann variety consisting of $\gamma_1$-dimensional subspaces in $\bbf^n$. Let $V$ be a $\gamma_1$-dimensional subspace in $\bbf^n$. Put $\gamma'_1=\dim (V\cap U_{(3)}^+)$. Then we can take an element $g\in Q'$ of the form
$$\bp A & * \\ 0 & I_{b_{10}+2} \ep$$
such that
$$gV=\bbf e_1\oplus\cdots\oplus \bbf e_{\gamma'_1}\oplus V'$$
with some subspace $V'$ of $U_{(10)}\oplus U_{(15)}$. Since $|Q''\backslash M_k(\bbf^{b_{10}+2})|<\infty$ for any $k$ by Lemma \ref{lem11.13} in the appendix, we have 
$|Q'\backslash M_{\gamma'_1}(\bbf^n)|<\infty$.

(E) Case of $b_1+b_2=b_5+b_6=1$. We have
$$U_+=U_{(i)}^+\oplus U_{(3)}^+\oplus U_{(10)}^+\oplus U_{(j)}^+\mand U_-=\bbf e_1\oplus \bbf e_{n+1}$$
with $i=1$ or $2$, $j=5$ or $\overline{6}$, $U_{(i)}^+=\bbf e_1,\ U_{(3)}^+=\bbf e_2\oplus\cdots\oplus \bbf e_{b_3+1},\ U_{(10)}^+=\bbf e_{b_3+2}\oplus\cdots\oplus \bbf e_{n-1}$ and $U_{(j)}^+=\bbf e_n$.

(E.1) Case of $(i,j)=(1,5),\ (2,5)$ or $(1,\overline{6})$. If $U_-^1=\bbf e_1$, then we have $Q'=R_V(U_+)$. So we may assume $U_-^1\ne \bbf e_1$ and hence
$$U_-^1=\bbf(e_{n+1}+\lambda e_1)$$
with some $\lambda\in\bbf$. Define an element $g_\lambda\in R_V$ by
$$g_\lambda e_{n+1}=e_{n+1}+\lambda e_1,\quad g_\lambda e_{2n}=e_{2n}-\lambda e_n$$
and $g_\lambda e_k=e_k$ for all $k\ne n+1,2n$. Then we have $g_\lambda^{-1}U_-^1=\bbf e_{n+1}$. So we may assume
$$U_-^1=\bbf e_{n+1}.$$

We can show $Q'=R_V(U_+)$ as follows. Let $g$ be an element of $R_V$. Then we have
$$g\bbf e_{n+1}=\bbf(e_{n+1}+\lambda e_1)$$
with some $\lambda\in\bbf$ since $gU_-=U_-$ and since $g\bbf e_1=\bbf e_1$. So we have $g_\lambda^{-1}g\bbf e_{n+1}=\bbf e_{n+1}$ and hence $g_\lambda^{-1}g|_{U_+}\in Q'$. Since $g_\lambda$ acts trivially on $U_+$, we have $g|_{U_+}\in Q'$. Hence $Q'=R_V(U_+)$.

(E.2) Case of $(i,j)=(2,\overline{6})$. If $U_-^1=\bbf e_1$ or $\bbf e_{n+1}$, then we have $Q'=R_V(U_+)$. So we may assume
$$U_-^1=\bbf(e_{n+1}+\lambda e_1)$$
with some $\lambda\in\bbf^\times$. Take an element $h\in R_V$ such that
$$he_1=\lambda^{-1}e_1,\quad he_{2n}=\lambda e_{2n}$$
and that $he_k=e_k$ for all $k\ne 1,2n$ by Lemma \ref{lem5.10}. Then we have $hU_-^1=\bbf(e_1+e_{n+1})$. So we may assume
$$U_-^1=\bbf(e_1+e_{n+1}).$$
For this $U_-^1$, the group $Q'$ consists of matrices
$$\bp \lambda & 0 & * & * \\
0 & A & * & * \\
0 & 0 & B & * \\
0 & 0 & 0 & \lambda^{-1} \ep$$
with $\lambda\in\bbf^\times,\ A\in{\rm GL}_{b_3}(\bbf)$ and $B\in{\rm GL}_{b_{10}}(\bbf)$. Since the subgroup of ${\rm GL}_{b_3+1}(\bbf)\times\bbf^\times$ consisting of elements
$$\left(\bp \lambda & 0 \\ 0 & A \ep,\ \lambda^{-1}\right)$$
has a finite number of orbits on $M(\bbf^{b_3+1})\times M(\bbf)$, we have $|Q'\backslash M|<\infty$ by Corollary \ref{cor8.7}.

(F) Case of $b_1=b_8=1$. We have
$$U_+=U_{(1)}^+\oplus U_{(3)}^+\oplus U_{(8)}^+\oplus U_{(10)}^+\oplus U_{(\overline{8})}^+\mand U_-=\bbf e_1\oplus \bbf e_{n+1}$$
with $U_{(1)}^+=\bbf e_1,\ U_{(3)}^+=\bbf e_2\oplus\cdots\oplus \bbf e_{b_3+1},\ U_{(8)}^+=\bbf e_{b_3+2},\ U_{(10)}^+=\bbf e_{b_3+3}\oplus\cdots\oplus \bbf e_{n-1}$ and $U_{(\overline{8})}^+=\bbf e_n$. If $U_-^1=\bbf e_1$, then we have $Q'=R_V(U_+)$. So we may assume $U_-^1\ne \bbf e_1$ and hence
$$U_-^1=\bbf(e_{n+1}+\lambda e_1).$$
By the same argument as in (E.1), we may assume $U_-^1=\bbf e_{n+1}$ and we have $Q'=R_V(U_+)$.

(G) Case of $b_2=b_8=1$. We have
$$U_+=U_{(2)}^+\oplus U_{(3)}^+\oplus U_{(8)}^+\oplus U_{(10)}^+\oplus U_{(\overline{8})}^+\mand U_-=\bbf e_1\oplus \bbf e_{n+1}$$
with $U_{(2)}^+=\bbf e_1,\ U_{(3)}^+=\bbf e_2\oplus\cdots\oplus \bbf e_{b_3+1},\ U_{(8)}^+=\bbf e_{b_3+2},\ U_{(10)}^+=\bbf e_{b_3+3}\oplus\cdots\oplus \bbf e_{n-1}$ and $U_{(\overline{8})}^+=\bbf e_n$.  If $U_-^1=\bbf e_1$, then we have $Q'=R_V(U_+)$. So we may assume $U_-^1\ne \bbf e_1$ and hence
$$U_-^1=\bbf(e_{n+1}+\lambda e_1)$$
with some $\lambda\in\bbf$. Define an element $g\in R_V$ by
$$ge_{n+1}=e_{n+1}+\lambda e_1,\quad ge_{b_3+2}=e_{b_3+2}-\lambda e_1,\quad ge_{2n}=e_{2n}+\lambda(e_{\overline{b_3+2}}-e_n)$$
and $ge_k=e_k$ for all $k\ne b_3+2,n+1,2n$. Then we have $g^{-1}U_-^1=\bbf e_{n+1}$. So we may assume
$$U_-^1=\bbf e_{n+1}.$$
For this $U_-^1$, the group $Q'$ consists of matrices
$$\bp \lambda & 0 & 0 & * & * \\
0 & A & * & * & * \\
0 & 0 & \mu & * & * \\
0 & 0 & 0 & B & * \\
0 & 0 & 0 & 0 & \mu^{-1} \ep$$
with $\lambda,\mu\in\bbf^\times,\ A\in{\rm GL}_{b_3}(\bbf)$ and $B\in{\rm GL}_{b_{10}}(\bbf)$. The subgroup of ${\rm GL}_{b_3+2}(\bbf)\times \bbf^\times$ consisting of elements
$$\left(\bp \lambda & 0 & 0 \\ 0 & A & * \\ 0 & 0 & \mu\ep,\ \mu^{-1}\right)$$
has a finite number of orbits on $M(\bbf^{b_3+2})\times M(\bbf)$ by Lemma \ref{lem5.1}. So we have $|Q'\backslash M|<\infty$ by Corollary \ref{cor8.7}.

(H) Case of $b_6=b_8=1$. We have
$$U_+=U_{(3)}^+\oplus U_{(8)}^+\oplus U_{(10)}^+\oplus U_{(\overline{8})}^+\oplus U_{(\overline{6})}^+\mand U_-=\bbf e_{n+1}\oplus \bbf e_{n+2}$$
with $U_{(3)}^+=\bbf e_1\oplus\cdots\oplus \bbf e_{b_3},\ U_{(8)}^+=\bbf e_{b_3+1},\ U_{(10)}^+=\bbf e_{b_3+2}\oplus\cdots\oplus \bbf e_{n-2},\ U_{(\overline{8})}^+=\bbf e_{n-1}$ and $U_{(\overline{6})}^+=\bbf e_n$. If $U_-^1=\bbf e_{n+1}$, then we have $Q'=R_V(U_+)$. So we may assume $U_-^1\ne \bbf e_{n+1}$ and hence
$$U_-^1=\bbf(e_{n+2}+\lambda e_{n+1})$$
with some $\lambda\in\bbf$. Define an element $g\in R_V$ by
$$ge_{n+2}=e_{n+2}+\lambda e_{n+1},\quad ge_n=e_n-\lambda e_{n-1}$$
and $ge_k=e_k$ for all $k\ne n,n+2$. Then we have $g^{-1}U_-^1=\bbf e_{n+2}$. So we may assume
$$U_-^1=\bbf e_{n+2}.$$
For this $U_-^1$, the group $Q'$ consists of matrices
$$\bp A & * & * & * & * \\
0 & \lambda & * & * & * \\
0 & 0 & B & * & * \\
0 & 0 & 0 & \lambda^{-1} & 0 \\
0 & 0 & 0 & 0 & \nu \ep$$
with $\lambda,\mu\in\bbf^\times,\ A\in{\rm GL}_{b_3}(\bbf)$ and $B\in{\rm GL}_{b_{10}}(\bbf)$. Since the subgroup of ${\rm GL}_1(\bbf)\times {\rm GL}_2(\bbf)$ consisting of elements
$$\left(\lambda,\ \bp \lambda^{-1} & 0 \\ 0 & \mu \ep\right)$$
has three orbits on $M(\bbf)\times M(\bbf^2)$, we have $|Q'\backslash M|<\infty$ by Corollary \ref{cor8.7}.

(I) Case of $b_5=b_8=1$. We have
$$U_+=U_{(3)}^+\oplus U_{(8)}^+\oplus U_{(10)}^+\oplus U_{(5)}^+\oplus U_{(\overline{8})}^+\mand U_-=\bbf e_{n+1}\oplus \bbf e_{n+2}$$
with $U_{(3)}^+=\bbf e_1\oplus\cdots\oplus \bbf e_{b_3},\ U_{(8)}^+=\bbf e_{b_3+1},\ U_{(10)}^+=\bbf e_{b_3+2}\oplus\cdots\oplus \bbf e_{n-2},\ U_{(5)}^+=\bbf e_{n-1}$ and $U_{(\overline{8})}^+=\bbf e_n$. If $U_-^1=\bbf e_{n+1}$, then we have $Q'=R_V(U_+)$. So we may assume $U_-^1\ne \bbf e_{n+1}$ and hence
$$U_-^1=\bbf(e_{n+2}+\lambda e_{n+1})$$
with some $\lambda\in\bbf$. Define an element $g\in R_V$ by
$$ge_{n+2}=e_{n+2}+\lambda e_{n+1},\quad ge_n=e_n-\lambda e_{n-1}$$
and $ge_k=e_k$ for all $k\ne n,n+2$. Then we have $g^{-1}U_-^1=\bbf e_{n+2}$. So we may assume
$$U_-^1=\bbf e_{n+2}.$$
With respect to the basis $e_1,\ldots,e_{n-2},e_n,e_{n-1}$ of $U_+$, elements of $Q'$ are represented by matrices
$$\bp A & * & * & * & * \\
0 & \lambda & * & * & 0 \\
0 & 0 & B & * & 0 \\
0 & 0 & 0 & \lambda^{-1} & 0 \\
0 & 0 & 0 & 0 & \mu \ep$$
with $\lambda,\mu\in\bbf^\times,\ A\in{\rm GL}_{b_3}(\bbf)$ and $B\in{\rm GL}_{b_{10}}(\bbf)$.

(i) We will first show that
$$|\bbf^\times/(\bbf^\times)^2|<\infty \Longrightarrow |Q'\backslash M|<\infty.$$
By Corollary \ref{cor8.7}, we have only to show that the subgroup $Q''$ of ${\rm GL}_{n-b_3}(\bbf)$ consisting of matrices
$$\bp \lambda & * & * & 0 \\
0 & B & * & 0 \\
0 & 0 & \lambda^{-1} & 0 \\
0 & 0 & 0 & \mu \ep$$
has a finite number of orbits on $M(\bbf^{n-b_3})$. Since the center $Z=\{\nu I_{n-b_3}\mid \nu\in\bbf^\times\}$ acts trivially on $M(\bbf^{n-b_3})$, we have only to consider the action of
$$ZQ''=\left\{\bp \lambda_1 & * & * & 0 \\
0 & B & * & 0 \\
0 & 0 & \lambda_2 & 0 \\
0 & 0 & 0 & \mu \ep\Bigm| \lambda_1,\lambda_2,\mu\in\bbf^\times,\ B\in{\rm GL}_{b_{10}}(\bbf),\ \frac{\lambda_1}{\lambda_2}\in (\bbf^\times)^2\right\}.$$
On the other hand, the group
$$Q'''=\left\{\bp \lambda_1 & * & * & 0 \\
0 & B & * & 0 \\
0 & 0 & \lambda_2 & 0 \\
0 & 0 & 0 & \mu \ep \Bigm|\lambda_1,\lambda_2,\mu\in\bbf^\times,\ B\in{\rm GL}_{b_{10}}(\bbf)\right\}$$
has a finite number of orbits on $M(\bbf^{n-b_3})$ by Lemma \ref{lem5.1}. Since the index of $ZQ''$ in $Q'''$ is $|\bbf^\times/(\bbf^\times)^2|$, we have
$$|ZQ''\backslash M(\bbf^{n-b_3})|<\infty.$$

(ii) Next we will show that $|Q'\backslash M_{\gamma_1}(\bbf^n)|<\infty$ without the assumption on $\bbf$. By the same argument as in (F), we have only to show $|Q''\backslash M_k(\bbf^{n-b_3})|<\infty$ for $k=0,\ldots,\gamma_1$. Let $H$ denote the subgroup of ${\rm GL}_{n-b_3-1}(\bbf)$ consisting of matrices
$$\bp \lambda & * & * \\
0 & B & * \\
0 & 0 & \lambda^{-1} \ep$$
with $B\in{\rm GL}_{b_{10}}(\bbf)$ and $\lambda\in\bbf^\times$. Then we have $|H\backslash M(\bbf^{n-b_3-1})|<\infty$ by the Bruhat decomposition. Applying Proposition 6.3 in \cite{M3}, we have
$$|Q''\backslash M_k(\bbf^{n-b_3})|<\infty.$$

Thus we have completed the proof of Proposition \ref{prop4.3''}.

\section{Proof of Proposition \ref{prop4.4}}

In each $G$-orbit of the triple flag variety $\mct_{(n),(1),(n)}$, we can take a triple flag $(U_+,U_-,V)$ such that
\begin{align*}
U_+ =\bbf e_1\oplus\cdots\oplus \bbf e_n,\quad U_- =\bbf e_1\mbox{ or }\bbf e_{n+1} \quad
\mbox{and that }V & =\bigoplus_{j=1}^{15} V_{(j)}
\end{align*}
by Proposition \ref{prop5.1} and Theorem \ref{th5.9}. Since $\dim U_+=n$, we have $a_-=a_2=0$ and hence
\begin{equation}
b_j=0\quad\mbox{for }j=4,7,9,11,12,13,14. \label{eq8.1}
\end{equation}
On the other hand, since $\dim U_-=a_0+a_1=1$, we have
$$b_1+b_2+b_5+b_6+b_8=1.$$
Note that
$$U_-=\begin{cases} \bbf e_1 & \text{if $b_1+b_2=a_0=1$,} \\ \bbf e_{n+1} & \text{if $b_5+b_6+b_8=a_1=1$.} \end{cases}$$
First we take representatives $U_+^1$ of $R_V$-orbits of $\beta$-dimensional subspaces of $U_+$. Then we show that the restriction $Q_V$ of $R'_V=\{g\in R_V\mid gU_+^1=U_+^1\}$ to $V$ has a finite number of orbits on the full flag variety $M(V)$ of ${\rm GL}(V)$.

(A) Case of $b_1=1$. The space $U_+$ is decomposed as
$$U_+=U_{(1)}^+\oplus U_{(3)}^+\oplus U_{(10)}^+$$
with $U_{(1)}^+=\bbf e_1,\ U_{(3)}^+=\bbf e_2\oplus\cdots\oplus \bbf e_{b_3+1}$ and $U_{(10)}^+=\bbf e_{b_3+2}\oplus\cdots\oplus \bbf e_n$.

First suppose that $U_+^1\supset U_{(1)}^+$. Then we can take a $g\in R_V$ such that
$$gU_+^1=U_{(1)}^+\oplus U_{(3),1}^+\oplus U_{(10),1}^+$$
where $U_{(3),1}^+=\bbf e_2\oplus\cdots\oplus \bbf e_{b'_3+1}$ and $U_{(10),1}^+=\bbf e_{b_3+2}\oplus\cdots\oplus \bbf e_{b_3+b'_{10}+1}$ with some $b'_3\le b_3$ and $b'_{10}\le b_{10}$, respectively, by Lemma \ref{lem5.10} and Lemma \ref{lem5.11}. So we may assume
$$U_+^1=U_{(1)}^+\oplus U_{(3),1}^+\oplus U_{(10),1}^+.$$
With respect to the basis
\begin{equation}
e_1,\ldots,e_{b_3+1},e_{n+1},\ldots,e_{\overline{b_3+2}} \label{eq8.2'}
\end{equation}
of $V$, elements of $Q_V$ are represented by matrices
$$\bp \lambda & * & * & * & * \\
0 & A & * & * & * \\
0 & 0 & B & * & * \\
0 & 0 & 0 & C & * \\
0 & 0 & 0 & 0 & D \ep$$
with $\lambda\in\bbf^\times,\ A\in{\rm GL}_{b'_3}(\bbf),\ B\in{\rm GL}_{b_3-b'_3}(\bbf),\ C\in{\rm GL}_{b_{10}-b'_{10}}(\bbf)$ and $D\in{\rm GL}_{b'_{10}}(\bbf)$ (c.f. Section 5 in \cite{M3}). So we have $|Q_V\backslash M(V)|<\infty$ by the Bruhat decomposition.

Next suppose that $U_+^1\not\supset U_{(1)}^+$. Then we can take a $g\in R_V$ such that
$$gU_+^1=U_{(3),1}^+\oplus U_{(10),1}^+$$
by Lemma \ref{lem5.10} and Lemma \ref{lem5.11}. So we may assume
$$U_+^1=U_{(3),1}^+\oplus U_{(10),1}^+.$$
With respect to the basis (\ref{eq8.2'}) of $V$, elements of $Q_V$ are represented by matrices
$$\bp \lambda & 0 & * & * & * \\
0 & A & * & * & * \\
0 & 0 & B & * & * \\
0 & 0 & 0 & C & * \\
0 & 0 & 0 & 0 & D \ep$$
with $\lambda\in\bbf^\times,\ A\in{\rm GL}_{b'_3}(\bbf),\ B\in{\rm GL}_{b_3-b'_3}(\bbf),\ C\in{\rm GL}_{b_{10}-b'_{10}}(\bbf)$ and $D\in{\rm GL}_{b'_{10}}(\bbf)$. Since the subgroup of ${\rm GL}_{b'_3+1}(\bbf)$ consisting of matrices
$$\bp \lambda & 0 \\ 0 & A \ep$$
has a finite number of orbits on $M(\bbf^{b'_3+1})$, we have $|Q_V\backslash M(V)|<\infty$ by Corollary \ref{cor8.7}.

(B) Case of $b_2=1$. The space $U_+$ is decomposed as
$$U_+=U_{(2)}^+\oplus U_{(3)}^+\oplus U_{(10)}^+$$
with $U_{(2)}^+=\bbf e_1,\ U_{(3)}^+=\bbf e_2\oplus\cdots\oplus \bbf e_{b_3+1}$ and $U_{(10)}^+=\bbf e_{b_3+2}\oplus\cdots\oplus \bbf e_n$.

First suppose that $U_+^1\supset U_{(2)}^+$. Then we may assume
$$U_+^1=U_{(2)}^+\oplus U_{(3),1}^+\oplus U_{(10),1}^+$$
by Lemma \ref{lem5.10} and Lemma \ref{lem5.11}. With respect to the basis
\begin{equation}
e_2,\ldots,e_{b_3+1},e_{n+1},\ldots,e_{\overline{b_3+2}},e_{2n} \label{eq8.3'}
\end{equation}
of $V$, elements of $Q_V$ are represented by matrices
\begin{equation}
\bp A & * & * & * & * \\
0 & B & * & * & * \\
0 & 0 & C & * & * \\
0 & 0 & 0 & D & * \\
0 & 0 & 0 & 0 & \lambda \ep \label{eq8.4}
\end{equation}
with $\lambda\in\bbf^\times,\ A\in{\rm GL}_{b'_3}(\bbf),\ B\in{\rm GL}_{b_3-b'_3}(\bbf),\ C\in{\rm GL}_{b_{10}-b'_{10}}(\bbf)$ and $D\in{\rm GL}_{b'_{10}}(\bbf)$. So we have $|Q_V\backslash M(V)|<\infty$ by the Bruhat decomposition.

Next suppose that $U_+^1\subset U_{(3)}^+\oplus U_{(10)}^+$. Then we 
may assume
$$U_+^1=U_{(3),1}^+\oplus U_{(10),1}^+$$
by Lemma \ref{lem5.10} and Lemma \ref{lem5.11}. With respect to the basis (\ref{eq8.3'}), elements of $Q_V$ are represented by the matrices (\ref{eq8.4}). So we have $|Q_V\backslash M(V)|<\infty$ by the Bruhat decomposition.

Finally suppose that $U_+^1\not\supset U_{(2)}^+$ and that $U_+^1\not\subset U_{(3)}^+\oplus U_{(10)}^+$. Then we can take a $g\in R_V$ such that
$$gU_+^1=U_{(3),1}^+\oplus \bbf (e_1+e_{b'_3+2})\oplus U_{(10),1}^+$$
by Lemma \ref{lem5.10} and Lemma \ref{lem5.11}. So we may assume
$$U_+^1=U_{(3),1}^+\oplus \bbf (e_1+e_{b'_3+2})\oplus U_{(10),1}^+.$$
With respect to the basis (\ref{eq8.3'}), elements of $Q_V$ are represented by the matrices
$$\bp A & * & * & * & * & * \\
0 & \lambda & * & * & * & * \\
0 & 0 & B & * & * & * \\
0 & 0 & 0 & C & * & * \\
0 & 0 & 0 & 0 & D & * \\
0 & 0 & 0 & 0 & 0 & \lambda^{-1} \ep$$
with $\lambda\in\bbf^\times,\ A\in{\rm GL}_{b'_3}(\bbf),\ B\in{\rm GL}_{b_3-b'_3-1}(\bbf),\ C\in{\rm GL}_{b_{10}-b'_{10}}(\bbf)$ and $D\in{\rm GL}_{b'_{10}}(\bbf)$. So we have $|Q_V\backslash M(V)|<\infty$ by the Bruhat decomposition.

(C) Case of $b_5=1$. The space $U_+$ is decomposed as
$$U_+=U_{(3)}^+\oplus U_{(10)}^+\oplus U_{(5)}^+$$
with $U_{(3)}^+=\bbf e_1\oplus\cdots\oplus \bbf e_{b_3},\ U_{(10)}^+=\bbf e_{b_3+1}\oplus\cdots\oplus \bbf e_{n-1}$ and $U_{(5)}^+=\bbf e_n$.

First suppose that $U_+^1\subset U_{(3)}^+\oplus U_{(10)}^+$. Then we may assume
$$U_+^1=U_{(3),1}^+\oplus U_{(10),1}^+$$
by Lemma \ref{lem5.10} and Lemma \ref{lem5.11}. With respect to the basis
\begin{equation}
e_1,\ldots,e_{b_3},e_n,e_{n+2},\ldots,e_{\overline{b_3+1}} \label{eq8.5}
\end{equation}
of $V$, elements of $Q_V$ are represented by matrices
$$\bp A & * & * & * & * \\
0 & B & * & * & * \\
0 & 0 & \lambda & 0 & 0 \\
0 & 0 & 0 & C & * \\
0 & 0 & 0 & 0 & D \ep$$
with $\lambda\in\bbf^\times,\ A\in{\rm GL}_{b'_3}(\bbf),\ B\in{\rm GL}_{b_3-b'_3}(\bbf),\ C\in{\rm GL}_{n_{10}-b'_{10}}(\bbf)$ and $D\in{\rm GL}_{b'_{10}}(\bbf)$. By Lemma \ref{lem5.1}, the subgroup of ${\rm GL}_{b_{10}+1}(\bbf)$ consisting of matrices
$$\bp \lambda & 0 & 0 \\ 0 & C & * \\ 0 & 0 & D \ep$$
has a finite number of orbits on $M(\bbf^{b_{10}+1})$. So we have $|Q_V\backslash M(V)|<\infty$ by Corollary \ref{cor8.7}.

Next suppose that $U_+^1\supset U_{(5)}^+$. Then we may assume
$$U_+^1=U_{(3),1}^+\oplus U_{(10),1}^+\oplus U_{(5)}^+$$
by Lemma \ref{lem5.10} and Lemma \ref{lem5.11}. With respect to the basis (\ref{eq8.5}), elements of $Q_V$ are represented by matrices
$$\bp A & * & * & * & * \\
0 & B & 0 & * & * \\
0 & 0 & \lambda & 0 & 0 \\
0 & 0 & 0 & C & * \\
0 & 0 & 0 & 0 & D \ep$$
with $\lambda\in\bbf^\times,\ A\in{\rm GL}_{b'_3}(\bbf),\ B\in{\rm GL}_{b_3-b'_3}(\bbf),\ C\in{\rm GL}_{b_{10}-b'_{10}}(\bbf)$ and $D\in{\rm GL}_{b'_{10}}(\bbf)$. By Lemma \ref{lem5.1}, the subgroup of ${\rm GL}_{n-b'_3}(\bbf)$ consisting of matrices
$$\bp B & 0 & * & * \\ 0 & \lambda & 0 & 0 \\ 0 & 0 & C & * \\ 0 & 0 & 0 & D \ep$$
has a finite number of orbits on $M(\bbf^{n-b_3})$. So we have $|Q_V\backslash M(V)|<\infty$ by Corollary \ref{cor8.7}.

Finally suppose that $U_+^1\not\subset U_{(3)}^+\oplus U_{(10)}^+$ and that $U_+^1\not\supset U_{(5)}^+$. Then we may assume
$$U_+^1=U_{(3),1}^+\oplus U_{(10),1}^+\oplus \bbf(e_n+e_{b_3+b'_{10}+1})$$
by Lemma \ref{lem5.10} and Lemma \ref{lem5.11}. With respect to the basis (\ref{eq8.5}), elements of $Q_V$ are represented by matrices
$$\bp A & * & * & * & * & * \\
0 & B & * & * & * & * \\
0 & 0 & \lambda & 0 & 0 & 0 \\
0 & 0 & 0 & C & * & * \\
0 & 0 & 0 & 0 & \lambda^{-1} & * \\
0 & 0 & 0 & 0 & 0 & D \ep$$
with $\lambda\in\bbf^\times,\ A\in{\rm GL}_{b'_3}(\bbf),\ B\in{\rm GL}_{b_3-b'_3}(\bbf),\ C\in{\rm GL}_{b_{10}-b'_{10}-1}(\bbf)$ and $D\in{\rm GL}_{b'_{10}}(\bbf)$. Let $H$ denote the subgroup of ${\rm GL}_{b_{10}+1}(\bbf)$ consisting of matrices
$$\bp \lambda & 0 & 0 & 0 & \\ 0 & C & * & * \\ 0 & 0 & \lambda^{-1} & * \\ 0 & 0 & 0 & D \ep$$
with $\lambda\in\bbf^\times,\ C\in{\rm GL}_{b_{10}-b'_{10}-1}(\bbf)$ and $D\in{\rm GL}_{b'_{10}}(\bbf)$. By Corollary \ref{cor8.7}, we have only to show that $|H\backslash M(\bbf^{b_{10}+1})|<\infty$. Let $H'$ be the subgroup of ${\rm GL}_{b_{10}+1}(\bbf)$ consisting of matrices
$$\bp \mu_1 & 0 & 0 & 0 & \\ 0 & C & * & * \\ 0 & 0 & \mu_2 & * \\ 0 & 0 & 0 & D \ep$$
with $\mu_1,\mu_2\in\bbf^\times,\ C\in{\rm GL}_{b_{10}-b'_{10}-1}(\bbf)$ and $D\in{\rm GL}_{b'_{10}}(\bbf)$. Then we have

\noindent $|H'\backslash M(\bbf^{b_{10}+1})|<\infty$ by Lemma \ref{lem5.1}. Since the center $Z=\{\nu I_{b_{10}+1}\mid \nu\in\bbf^\times\}$ of ${\rm GL}_{b_{10}+1}(\bbf)$ acts trivially on $M(\bbf^{b_{10}+1})$ and since $|H'/ZH|=|\bbf^\times/(\bbf^\times)^2|<\infty$, we have $|H\backslash M(\bbf^{b_{10}+1})|<\infty$.

(D) Case of $b_6=1$. The space $U_+$ is decomposed as
$$U_+=U_{(3)}^+\oplus U_{(10)}^+\oplus U_{(\overline{6})}^+$$
with $U_{(3)}^+=\bbf e_1\oplus\cdots\oplus \bbf e_{b_3},\ U_{(10)}^+=\bbf e_{b_3+1}\oplus\cdots\oplus \bbf e_{n-1}$ and $U_{(\overline{6})}^+=\bbf e_n$.

First suppose that $U_+^1\subset U_{(3)}^+\oplus U_{(10)}^+$. Then we may assume
$$U_+^1=U_{(3),1}^+\oplus U_{(10),1}^+$$
by Lemma \ref{lem5.10} and Lemma \ref{lem5.11}. With respect to the basis
\begin{equation}
e_1,\ldots,e_{b_3},e_{n+1},e_{n+2},\ldots,e_{\overline{b_3+1}} \label{eq8.6}
\end{equation}
of $V$, elements of $Q_V$ are represented by matrices
$$\bp A & * & 0 & * & * \\
0 & B & 0 & * & * \\
0 & 0 & \lambda & * & * \\
0 & 0 & 0 & C & * \\
0 & 0 & 0 & 0 & D \ep$$
with $\lambda\in\bbf^\times,\ A\in{\rm GL}_{b'_3}(\bbf),\ B\in{\rm GL}_{b_3-b'_3}(\bbf),\ C\in{\rm GL}_{b_{10}-b'_{10}}(\bbf)$ and $D\in{\rm GL}_{b'_{10}}(\bbf)$. By Lemma \ref{lem5.1}, the subgroup of ${\rm GL}_{b_3+1}(\bbf)$ consisting of matrices
$$\bp A & * & 0 \\ 0 & B & 0 \\ 0 & 0 & \lambda \ep$$
has a finite number of orbits on $M(\bbf^{b_3+1})$. So we have $|Q_V\backslash M(V)|<\infty$ by Corollary \ref{cor8.7}.

Next suppose that $U_+^1\not\subset U_{(3)}^+\oplus U_{(10)}^+$. Then we may assume
$$U_+^1=U_{(3),1}^+\oplus U_{(10),1}^+\oplus U_{(\overline{6})}^+$$
by Lemma \ref{lem5.10} and Lemma \ref{lem5.11}. With respect to the basis (\ref{eq8.6}), elements of $Q_V$ are represented by matrices
$$\bp A & * & 0 & * & * \\
0 & B & 0 & * & * \\
0 & 0 & \lambda & 0 & * \\
0 & 0 & 0 & C & * \\
0 & 0 & 0 & 0 & D \ep$$
with $\lambda\in\bbf^\times,\ A\in{\rm GL}_{b'_3}(\bbf),\ B\in{\rm GL}_{b_3-b'_3}(\bbf),\ C\in{\rm GL}_{b_{10}-b'_{10}}(\bbf)$ and $D\in{\rm GL}_{b'_{10}}(\bbf)$. By Lemma \ref{lem5.1}, the subgroup of ${\rm GL}_{b_3+b'_{10}+1}(\bbf)$ consisting of matrices
$$\bp A & * & 0 & * \\ 0 & B & 0 & * \\ 0 & 0 & \lambda & 0 \\ 0 & 0 & 0 & C \ep$$
has a finite number of orbits on $M(\bbf^{b_3+b'_{10}+1})$. So we have $|Q_V\backslash M(V)|<\infty$ by Corollary \ref{cor8.7}.

(E) Case of $b_8=1$. The space $U_+$ is decomposed as
$$U_+=U_{(3)}^+\oplus U_{(8)}^+\oplus U_{(10)}^+\oplus U_{(\overline{8})}^+$$
with $U_{(3)}^+=\bbf e_1\oplus\cdots\oplus \bbf e_{b_3},\ U_{(8)}^+= \bbf e_{b_3+1},\ U_{(10)}^+=\bbf e_{b_3+2}\oplus\cdots\oplus \bbf e_{n-1}$ and $U_{(\overline{8})}^+=\bbf e_n$. By Lemma \ref{lem5.10} and Lemma \ref{lem5.11}, we may assume
$$U_+^1=U_{(3),1}^+\oplus U_{(8),1}^+\oplus U_{(10),1}^+\oplus U_{(\overline{8}),1}^+$$
with $U_{(8),1}^+=\{0\}$ or $U_{(8)}^+$ and $U_{(\overline{8}),1}^+=\{0\}$ or $U_{(\overline{8})}^+$.

(E.1) Case of $U_{(8),1}^+=U_{(\overline{8}),1}^+=\{0\}$. With respect to the basis
\begin{equation}
e_1,\ldots,e_{b_3},e_{b_3+1}+e_{n+1},e_{n+2},\ldots,e_{\overline{b_3+2}},e_{\overline{b_3+1}}-e_n \label{eq8.7}
\end{equation}
of $V$, elements of $Q_V$ are represented by matrices
$$\bp A & * & * & * & * & * \\
0 & B & * & * & * & * \\
0 & 0 & \lambda & * & * & * \\
0 & 0 & 0 & C & * & * \\
0 & 0 & 0 & 0 & D & 0 \\
0 & 0 & 0 & 0 & 0 & \lambda^{-1} \ep$$
with $\lambda\in\bbf^\times,\ A\in{\rm GL}_{b'_3}(\bbf),\ B\in{\rm GL}_{b_3-b'_3}(\bbf),\ C\in{\rm GL}_{b_{10}-b'_{10}}(\bbf)$ and $D\in{\rm GL}_{b'_{10}}(\bbf)$. The subgroup of $\bbf^\times\times {\rm GL}_{b'_{10}}(\bbf)$ consisting of elements
$$\left(\lambda,\bp D & 0 \\ 0 & \lambda^{-1} \ep\right)$$
has a finite number of orbits on $M(\bbf)\times M(\bbf^{b'_{10}})$. So we have $|Q_V\backslash M(V)|<\infty$ by Corollary \ref{cor8.7}.

(E.2) Case of $U_{(8),1}^+=U_{(8)}^+$ and $U_{(\overline{8}),1}^+=\{0\}$. With respect to the basis (\ref{eq8.7}), elements of $Q_V$ are represented by matrices
$$\bp A & * & * & * & * & * \\
0 & B & 0 & * & * & * \\
0 & 0 & \lambda & * & * & * \\
0 & 0 & 0 & C & * & * \\
0 & 0 & 0 & 0 & D & * \\
0 & 0 & 0 & 0 & 0 & \lambda^{-1} \ep$$
with $\lambda\in\bbf^\times,\ A\in{\rm GL}_{b'_3}(\bbf),\ B\in{\rm GL}_{b_3-b'_3}(\bbf),\ C\in{\rm GL}_{b_{10}-b'_{10}}(\bbf)$ and $D\in{\rm GL}_{b'_{10}}(\bbf)$. Since the subgroup of ${\rm GL}_{b_3-b'_3}(\bbf)\times \bbf^\times$ consisting of elements
$$\left(\bp B & 0 \\ 0 & \lambda \ep,\lambda^{-1}\right)$$
has a finite number of orbits on $M(\bbf^{b_3-b'_3})\times M(\bbf)$, we have $|Q_V\backslash M(V)|<\infty$ by Corollary \ref{cor8.7}.

(E.3) Case of $U_{(8),1}^+=\{0\}$ and $U_{(\overline{8}),1}^+=U_{(\overline{8})}^+$. With respect to the basis (\ref{eq8.7}), elements of $Q_V$ are represented by matrices
$$\bp A & * & * & * & * & * \\
0 & B & * & * & * & * \\
0 & 0 & \lambda & 0 & * & 0 \\
0 & 0 & 0 & C & * & * \\
0 & 0 & 0 & 0 & D & 0 \\
0 & 0 & 0 & 0 & 0 & \lambda^{-1} \ep$$
with $\lambda\in\bbf^\times,\ A\in{\rm GL}_{b'_3}(\bbf),\ B\in{\rm GL}_{b_3-b'_3}(\bbf),\ C\in{\rm GL}_{b_{10}-b'_{10}}(\bbf)$ and $D\in{\rm GL}_{b'_{10}}(\bbf)$. Let $H$ denote the subgroup of ${\rm GL}_{b_{10}+2}(\bbf)$ consisting of matrices
$$\bp \lambda & 0 & * & 0 \\ 0 & C & * & * \\ 0 & 0 & D & 0 \\ 0 & 0 & 0 & \lambda^{-1} \ep$$
with $\lambda\in\bbf^\times,\ C\in{\rm GL}_{b_{10}-b'_{10}}(\bbf)$ and $D\in{\rm GL}_{b'_{10}}(\bbf)$. Then we have only to show that $|H\backslash M(\bbf^{b_{10}+2})|<\infty$ by Corollary \ref{cor8.7}. Let $H'$ be the subgroup of ${\rm GL}_{b_{10}+2}(\bbf)$ consisting of matrices
$$\bp \mu_1 & 0 & * & 0 \\ 0 & C & * & * \\ 0 & 0 & D & 0 \\ 0 & 0 & 0 & \mu_2 \ep$$
with $\mu_1,\mu_2\in\bbf^\times,\ C\in{\rm GL}_{b_{10}-b'_{10}}(\bbf)$ and $D\in{\rm GL}_{b'_{10}}(\bbf)$. Then we have $|H'\backslash M(\bbf^{b_{10}+2})|<\infty$ by Lemma \ref{lem9.8}. Since $Z=\{\nu I_{b_{10}+2}\mid \nu\in\bbf^\times\}$ acts trivially on $M(\bbf^{b_{10}+2})$ and since $|H'/ZH|=|\bbf^\times/(\bbf^\times)^2|<\infty$, we have $|H\backslash M(\bbf^{b_{10}+2})|<\infty$.

(E.4) Case of $U_{(8),1}^+=U_{(8)}^+$ and $U_{(\overline{8}),1}^+=U_{(\overline{8})}^+$. With respect to the basis (\ref{eq8.7}), elements of $Q_V$ are represented by matrices
$$\bp A & * & * & * & * & * \\
0 & B & 0 & * & * & * \\
0 & 0 & \lambda & 0 & * & * \\
0 & 0 & 0 & C & * & * \\
0 & 0 & 0 & 0 & D & * \\
0 & 0 & 0 & 0 & 0 & \lambda^{-1} \ep$$
with $\lambda\in\bbf^\times,\ A\in{\rm GL}_{b'_3}(\bbf),\ B\in{\rm GL}_{b_3-b'_3}(\bbf),\ C\in{\rm GL}_{b_{10}-b'_{10}}(\bbf)$ and $D\in{\rm GL}_{b'_{10}}(\bbf)$. By Lemma \ref{lem5.1}, the subgroup of ${\rm GL}_{b_3-b'_3+b_{10}-b'_{10}+1}(\bbf)\times \bbf^\times$ consisting of elements
$$\left(\bp B & 0 & * \\ 0 & \lambda & 0 \\ 0 & 0 & C \ep,\lambda^{-1}\right)$$
has a finite number of orbits on $M(\bbf^{b_3-b'_3+b_{10}-b'_{10}+1})\times M(\bbf)$. So we have $|Q_V\backslash M(V)|<\infty$ by Corollary \ref{cor8.7}.

Thus we have completed the proof of Proposition \ref{prop4.4}.

\section{Appendix}

\subsection{Two lemmas} \label{sec9.1}

Consider the general linear group $G={\rm GL}_n(\bbf)$ over an arbitrary field $\bbf$. Let $\mcb_n$ denote the Borel subgroup of $G$ consisting of upper triangular matrices in $G$ and $\mcb'$ the subgroup of $\mcb_n$ defined by
\begin{align*}
\mcb' & =\{g\in \mcb_n\mid ge_n\in \bbf e_n\} 
=\left\{\bp A & 0 \\ 0 & b \ep \Bigm| A\in\mcb_{n-1},\ b\in\bbf^\times\right\}.
\end{align*}
The following lemma is given in \cite{H}.

\begin{lemma} $|\mcb_n\backslash G/\mcb'|<\infty$.
\label{lem5.1}
\end{lemma}

For even $m$, define an alternating form $\langle\ ,\ \rangle$ on $\bbf^m$ by
$$\langle e_i,e_j\rangle =\begin{cases} -\delta_{i,m+1-j} & \text{if $i\le m/2$,} \\ \delta_{i,m+1-j} & \text{if $i> m/2$.} \end{cases}$$
Define $H={\rm Sp}'_m(\bbf)(\cong {\rm Sp}_m(\bbf))$ by
$$H=\{g\in{\rm GL}_m(\bbf)\mid \langle gu,gv\rangle =\langle u,v\rangle \mbox{ for }u,v\in\bbf^m\}.$$
For a decomposition $n=\alpha_1+\alpha_2+\alpha_3$ with even $\alpha_2$, define a subgroup $K=\mcb_{sp}(\alpha_1,\alpha_2,\alpha_3)$ of ${\rm GL}_n(\bbf)$ consisting of matrices
$$\bp A & * & * \\ 0 & B & * \\ 0 & 0 & C \ep$$
with $A\in\mcb_{\alpha_1},\ B\in{\rm Sp}'_{\alpha_2}(\bbf)$ and $C\in\mcb_{\alpha_3}$.

We prove the following extension of Lemma \ref{lem5.1} in Section \ref{sec11.3}.

\begin{lemma} $|K\backslash G/\mcb'|<\infty$.
\label{lem5.2}
\end{lemma}

\begin{remark} For $i=1,\ldots,n$, write $U_{(i)}=\bigoplus_{j\in\{1,\ldots,n\}-\{i\}} \bbf e_j$. Define a subgroup $\mcb'_i$ of $\mcb_n$ by
$$\mcb'_i=\{g\in \mcb_n\mid g\bbf e_i=\bbf e_i\mbox{ and }gU_{(i)}=U_{(i)}\}.$$
Then it is clear that $\mcb'_i$ is conjugate to $\mcb'\ (=\mcb'_n)$ by some permutation matrix.
\label{rem5.3}
\end{remark}

\subsection{Extensions of Proposition 6.3 in \cite{M3}}

Let $U=\bbf^{m+n}$ be the $m+n$-dimensional vector space over an arbitrary field $\bbf$. Consider the canonical direct sum decomposition
$$U=U_1\oplus U_2$$
where $U_1=\bbf e_1\oplus\cdots\oplus \bbf e_m$ and $U_2=\bbf e_{m+1}\oplus\cdots\oplus \bbf e_{m+n}$. Write $G_1={\rm GL}_m(\bbf)$ and $G_2={\rm GL}_n(\bbf)$. Any subgroup of $G_1\times G_2$ is identified with a subgroup of ${\rm GL}(U)$ by the canonical inclusion
$$G_1\times G_2\ni (A,B)\mapsto \bp A & 0 \\ 0 & B \ep \in {\rm GL}(U).$$
Let $\mcb_{(1)}$ and $\mcb_{(2)}$ denote the canonical Borel subgroups of $G_1$ and $G_2$, respectively, consisting of upper triangular matrices.

Let $M_{k_1,k_2}$ denote the flag variety
$$M_{k_1,k_2}=\{W_1\subset W_2\subset U \mid \dim W_1=k_1\mbox{ and }\dim W_2=k_1+k_2\}$$
of ${\rm GL}(U)$.

\begin{proposition} For $K=\mcb_{sp}(\alpha_1,\alpha_2,\alpha_3)$ with $m=\alpha_1+\alpha_2+\alpha_3$, we have$:$

{\rm (i)} $|(K\times \mcb_{(2)})\backslash M_{1,k}|<\infty$.

{\rm (ii)} $|(K\times \mcb_{(2)})\backslash M_{k,1}|<\infty$.
\label{prop7.4}
\end{proposition}

\begin{proof} Let $S$ be a $k+1$-dimensional subspace of $U=\bbf^{m+n}$. There are two canonical invariants
$$p=\dim(U_1\cap S)\mand q=\dim(U_2\cap S).$$
Put $r=k+1-p-q$. Let $\pi_i$ denote the projection $U_1\oplus U_2\to U_i$ for $i=1,2$.

Consider the canonical full flag
$$U_{2,m+1}\subset\cdots\subset U_{2,m+n-1}$$
of $U_2$ where $U_{2,i}=\bbf e_{m+1}\oplus\cdots\oplus \bbf e_i$. As in \cite{M3} Section 6, we define subsets
\begin{align*}
\{i_1,\ldots,i_q\} & =\{i\in I_2\mid \dim(U_{2,i}\cap S)=\dim(U_{2,i-1}\cap S)+1\} \\
\mand \{j_1,\ldots,j_r\} & =\{i\in I_2 \mid \dim(U_{2,i}\cap S)=\dim(U_{2,i-1}\cap S), \\
& \qquad \qquad \dim(\pi_2(S)\cap U_{2,i})=\dim(\pi_2(S)\cap U_{2,i-1})+1\}
\end{align*}
of $I_2=\{m+1,\ldots,m+n\}$ with $i_1<\cdots<i_q$ and $j_1<\cdots<j_r$. By the action of $\mcb_{(2)}$, we may assume
$$\pi_2(S)=W\oplus T \mand U_2\cap S=T$$
where
$$W=\bbf e_{i_1}\oplus\cdots\oplus \bbf e_{i_q}\mand T=\bbf e_{j_1}\oplus\cdots\oplus \bbf e_{j_r}.$$

(i) Let $S'$ be a one-dimensional subspace of $S$.

(A) First suppose $\pi_2(S')\not\subset W$. Then there exists an index $x\in\{1,\cdots,r\}$ such that
$$\pi_2(S')\subset W+U_{2,j_x} \quad\mbox{and that}\quad \pi_2(S')\not\subset W+U_{2,j_{x-1}}.$$
We can write
$$\pi_2(S')=\bbf(w+\lambda_1e_{j_1}+\cdots+\lambda_{x-1} e_{j_{x-1}}+e_{j_x})$$
with some $w\in W$ and $\lambda_1,\ldots,\lambda_{x-1}\in\bbf$. Take a $b\in\mcb_{(2)}$ such that
$$be_{j_x}=\lambda_1e_{j_1}+\cdots+\lambda_{x-1} e_{j_{x-1}}+e_{j_x}$$
and that $be_\ell=e_\ell$ for $\ell\in\{m+1,\ldots,m+n\}-\{j_x\}$. Then we have
$$b^{-1}\pi_2(S')=\bbf(w+e_{j_x}).$$
Since $b^{-1}\pi_2(S)=\pi_2(S)$, we can write
\begin{align*}
b^{-1}S & =\bbf u_1\oplus\cdots\oplus \bbf u_p\oplus W \oplus \bbf (v_1+e_{j_1})\oplus\cdots\oplus \bbf (v_r+e_{j_r}) \\
\mand b^{-1}S' & =\bbf(w+v_x+e_{j_x})
\end{align*}
with some linearly independent vectors $u_1,\ldots,u_p,v_1,\ldots,v_r$ in $U_1$. So we can take a $g\in G_1$ such that
\begin{align*}
S_0 & =gb^{-1}S=U_{1,p}\oplus W \oplus \bbf (e_{p+1}+e_{j_1})\oplus\cdots\oplus \bbf (e_{p+r}+e_{j_r}) \\
\mbox{and that}\quad S'_0 & =gb^{-1}S'=\bbf(w+e_{p+x}+e_{j_x})
\end{align*}
where $U_{1,\ell}=\bbf e_1\oplus\cdots\oplus \bbf e_\ell$ for $\ell=1,\ldots, m$.

Let $Q_{S_0}$ denote the isotropy subgroup of $S_0$ in $G_1\times \mcb_{(2)}$. Then it is shown in \cite{M3} Lemma 6.2 that
$$\pi_1(Q_{S_0})=P_1$$
where $P_1$ is the parabolic subgroup of $G_1$ stabilizing the flag $U_{1,p}\subset U_{1,p+1}\subset\cdots\subset U_{1,p+r}$ in $U_1$.
Let $Q_{S_0,S'_0}$ denote the isotropy subgroup
$$Q_{S_0,S'_0}=\{g\in G_1\times \mcb_{(2)}\mid gS_0=S_0\mbox{ and }gS'_0=S'_0\}$$
of the flag $S'_0\subset S_0$ in $G_1\times \mcb_{(2)}$. Then we have
$$\pi_1(Q_{S_0,S'_0})=\{g\in P_1\mid g\bbf e_{p+x}=\bbf e_{p+x}\}.$$
We have only to show that $(G_1\times \mcb_{(2)})/Q_{S_0,S'_0}$ is decomposed into a finite number of $K\times \mcb_{(2)}$-orbits. By the map $\pi_1$, we have
$$(K\times \mcb_{(2)})\backslash (G_1\times \mcb_{(2)})/Q_{S_0,S'_0}\stackrel{\sim}{\to} K\backslash G_1/\pi_1(Q_{S_0,S'_0}).$$
Write $U_{1,(p+x)}=\bigoplus_{i\in I_1-\{p+x\}} \bbf e_i$ with $I_1=\{1,\ldots,m\}$. Since $\pi_1(Q_{S_0,S'_0})$ contains a subgroup
$$\mcb'_{(1),p+x}=\{g\in \mcb_{(1)}\mid g\bbf e_{p+x}=\bbf e_{p+x},\ gU_{1,(p+x)}=U_{1,(p+x)}\}$$
and since $|K\backslash G_1/\mcb'_{(1),p+x}|<\infty$ by Lemma \ref{lem5.2} and Remark \ref{rem5.3}, we have
$$|K\backslash G_1/\pi_1(Q_{S_0,S'_0})|<\infty.$$

(B) Next suppose $\pi_2(S')\subset W$. Then we can write
\begin{align*}
S & =\bbf u_1\oplus\cdots\oplus \bbf u_p\oplus W \oplus \bbf (v_1+e_{j_1})\oplus\cdots\oplus \bbf (v_r+e_{j_r}) \\
\mand S' & =\bbf(\varepsilon u_1+w)
\end{align*}
with some linearly independent vectors $u_1,\ldots,u_p,v_1,\ldots,v_r$ in $U_1$, $w\in W$ and $\varepsilon\in\{0,1\}$. So we can take a $g\in G_1$ such that
\begin{align*}
S_0 & =gS=U_{1,p}\oplus W \oplus \bbf (e_{p+1}+e_{j_1})\oplus\cdots\oplus \bbf (e_{p+r}+e_{j_r}) \\
\mbox{and that}\quad S'_0 & =gS'=\bbf(\varepsilon e_1+w).
\end{align*}
Hence we have
$$\pi_1(Q_{S_0,S'_0})=\begin{cases} P'_1=\{g\in P_1\mid g\bbf e_1=\bbf e_1\} & \text{if $\varepsilon=1$,} \\ P_1 & \text{if $\varepsilon=0$.}\end{cases}$$
Since $P_1$ and $P'_1$ are parabolic subgroups of $G_1$, we have $|K\backslash G_1/\pi_1(Q_{S_0,S'_0})|<\infty$ by Lemma \ref{lem5.2} (or Lemma \ref{lem9.8}).

\bigskip
(ii) Let $S'$ be a $k$-dimensional subspace of $S$. Then we can take a nontrivial linear form $f$ on $S$ such that $S'=\{v\in S\mid f(v)=0\}$.

(A) First suppose $S'\supset S\cap U_1$. Then we can take a linear form $f'$ on $\pi_2(S)$ such that $f=f'\circ \pi_2$.

(A.1) Case of $f'(W)=\{0\}$. There exists an index $x\in\{1,\cdots,r\}$ such that
$$f'(e_{j_1})=\cdots=f'(e_{j_{x-1}})=0\quad\mbox{and that}\quad f'(e_{j_x})\ne 0.$$
Take a $b\in \mcb_{(2)}$ such that
$$be_{j_z}=e_{j_z}-\frac{f'(e_{j_z})}{f'(e_{j_x})}e_{j_x}$$
for $z=x+1,\ldots,r$ and that $be_\ell=e_\ell$ for $\ell\in\{m+1,\ldots,m+n\}-\{j_{x+1},\ldots,j_r\}$. Then
$$b^{-1}\pi_2(S')=W\oplus T_x$$
where $T_x=\bigoplus_{z\in\{j_1,\ldots,j_r\}-\{j_x\}} \bbf e_z$. Since $b^{-1}S=S$, we can write
\begin{align*}
b^{-1}S & =\bbf u_1\oplus\cdots\oplus \bbf u_p\oplus W \oplus \bbf (v_1+e_{j_1})\oplus\cdots\oplus \bbf (v_r+e_{j_r}) \\
\mand b^{-1}S' & =\bbf u_1\oplus\cdots\oplus \bbf u_p\oplus W \oplus T'_x
\end{align*}
with some linearly independent vectors $u_1,\ldots,u_p,v_1,\ldots,v_r$ in $U_1$ and $T'_x=$

\noindent $\bigoplus_{z\in\{1,\ldots,r\}-\{x\}} \bbf (v_z+e_{j_z})$. So we can take a $g\in G_1$ such that
\begin{align*}
S_0 & =gb^{-1}S=U_{1,p}\oplus W \oplus T'' \quad
\mbox{and that}\quad S'_0 =gb^{-1}S'=U_{1,p}\oplus W \oplus T''_x
\end{align*}
where $T''=\bigoplus_{z\in\{1,\ldots,r\}} \bbf (e_{p+z}+e_{j_z})$ and $T''_x=\bigoplus_{z\in\{1,\ldots,r\}-\{x\}} \bbf (e_{p+z}+e_{j_z})$.

Let $Q_{S_0,S'_0}$ denote the isotropy subgroup of the flag $S'_0\subset S_0$ in $G_1\times \mcb_{(2)}$. Then we have
$$\pi_1(Q_{S_0,S'_0})=\{g=\{g_{i,j}\}\in P_1\mid g_{p+x,j}=0\mbox{ for }j=p+x+1,\ldots,p+r\}.$$
Since $\pi_1(Q_{S_0,S'_0})$ contains the subgroup $\mcb'_{(1)}=\mcb'_{(1),p+x}$ defined in (i) (A), it follows from Lemma \ref{lem5.2} that
$$|K\backslash G_1/\pi_1(Q_{S_0,S'_0})|<\infty.$$

(A.2) Case of $f'(T)=\{0\}$. We can take a $g\in G_1$ such that
\begin{align*}
S_0 & =gS=U_{1,p}\oplus W \oplus T'' \quad
\mbox{and that}\quad S'_0 =gS'=U_{1,p}\oplus (U_2\cap S') \oplus T''.
\end{align*}
So we have $\pi_1(Q_{S_0,S'_0})=\pi_1(Q_{S_0})=P_1$ and hence
$$|K\backslash G_1/\pi_1(Q_{S_0,S'_0})|<\infty.$$

(A.3) Case of $f'(U_2\cap S)=f'(T)=\bbf$. In the same way as in (A.1), we can take a $b\in\mcb_{(2)}$ such that $b^{-1}S=S$ and that
$$b^{-1}\pi_2(S')=W_y\oplus T_x \oplus \bbf(e_{i_y}+e_{j_x})$$
where $x\in\{1,\ldots,r\},\ y\in \{1,\ldots,q\}$ and $W_y=\bigoplus_{z\in\{1,\ldots,q\}-\{y\}} \bbf e_{i_z}$. If $i_y<j_x$, then we can take a $b'\in\mcb_{(2)}$ such that $b'e_{j_x}=e_{i_y}+e_{j_x}$ and that $b'e_\ell=e_\ell$ for $\ell\ne j_x$. Since
$${b'}^{-1}b^{-1}\pi_2(S')=W_y\oplus T,$$
the problem is reduced to the case of (A.2). So we may assume $i_y>j_x$.

We can take a $g\in G_1$ such that
\begin{align*}
S_0 & =gb^{-1}S=U_{1,p}\oplus W \oplus T'' \\
\mbox{and that}\quad S'_0 & =gb^{-1}S'=U_{1,p}\oplus W_y \oplus T''_x\oplus \bbf(e_{i_y}+e_{p+x}+e_{j_x}).
\end{align*}
We will later describe $\pi_1(Q_{S_0,S'_0})$ in Lemma \ref{lem11.5'}. In particular, we have $\pi_1(Q_{S_0,S'_0})\supset \mcb'_{(1),p+x}$. So we have
$$|K\backslash G_1/\pi_1(Q_{S_0,S'_0})|<\infty$$
by Lemma \ref{lem5.2}.

(B) Next suppose $S'\not\supset S_0\cap U_1$. Then we can take linearly independent vectors $u_1,\ldots,u_p,v_1,\ldots,v_r\in U_1$ such that
\begin{align*}
S & =\bbf u_1\oplus\cdots\oplus \bbf u_p\oplus W \oplus \bbf (v_1+e_{j_1})\oplus\cdots\oplus \bbf (v_r+e_{j_r}) \\
\mbox{and that}\quad f(u_1) & =\cdots=f(u_{p-1})=f(v_1+e_{j_1})=\cdots=f(v_r+e_{j_r})=0,\ f(u_p)=1.
\end{align*}
So we have
$$S'=\begin{cases} \bbf u_1\oplus\cdots\oplus \bbf u_{p-1}\oplus W \oplus T' & \text{if $f(W)=\{0\}$,} \\
\bbf u_1\oplus\cdots\oplus \bbf u_{p-1}\oplus (U_2\cap S')\oplus \bbf(u_p+w)\oplus T' & \text{if $f(W)=\bbf$} \end{cases}$$
where $T'=\bbf (v_1+e_{j_1})\oplus\cdots\oplus \bbf (v_r+e_{j_r}),\ w\in U_2\cap S$ and $f(u_p+w)=0$. Hence we can take a $g\in G_1$ such that
\begin{align*}
S_0 & =gS=U_{1,p}\oplus W \oplus T'' \\
\mbox{and that}\quad S'_0 & =gS'=\begin{cases} U_{1,p-1}\oplus W \oplus T'' & \text{if $f(W)=\{0\}$,} \\
U_{1,p-1}\oplus (U_2\cap S')\oplus \bbf(e_p+w) \oplus T'' & \text{if $f(W)=\bbf$.} \end{cases}
\end{align*}
We have $\pi_1(Q_{S_0,S'_0})=P''_1=\{g\in P_1\mid g(U_{1,p-1})=U_{1,p-1}\}$. Since $P''_1$ is a parabolic subgroup of $G_1$, we have
$$|K\backslash G_1/\pi_1(Q_{S_0,S'_0})|<\infty$$
by Lemma \ref{lem5.2} (or Lemma \ref{lem9.8}).
\end{proof}

\begin{lemma} Suppose $S_0 =U_{1,p}\oplus W \oplus T''$ and $S'_0 =U_{1,p}\oplus W_y \oplus T''_x\oplus \bbf(e_{i_y}+e_{p+x}+e_{j_x})$ with $i_y>j_x$. Then
$$\pi_1(Q_{S_0,S'_0})=\{g=\{g_{i,j}\}\in P_1\mid x<z\le r,\ j_z<i_y\Longrightarrow g_{p+x,p+z}=0\}.$$
\label{lem11.5'}
\end{lemma}

\begin{proof} Write $Q=\{g=\{g_{i,j}\}\in P_1\mid x<z\le r,\ j_z<i_y\Longrightarrow g_{p+x,p+z}=0\}$.

(i) Proof of $Q\subset \pi_1(Q_{S_0,S'_0})$. We have only to consider generators of $Q$. First consider $g={\rm diag}(\lambda_1,\ldots,\lambda_m)\in P_1$ with $\lambda_i\in\bbf^\times$. Take $g'={\rm diag}(\mu_1,\ldots,\mu_n)\in\mcb_{(2)}$ such that
$$\mu_i=\begin{cases} \lambda_{p+z} & \text{if $i=j_z$ with some $z=1,\ldots,r$,} \\
\lambda_{p+x} & \text{if $i=i_y$,} \\
1 & \text{if $i\notin\{j_1,\ldots,j_r,i_y\}$}. \end{cases}$$
Then we have $(g,g')\in Q_{S_0,S'_0}$.

Next consider $g=g_{i,k}(\lambda)$ with $\lambda\in\bbf$ such that $ge_k=e_k+\lambda e_i$ and that $ge_\ell=e_\ell$ for $\ell\ne k$. If $i\le p$ or $k>p+r$, then we have $(g,e)\in Q_{S_0,S'_0}$. So we may assume $i=p+w$ and $k=p+z$ with $1\le w<z\le r$.

Suppose $w\ne x$. Define $g'\in\mcb_{(2)}$ by $g'e_{j_z}=e_{j_z}+\lambda e_{j_w}$ and $g'e_\ell=e_\ell$ for $\ell\ne j_z$. Then we have $(g,g')\in Q_{S_0,S'_0}$. So we have only to consider the case of $w=x$.

Suppose $j_z>i_y$. Define $g'\in\mcb_{(2)}$ by
$$g'e_{j_z}=e_{j_z}+\lambda e_{j_x}+\lambda e_{i_y}$$
and $g'e_\ell=e_\ell$ for $\ell\ne j_z$. Then we have $(g,g')\in Q_{S_0,S'_0}$.

Thus we have proved $Q\subset \pi_1(Q_{S_0,S'_0})$.

(ii) Proof of $\pi_1(Q_{S_0,S'_0})\subset Q$. Suppose $(g,g')\in Q_{S_0,S'_0}$. Take a $z$ such that $x<z$ and that $j_z<i_y$. Then we have only to show $ge_{p+z}\in U_{1,(p+x)}$.

Since $g\in P_1$ and $g'\in \mcb_{(2)}$, we have
$$(g,g')(e_{p+z}+e_{j_z})=ge_{p+z}+g'e_{j_z}\in U_{1,p+z}\oplus U_{2,j_z}.$$
On the other hand, since $(g,g')(e_{p+z}+e_{j_z})\in S'_0$ and since $i_y>j_z$, we have
$$(g,g')(e_{p+z}+e_{j_z})\in U_{1,p}\oplus W_y \oplus T''_x.$$
Hence $ge_{p+z}\in \pi_1(U_{1,p}\oplus W_y \oplus T''_x)\subset U_{1,(p+x)}$.
\end{proof}

\begin{lemma} Let $S$ be a two-dimensional subspace of $U=\bbf^{m+n}$. Then we can take a basis $v_1,v_2$ of $S$ such that $\mcb_{(1)}\times \mcb_{(2)}$ contains elements $g_{\lambda,\mu}$ satisfying
$$g_{\lambda,\mu}v_1=\lambda v_1\mand g_{\lambda,\mu}v_2=\mu v_2$$
for $\lambda,\mu\in\bbf^\times$.
\label{lem11.5}
\end{lemma}

\begin{proof} If $S=(S\cap U_1)\oplus (S\cap U_2)$, then the assertion is clear. So we may assume $\dim((S\cap U_1)\oplus (S\cap U_2))\le 1$.

(A) Case of $\dim(S\cap U_1)=1$. By the action of $\mcb_{(1)}$, we may assume $S\cap U_1=\bbf e_i$ with some $1\le i\le m$. We can take a nonzero $v_2\in S$ such that $v_2\in U_{1,(i)}\oplus U_2$ where $U_{1,(i)}=\bbf e_1\oplus\cdots\oplus\bbf e_{i-1}\oplus \bbf e_{i+1}\oplus\cdots\oplus\bbf e_m$. The element $g_{\lambda,\mu}$ of $\mcb_{(1)}\times \mcb_{(2)}$ defined by
$$g_{\lambda,\mu}e_j=\begin{cases} \lambda e_j & \text{if $j=i$,} \\
\mu e_j & \text{if $j\ne i$} \end{cases}$$
satisfies the desired condition for $v_1=e_i$ and $v_2$.

(B) Case of  $\dim(S\cap U_2)=1$ is proved in the same way as (A).

(C) Case of  $\dim(S\cap U_1)=\dim(S\cap U_2)=0$. There exists an $i$ such that
$$\dim(S\cap (U_{1,i}\oplus U_2))=1\quad \mbox{and that}\quad \dim(S\cap (U_{1,i-1}\oplus U_2))=0.$$
Write $S\cap (U_{1,i}\oplus U_2)=\bbf v_1$. Then there exists a $j$ with $m+1\le j\le m+n$ such that $(\mcb_{(1)}\times \mcb_{(2)})v_1\ni e_i+e_j$. So we may assume $v_1=e_i+e_j$. Take a nonzero element $v_2$ of $S\cap (U_1\oplus U_{2,(j)})$ where $U_{2,(j)}=\bbf e_{m+1}\oplus\cdots\oplus\bbf e_{j-1}\oplus\bbf e_{j+1}\oplus\cdots\oplus\bbf e_{m+n}$. Since $v_2\notin U_{1,i}\oplus U_2$, we can take an $i'>i$ such that
$$v_2\in U_{1,i'}\oplus U_2\quad\mbox{and that}\quad v_2\notin U_{1,i'-1}\oplus U_2.$$
We can take a $b\in \mcb_{(1)}$ such that $bv_2\in e_{i'}+U_{2,(j)}$ and that $be_i=e_i$. The element $g_{\lambda,\mu}$ of $\mcb_{(1)}\times \mcb_{(2)}$ defined by
$$g_{\lambda,\mu}e_k=\begin{cases} \lambda e_k & \text{if $k=i,\ j$,} \\
\mu e_k & \text{if $k\ne i,\ j$} \end{cases}$$
satisfies the desired condition for $v_1=e_i+e_j$ and $bv_2$.
\end{proof}

\subsection{Orbit decompositions of the full flag variety $M(\bbf^n)\cong{\rm GL}_n(\bbf)/\mcb$} \label{sec11.3}

(I) Let $V_1^0\subset V_2^0\subset\cdots\subset V_{n-1}^0$ be the canonical full flag in $\bbf^n$ defined by
$$V_1^0=\bbf e_1,\ V_2^0=\bbf e_1\oplus \bbf e_2,\ \ldots,\ V_{n-1}^0=\bbf e_1\oplus\cdots\oplus \bbf e_{n-1}.$$
Then
\begin{align*}
\mcb & =\mcb_n=\{g\in G\mid gV_i^0=V_i^0\mbox{ for }i=1,\ldots,n-1\} \\
& =\{\mbox{upper triangular matrices in }G\}
\end{align*}
is a Borel subgroup of $G={\rm GL}_n(\bbf)$.

Suppose $n=\alpha_1+\cdots+\alpha_p$ with positive integers $\alpha_1,\ldots,\alpha_p$. Define a partition $I=I_1\sqcup\cdots\sqcup I_p$ of $I=\{1,\ldots,n\}$ by
$$I_j=\{i(j,1),\ldots,i(j,\alpha_j)\}$$
for $j=1,\ldots,p$ where $i(j,k)=\alpha_1+\cdots+\alpha_{j-1}+k$.
Put $U_j=\oplus_{\ell\in I_j} \bbf e_\ell$. Then we have a direct sum decomposition
$\bbf^n=U_1\oplus\cdots\oplus U_p$.
Let $P$ be the parabolic subgroup of $G$ defined by
$$P=\left\{\bp A_1 && * \\ & \ddots & \\ 0 && A_p \ep \Bigm| A_j\in {\rm GL}_{\alpha_j}(\bbf)\mbox{ for }j=1,\ldots,p\right\}.$$
Then $P$ is the isotropy subgroup in $G$ for the flag
$U_1\subset U_1\oplus U_2\subset\cdots\subset U_1\oplus\cdots\oplus U_{p-1}$.

Let $S_n$ denote the symmetric group for $I$. For an element $\sigma\in S_n$, there corresponds a permutation matrix $w_\sigma$ defined by
$$w_\sigma(e_i)=e_{\sigma(i)}\quad\mbox{for }i\in I.$$
Let $\tau=\tau(\sigma)$ denote the unique element in $S_n$ such that $\tau(I_j)=I_j$
and that
\begin{equation}
\sigma^{-1}\tau^{-1}(i(j,1))<\cdots<\sigma^{-1}\tau^{-1}(i(j,\alpha_j)) \label{eq8.2}
\end{equation}
for $j=1,\ldots,p$. Since $w_\tau\in P$, we have
$$Pw_\sigma \mcb=Pw_\tau w_\sigma \mcb=Pw\mcb$$
where $w=w_\tau w_\sigma$. It follows from (\ref{eq8.2}) that
\begin{equation}
U_j\cap wV_{r(j,k)}^0=\bbf e_{i(j,1)}\oplus\cdots\oplus \bbf e_{i(j,k)}\quad\mbox{for }j=1,\ldots p\mbox{ and }k=1,\ldots,\alpha_j \label{eq8.3}
\end{equation}
where $r(j,k)=\sigma^{-1}\tau^{-1}(i(j,k))$.

Let $L$ be the canonical Levi subgroup of $P$ defined by
$$L=\left\{\bp A_1 && 0 \\ & \ddots & \\ 0 && A_p \ep \Bigm| A_j\in {\rm GL}_{\alpha_j}(\bbf)\right\}$$
and $N$ the unipotent part of $P$ defined by
$$N=\left\{\bp I_{\alpha_1} && * \\ & \ddots & \\ 0 && I_{\alpha_p} \ep \Bigm| A_j\in{\rm GL}_{\alpha_j}(\bbf)\right\}.$$
Then we have a Levi decomposition $P=LN=NL$ of $P$.

\begin{lemma} {\rm (i)} $L\cap w\mcb w^{-1}=L\cap \mcb$.

{\rm (ii)} $P\cap w\mcb w^{-1}=(L\cap \mcb)(N\cap w\mcb w^{-1}).$
\label{lem8.5}
\end{lemma}

\begin{proof} (i) Let $g$ be an element of $L$. Then $g\in w\mcb w^{-1}$ if and only if $ge_k\in wV_{\sigma^{-1}\tau^{-1}(k)}^0$ for $k=1,\ldots,n$. Hence the assertion follows from (\ref{eq8.3}).

(ii) Suppose $g\in P\cap w\mcb w^{-1}$. Then we have
$$ge_{i(j,k)}\in (U_1\oplus\cdots\oplus U_j)\cap wV_{r(j,k)}^0=(U_1\cap wV_{r(j,k)}^0)\oplus\cdots\oplus (U_j\cap wV_{r(j,k)}^0).$$
We also have $U_j\cap wV_{r(j,k)}^0=\bbf e_{i(j,1)}\oplus\cdots\oplus \bbf e_{i(j,k)}$ for $j=1,\ldots,p$ and $k=1,\ldots,\alpha_j$ by (\ref{eq8.3}). Write
$$ge_{i(j,k)}=u_{j,k}+v_{j,k}$$
with $u_{j,k}\in U_1\oplus\cdots\oplus U_{j-1}$ and $v_{j,k}\in U_j$. Since $g$ defines a linear isomorphism on the factor space $(U_1\oplus\cdots\oplus U_j)/(U_1\oplus\cdots\oplus U_{j-1})$, the map $\ell: e_{i(j,k)}\mapsto v_{j,k}$ defines a linear isomorphism on $U_j$. Hence $\ell\in L$. Since $v_{j,k}\in \bbf e_{i(j,1)}\oplus\cdots\oplus \bbf e_{i(j,k)}$ for $k=1,\ldots,\alpha_j$, we have $\ell\in L\cap \mcb$. Hence $\ell^{-1}g\in w\mcb w^{-1}$ by (i). Since $\ell^{-1}ge_{i(j,k)}=e_{i(j,k)}+u_{j,k}$, we have
$\ell^{-1}g\in N$.
\end{proof}

We can extend Proposition 6.5 and Corollary 6.6 in \cite{M3} as follows. Let $H$ be a subgroup of $L$ such that $|H\backslash L/(L\cap \mcb)|<\infty$ and let $K$ be the subgroup of $P$ defined by $K=HN=NH$.

\begin{proposition} Suppose $L=\bigsqcup_{k=1}^{n_H} Hg_k(L\cap \mcb)$. Then
$$PwB=\bigsqcup_{k=1}^{n_H} Kg_kw\mcb.$$
\label{prop8.6}
\end{proposition}

\begin{corollary} $\displaystyle{|K\backslash G/\mcb|=\frac{n_Hn!}{\alpha_1!\cdots \alpha_p!}}$.
\label{cor8.7}
\end{corollary}

\bigskip
(II) Suppose $p=6$ and $\alpha_6=1$. Then $n$ is decomposed as
$$n=\alpha_1+\alpha_2+\alpha_3+\alpha_4+\alpha_5+1.$$
Define a subgroup $H$ of $G={\rm GL}_n(\bbf)$ by
$$H=\left\{\bp A & 0 & * & 0 \\
0 & B & * & * \\
0 & 0 & C & 0 \\
0 & 0 & 0 & \lambda \ep \Bigm| A\in{\rm GL}_{\alpha_1}(\bbf),\ B\in{\rm GL}_{\alpha_2}(\bbf),\ C\in \mcb_{sp}(\alpha_3,\alpha_4,\alpha_5),\ \lambda\in\bbf^\times\right\}$$
%where $\mcb_{sp}(\alpha_3,\alpha_4,\alpha_5)$ is the subgroup of ${\rm GL}_{\alpha_3+\alpha_4+\alpha_5}(\bbf)$ consisting of matrices
%$$\bp C_1 & * & * \\ 0 & C_2 & * \\ 0 & 0 & C_3 \ep$$
%with $C_1\in\mcb_{\alpha_3},\ C_2\in {\rm Sp}'_{\alpha_4}(\bbf)$ and $C_3\in\mcb_{\alpha_5}$. ($\mcb_\ell$ is the subgroup of ${\rm GL}_\ell(\bbf)$ consisting of upper triangular matrices.)

Let $L$ denote the subgroup of $H$ consisting of matrices
$$\ell(A,B,C,D,E,\lambda)=\bp A & 0 & 0 & 0 & 0 & 0 \\
0 & B & 0 & 0 & 0 & 0 \\
0 & 0 & C & 0 & 0 & 0 \\
0 & 0 & 0 & D & 0 & 0 \\
0 & 0 & 0 & 0 & E & 0 \\
0 & 0 & 0 & 0 & 0 & \lambda \ep$$
with $A\in{\rm GL}_{\alpha_1}(\bbf),\ B\in{\rm GL}_{\alpha_2}(\bbf),\ C\in \mcb_{\alpha_3},\ D\in {\rm Sp}'_{\alpha_4}(\bbf),\ E\in\mcb_{\alpha_5}$ and $\lambda\in\bbf^\times$.

\begin{lemma} $H\backslash G/\mcb<\infty$.
\label{lem9.8}
\end{lemma}

\begin{proof} We will proceed by induction on $n$. We will first show that there are a finite number of $H$-orbits of $n-1$-dimensional subspaces in $\bbf^n$. Then for each representative $V$, we have only to show that the restriction $H_V$ of $\{g\in H\mid gV=V\}$ to $V$ has a finite number of orbits on the full flag variety $M(V)$ of $V$.

Let $V=\{v\in\bbf^n\mid f(v)=0\}$ be an $n-1$-dimensional subspace in $\bbf^n$ where $f:\bbf^n\to\bbf$ is a linear form on $\bbf^n$.

(A) Case of $f(U_1)=f(U_2)=\bbf$. We can take an element $\ell=\ell(A,B,I_{\alpha_3},I_{\alpha_4},I_{\alpha_5},1)\in L$ with some $A\in{\rm GL}_{\alpha_1}(\bbf)$ and $B\in{\rm GL}_{\alpha_2}(\bbf)$ such that
$$f(\ell e_i)=\begin{cases} 0 & \text{for $i=1,\ldots,\alpha_1-1,\alpha_1+1,\ldots,\alpha_1+\alpha_2-1$,} \\
1 & \text{for $i=\alpha_1,\alpha_1+\alpha_2$.} \end{cases}$$
Take an element $g\in H$ such that
$$ge_i=\begin{cases} e_i & \text{for $i=1,\ldots,\alpha_1+\alpha_2$,} \\
e_i-f(e_i)e_{\alpha_1+\alpha_2} & \text{for $i=\alpha_1+\alpha_2+1,\ldots,n$.} \end{cases}$$
Then we have
$$f(\ell ge_i)=\begin{cases} 0 & \text{for $i\in \{1,\ldots,n\}-\{\alpha_1,\alpha_1+\alpha_2\}$,} \\
1 & \text{for $i=\alpha_1,\alpha_1+\alpha_2$.} \end{cases}$$
So we may assume $V$ is defined by
$V=\{v\in \bbf^n\mid f(\ell gv)=0\}$
and hence it has a basis
$$e_1,\ldots,e_{\alpha_1-1},e_{\alpha_1+1},\ldots,e_{\alpha_1+\alpha_2-1}, e_{\alpha_1}-e_{\alpha_1+\alpha_2},e_{\alpha_1+\alpha_2+1},\ldots,e_n.$$
With respect to this basis, elements of $H_V$ are represented by matrices
$$\bp A & 0 & * & 0 \\
0 & B & * & * \\
0 & 0 & C & 0 \\
0 & 0 & 0 & \lambda \ep$$
with $A\in{\rm GL}_{\alpha_1-1}(\bbf),\ B\in{\rm GL}_{\alpha_2-1}(\bbf),\ C\in \mcb_{sp}(\alpha_3+1,\alpha_4,\alpha_5)$ and $\lambda\in\bbf^\times$. So we have $|H_V\backslash M(V)|<\infty$  by the assumption of induction.

(B) Case of $f(U_1)=\{0\}$ and $f(U_2)=\bbf$. We can take an element $\ell=$ 

\noindent $\ell(I_{\alpha_1},B,I_{\alpha_3},I_{\alpha_4},I_{\alpha_5},1)\in L$ with some $B\in{\rm GL}_{\alpha_2}(\bbf)$ such that
$$f(\ell e_i)=\begin{cases} 0 & \text{for $i=\alpha_1+1,\ldots,\alpha_1+\alpha_2-1$,} \\
1 & \text{for $i=\alpha_1+\alpha_2$.} \end{cases}$$
Take an element $g\in H$ such that
$$ge_i=\begin{cases} e_i & \text{for $i=1,\ldots,\alpha_1+\alpha_2$,} \\
e_i-f(e_i)e_{\alpha_1+\alpha_2} & \text{for $i=\alpha_1+\alpha_2+1,\ldots,n$.} \end{cases}$$
Then we have
$$f(\ell ge_i)=\begin{cases} 0 & \text{for $i\in \{1,\ldots,n\}-\{\alpha_1+\alpha_2\}$,} \\
1 & \text{for $i=\alpha_1+\alpha_2$.} \end{cases}$$
So we may assume $V$ is defined by
$V=\{v\in \bbf^n\mid f(\ell gv)=0\}$
and hence it has a basis
$e_1,\ldots,e_{\alpha_1+\alpha_2-1}, e_{\alpha_1+\alpha_2+1},\ldots,e_n$.
With respect to this basis, elements of $H_V$ are represented by matrices
$$\bp A & 0 & * & 0 \\
0 & B & * & * \\
0 & 0 & C & 0 \\
0 & 0 & 0 & \lambda \ep$$
with $A\in{\rm GL}_{\alpha_1}(\bbf),\ B\in{\rm GL}_{\alpha_2-1}(\bbf),\ C\in \mcb_{sp}(\alpha_3,\alpha_4,\alpha_5)$ and $\lambda\in\bbf^\times$. So we have $|H_V\backslash M(V)|<\infty$ by the assumption of induction.

(C) Case of $f(U_1)=f(U_6)=\bbf$ and $f(U_2)=\{0\}$. We may assume $f(e_n)=1$. We can take an element $\ell=\ell(A,I_{\alpha_2},I_{\alpha_3},I_{\alpha_4},I_{\alpha_5},1)\in L$ with some $A\in{\rm GL}_{\alpha_1}(\bbf)$ such that
$$f(\ell e_i)=\begin{cases} 0 & \text{for $i=1,\ldots,\alpha_1-1$,} \\
1 & \text{for $i=\alpha_1$.} \end{cases}$$
Take an element $g\in H$ such that
$$ge_i=\begin{cases} e_i & \text{for $i\in\{1,\ldots,\alpha_1+\alpha_2\}\sqcup\{n\}$,} \\
e_i-f(e_i)e_{\alpha_1} & \text{for $i\in \{\alpha_1+\alpha_2+1,\ldots,n-1\}$.} \end{cases}$$
Then we have
$$f(\ell ge_i)=\begin{cases} 0 & \text{for $i\in \{1,\ldots,n\}-\{\alpha_1,\ n\}$,} \\
1 & \text{for $i=\alpha_1,\ n$.} \end{cases}$$
Hence we may assume $V$ is defined by
$V=\{v\in \bbf^n\mid f(\ell gv)=0\}$
and hence it has a basis
$e_1,\ldots,e_{\alpha_1-1},e_{\alpha_1+1},\ldots, e_{n-1},e_{\alpha_1}-e_n$. With respect to this basis, elements of $H_V$ are represented by matrices
$$\bp A & 0 & * & * \\
0 & B & * & * \\
0 & 0 & C & 0 \\
0 & 0 & 0 & \lambda \ep$$
with $A\in{\rm GL}_{\alpha_1-1}(\bbf),\ B\in{\rm GL}_{\alpha_2}(\bbf),\ C\in \mcb_{sp}(\alpha_3,\alpha_4,\alpha_5)$ and $\lambda\in\bbf^\times$. The subgroup of ${\rm GL}_{\alpha_1+\alpha_2-1}(\bbf)$ consisting of matrices
$$\bp A & 0 \\ 0 & B \ep$$
has a finite number of orbits on $M(\bbf^{\alpha_1+\alpha_2-1})$. On the other hand, the subgroup of ${\rm GL}_{\alpha_3+\alpha_4+\alpha_5+1}(\bbf)$ consisting of matrices
$$\bp C & 0 \\ 0 & \lambda \ep$$
has a finite number of orbits on $M(\bbf^{\alpha_3+\alpha_4+\alpha_5+1})$ by the assumption of induction. So we have $|H_V\backslash M(V)|<\infty$ by Corollary \ref{cor8.7}..

(D) Case of $f(U_1)=\bbf$ and $f(U_2)=f(U_6)=\{0\}$. We can take an element $\ell=\ell(A,I_{\alpha_2},I_{\alpha_3},I_{\alpha_4},I_{\alpha_5},1)\in L$ with some $A\in{\rm GL}_{\alpha_1}(\bbf)$ such that
$$f(\ell e_i)=\begin{cases} 0 & \text{for $i=1,\ldots,\alpha_1-1$,} \\
1 & \text{for $i=\alpha_1$.} \end{cases}$$
Take an element $g\in H$ such that
$$ge_i=\begin{cases} e_i & \text{for $i\in\{1,\ldots,\alpha_1+\alpha_2\}\sqcup\{n\}$,} \\
e_i-f(e_i)e_{\alpha_1} & \text{for $i\in \{\alpha_1+\alpha_2+1,\ldots,n-1\}$.} \end{cases}$$
Then we have
$$f(\ell ge_i)=\begin{cases} 0 & \text{for $i\in \{1,\ldots,n\}-\{\alpha_1\}$,} \\
1 & \text{for $i=\alpha_1$.} \end{cases}$$
So we may assume $V$ is defined by
$V=\{v\in \bbf^n\mid f(\ell gv)=0\}$
and hence it has a basis
$e_1,\ldots,e_{\alpha_1-1},e_{\alpha_1+1},\ldots, e_n$. With respect to this basis, elements of $H_V$ are represented by matrices
$$\bp A & 0 & * & 0 \\
0 & B & * & * \\
0 & 0 & C & 0 \\
0 & 0 & 0 & \lambda \ep$$
with $A\in{\rm GL}_{\alpha_1-1}(\bbf),\ B\in{\rm GL}_{\alpha_2}(\bbf),\ C\in \mcb_{sp}(\alpha_3,\alpha_4,\alpha_5)$ and $\lambda\in\bbf^\times$. So we have $|H_V\backslash M(V)|<\infty$ by the assumption of induction.

(E) Case of $f(U_1)=f(U_2)=\{0\}$ and $f(U_3)=f(U_6)=\bbf$. We may assume $f(e_n)=1$. Suppose that
$$f(e_j)\begin{cases} =0 & \text{for $j=\alpha_1+\alpha_2+1,\ldots,k-1$,} \\
\ne 0 & \text{for $j=k$} \end{cases}$$
($\alpha_1+\alpha_2+1\le k\le \alpha_1+\alpha_2+\alpha_3$). Then we can take an element $\ell\in L$ such that
$\ell e_k= f(e_k)^{-1}e_k$
and that $\ell e_i=e_i$ for $i\ne k$. Take an element $g\in H$ such that
$$ge_i= \begin{cases} e_i & \text{for $i\in\{1,\ldots,n\}-\{k+1,\ldots,n-1\}$}, \\
e_i-f(e_i)e_k & \text{for $i\in\{k+1,\ldots,n-1\}$}. \end{cases}$$
Then we have
$$f(\ell ge_i)=\begin{cases} 0 & \text{for $i\in \{1,\ldots,n\}-\{k,n\}$,} \\
1 & \text{for $i=k,n$.} \end{cases}$$
So we may assume $V$ is defined by
$V=\{v\in \bbf^n\mid f(\ell gv)=0\}$
and hence it has a basis
$e_1,\ldots,e_{k-1},e_{k+1}\ldots, e_{n-1},e_k-e_n$. With respect to this basis, elements of $H_V$ are represented by matrices
$$\bp A & 0 & * & * & * \\
0 & B & * & * & * \\
0 & 0 & C & * & * \\ 
0 & 0 & 0 & D & 0 \\
0 & 0 & 0 & 0 & \lambda \ep$$
with $A\in{\rm GL}_{\alpha_1}(\bbf),\ B\in{\rm GL}_{\alpha_2}(\bbf),\ C\in \mcb_{k'-1},\ D\in\mcb_{sp}(\alpha_3-k',\alpha_4,\alpha_5)$ and $\lambda\in\bbf^\times$ where $k'=k-\alpha_1-\alpha_2$. By the assumption of induction, the subgroup of ${\rm GL}_{\alpha_3-k'+\alpha_4+\alpha_5+1}(\bbf)$ consisting of matrices
$$\bp D & 0 \\ 0 & \lambda \ep$$
has a finite number of orbits on $M(\bbf^{\alpha_3-k'+\alpha_4+\alpha_5+1})$. So we have $|H_V\backslash M(V)|<\infty$ by Corollary \ref{cor8.7}.

(F) Case of $f(U_1)=f(U_2)=f(U_3)=\{0\}$ and $f(U_4)=f(U_6)=\bbf$. We may assume $f(e_n)=1$. We can take an element $\ell=\ell(I_{\alpha_1},I_{\alpha_2},I_{\alpha_3},D,I_{\alpha_5},1)$ with some $D\in{\rm Sp}'_{\alpha_4}(\bbf)$ such that
$$f(\ell e_j)= \begin{cases} 0 & \text{for $j=\alpha_1+\alpha_2+\alpha_3+1,\ldots,k-1$,} \\
1 & \text{for $j=k$} \end{cases}$$
where $k=\alpha_1+\alpha_2+\alpha_3+\alpha_4$. Take an element $g\in H$ such that
$$ge_i= \begin{cases} e_i & \text{for $i\in\{1,\ldots,n\}-\{k+1,\ldots,n-1\}$}, \\
e_i-f(e_i)e_k & \text{for $i\in\{k+1,\ldots,n-1\}$}. \end{cases}$$
Then we have
$$f(\ell ge_i)=\begin{cases} 0 & \text{for $i\in \{1,\ldots,n\}-\{k,n\}$,} \\
1 & \text{for $i=k,n$.} \end{cases}$$
So we may assume $V$ is defined by
$V=\{v\in \bbf^n\mid f(\ell gv)=0\}$
and hence it has a basis
$e_1,\ldots,e_{k-1},e_k-e_n,e_{k+1}\ldots, e_{n-1}$. With respect to this basis, elements of $H_V$ are represented by matrices
$$\bp A & 0 & * & * & * \\
0 & B & * & * & * \\
0 & 0 & C & * & * \\ 
0 & 0 & 0 & D & X \\
0 & 0 & 0 & 0 & E \ep$$
with $A\in{\rm GL}_{\alpha_1}(\bbf),\ B\in{\rm GL}_{\alpha_2}(\bbf),\ C\in \mcb_{\alpha_3},\ D\in Q_{\alpha_4},\ E\in \mcb_{\alpha_5}$ and $X=\{x_{i,j}\}$ with $x_{\alpha_4,j}=0$ for $j=1,\ldots,\alpha_5$ where
\begin{align*}
Q_{\alpha_4} & =\{g\in {\rm Sp}'_{\alpha_4}(\bbf)\mid g\bbf e_1=\bbf e_1\} \\
& =\{g\in {\rm Sp}'_{\alpha_4}(\bbf)\mid g(\bbf e_1\oplus\cdots\oplus \bbf e_{\alpha_4-1})=\bbf e_1\oplus\cdots\oplus \bbf e_{\alpha_4-1}\}.
\end{align*}
The subgroup of ${\rm GL}_{\alpha_4+\alpha_5}(\bbf)$ consisting of matrices
$$\bp D & X \\ 0 & E \ep$$
has a finite number of orbits on $M(\bbf^{\alpha_4+\alpha_5})$ by Lemma \ref{lem9.12} in the next subsection. So we have $|H_V\backslash M(V)|<\infty$ by Corollary \ref{cor8.7}.

(G) Case of $f(U_1)=f(U_2)=f(U_3)=f(U_4)=\{0\}$ and $f(U_5)=f(U_6)=\bbf$. We may assume $f(e_n)=1$. Suppose that
$$f(e_j)\begin{cases} =0 & \text{for $j=\alpha_1+\alpha_2+\alpha_3+\alpha_4+1,\ldots,k-1$,} \\
\ne 0 & \text{for $j=k$} \end{cases}$$
($\alpha_1+\alpha_2+\alpha_3+\alpha_4+1\le k\le n-1$). Then we can take an element $\ell=\ell(I_{\alpha_1},I_{\alpha_2},I_{\alpha_3},I_{\alpha_4},E,1)\in L$ with some $E\in\mcb_{\alpha_5}$ and $\lambda\in\bbf^\times$ such that
$$f(\ell e_i)=\begin{cases} 0 & \text{for $i\in \{1,\ldots,n\}-\{k,n\}$,} \\
1 & \text{for $i=k,n$.} \end{cases}.$$
So we may assume $V$ is defined by
$V=\{v\in \bbf^n\mid f(\ell v)=0\}$
and hence it has a basis
$e_1,\ldots,e_{k-1},e_{k+1}\ldots, e_{n-1},e_k-e_n$. With respect to this basis, elements of $H_V$ are represented by matrices
$$\bp A & 0 & * & * & * \\
0 & B & * & * & * \\
0 & 0 & C & * & * \\ 
0 & 0 & 0 & D & 0 \\
0 & 0 & 0 & 0 & \lambda \ep$$
with $A\in{\rm GL}_{\alpha_1}(\bbf),\ B\in{\rm GL}_{\alpha_2}(\bbf),\ C\in \mcb_{sp}(\alpha_3,\alpha_4,k'),\ D\in\mcb_{\alpha_5-k'-1}$ and $\lambda\in\bbf^\times$ where $k'=k-\alpha_1-\alpha_2-\alpha_3-\alpha_4$. Since the subgroup of ${\rm GL}_{\alpha_5-k'}(\bbf)$ consisting of matrices
$$\bp D & 0 \\ 0 & \lambda \ep$$
has a finite number of orbits on $M(\bbf^{\alpha_5-k'})$ by Lemma \ref{lem5.1}, we have $|H_V\backslash M(V)|<\infty$ by Corollary \ref{cor8.7}.

(H) Case of $f(U_1)=f(U_2)=f(U_3)=f(U_4)=f(U_5)=\{0\}$ and $f(U_6)=\bbf$. Clearly we may assume $f(e_n)=1$. With respect to the basis $e_1,\ldots, e_{n-1}$ of $V$, elements of $H_V$ are represented by matrices
$$\bp A & 0 & * \\
0 & B & * \\
0 & 0 & C \ep$$
with $A\in{\rm GL}_{\alpha_1}(\bbf),\ B\in{\rm GL}_{\alpha_2}(\bbf)$ and $C\in \mcb_{sp}(\alpha_3,\alpha_4,\alpha_5)$. Hence we have $|H_V\backslash M(V)|<\infty$ by Corollary \ref{cor8.7}.

(I) Case of $f(U_1)=f(U_2)=f(U_6)=\{0\}$ and $f(U_3)=\bbf$. As in (E), we may assume
$$f(e_k)=1\mand f(e_i)=0\mbox{ for }i\ne k$$
with some $k\in\{\alpha_1+\alpha_2+1,\ldots,\alpha_1+\alpha_2+\alpha_3\}$. With respect to the basis $e_1,\ldots, e_{k-1},$ $e_{k+1},\ldots,e_n$ of $V$, elements $H_V$ are represented by matrices
$$\bp A & 0 & * & 0 \\
0 & B & * & * \\
0 & 0 & C & 0 \\
0 & 0 & 0 & \lambda \ep$$
with $A\in{\rm GL}_{\alpha_1}(\bbf),\ B\in{\rm GL}_{\alpha_2}(\bbf),\ C\in \mcb_{sp}(\alpha_3-1,\alpha_4,\alpha_5)$ and $\lambda\in\bbf^\times$. So we have $|H_V\backslash M(V)|<\infty$ by the assumption of induction.

(J) Case of $f(U_1)=f(U_2)=f(U_3)=f(U_6)=\{0\}$ and $f(U_4)=\bbf$. As in (F), we may assume
$$f(e_k)=1\mand f(e_i)=0\mbox{ for }i\ne k$$
with $k=\alpha_1+\alpha_2+\alpha_3+\alpha_4$. With respect to the basis $e_1,\ldots, e_{k-1},$ $e_{k+1},\ldots,e_n$ of $V$, elements of $H_V$ are represented by matrices
$$\bp A & 0 & * & 0 \\
0 & B & * & * \\
0 & 0 & C & 0 \\
0 & 0 & 0 & \lambda \ep$$
with $A\in{\rm GL}_{\alpha_1}(\bbf),\ B\in{\rm GL}_{\alpha_2}(\bbf),\ C\in \mcb_{sp}(\alpha_3+1,\alpha_4-2,\alpha_5)$ and $\lambda\in\bbf^\times$. (Here we note that the restrictions of elements in $Q_{\alpha_4}$ to $\bbf e_1\oplus\cdots\oplus \bbf e_{\alpha_4-1}$ are represented by matrices
$$\bp \mu & * \\ 0 & D \ep$$
with some $\mu\in\bbf^\times$ and $D\in {\rm Sp}'_{\alpha_4-2}(\bbf)$.) So we have $|H_V\backslash M(V)|<\infty$ by the assumption of induction.

(K) Case of $f(U_1)=f(U_2)=f(U_3)=f(U_4)=f(U_6)=\{0\}$ and $f(U_5)=\bbf$. As in (G), we may assume
$$f(e_k)=1\mand f(e_i)=0\mbox{ for }i\ne k$$
with some $k\in\{\alpha_1+\alpha_2+\alpha_3+\alpha_4+1,\ldots,n-1\}$. With respect to the basis $e_1,\ldots, e_{k-1},$ $e_{k+1},\ldots,e_n$ of $V$, elements of $H_V$ are represented by matrices
$$\bp A & 0 & * & 0 \\
0 & B & * & * \\
0 & 0 & C & 0 \\
0 & 0 & 0 & \lambda \ep$$
with $A\in{\rm GL}_{\alpha_1}(\bbf),\ B\in{\rm GL}_{\alpha_2}(\bbf),\ C\in \mcb_{sp}(\alpha_3,\alpha_4,\alpha_5-1)$ and $\lambda\in\bbf^\times$. So we have $|H_V\backslash M(V)|<\infty$ by the assumption of induction.
\end{proof}

Let $\tau$ be an automorphism of $G$ defined by
$\tau(g)=J_n\,{}^tg^{-1}J_n$.
Then $\tau(H)$ consists of matrices
$$\bp \lambda & 0 & * & 0 \\
0 & C & * & * \\
0 & 0 & B & 0 \\
0 & 0 & 0 & A \ep$$
with $A\in{\rm GL}_{\alpha_1}(\bbf),\ B\in{\rm GL}_{\alpha_2}(\bbf),\ C\in \mcb_{sp}(\alpha_5,\alpha_4,\alpha_3)$ and $\lambda\in\bbf^\times$.

\begin{corollary} $|\tau(H)\backslash G/\mcb|<\infty$.
\label{cor10.9}
\end{corollary}

\noindent {\it Proof of Lemma \ref{lem5.2}.} Consider subgroups
$$H=\left\{\bp \lambda & 0 \\ 0 & A \ep \Bigm| \lambda\in\bbf^\times,\ A\in{\rm GL}_n(\bbf)\right\}\mand H'=\left\{\bp \lambda & 0 \\ 0 & A \ep \Bigm| \lambda\in\bbf^\times,\ A\in K\right\}$$
of ${\rm GL}_{n+1}(\bbf)$ where $K=\mcb_{sp}(\alpha_3,\alpha_4,\alpha_5)$ with $\alpha_3+\alpha_4+\alpha_5=n$. Take a full flag $m: V_1\subset V_2\subset\cdots\subset V_n$ in $\bbf^{n+1}$ defined by
$$V_1=\bbf(e_1+e_{n+1})\mand V_j=\bbf(e_1+e_{n+1})\oplus\bbf e_2\oplus\cdots\oplus \bbf e_j$$
for $j=2,\ldots,n$. Then the isotropy subgroup of $m$ in $H$ is
$$H_0=\left\{\bp \lambda & 0 & 0 \\ 0 & A & 0 \\ 0 & 0 & \lambda \ep \Bigm| \lambda\in\bbf^\times,\ A\in \mcb_{n-1} \right\}.$$
By Lemma \ref{lem9.8} (case of $\alpha_1=\alpha_2=0$), we have
$|H'\backslash M(\bbf^{n+1})|<\infty$.
Hence
$$|H'\backslash H/H_0|<\infty.$$
Consider the projection $\pi: H\to {\rm GL}_n(\bbf)$ given by
$$\pi\left(\bp \lambda & 0 \\ 0 & A \ep\right)=A.$$
Then we have
$H'\backslash H/H_0\cong \pi(H')\backslash \pi(H)/\pi(H_0)$.
Since $\pi(H')=K,\ \pi(H)={\rm GL}_n(\bbf)=G$ and $\pi(H_0)=\mcb'$, we have $|K\backslash G/\mcb'|<\infty$ as desired. \hfill$\square$

\subsection{A lemma}

Suppose $n=\alpha+\beta$ with even $\alpha$. Consider the subgroup
$$H=H_{\alpha,\beta}=\left\{\bp A & X \\ 0 & B \ep \Bigm| A\in Q_\alpha,\ B\in\mcb_\beta,\ X\in \mcm'_{\alpha,\beta}(\bbf)\right\}$$
of ${\rm GL}_n(\bbf)$ where
\begin{align*}
Q_\alpha & =\{g\in {\rm Sp}'_\alpha(\bbf)\mid g\bbf e_1=\bbf e_1\} \\
& =\{g\in {\rm Sp}'_\alpha(\bbf)\mid g(\bbf e_1\oplus\cdots\oplus \bbf e_{\alpha-1})=\bbf e_1\oplus\cdots\oplus \bbf e_{\alpha-1}\}, \\
\mcm_{\alpha,\beta}(\bbf) & = \mbox{  the space of $\alpha\times\beta$ matrices with entries in $\bbf$} \\
\mand \mcm'_{\alpha,\beta}(\bbf) & =\{X=\{x_{i,j}\}\in \mcm_{\alpha,\beta}(\bbf)\mid x_{\alpha,1}=\cdots=x_{\alpha,\beta}=0\}.
\end{align*}

\begin{lemma} $|H_{\alpha,\beta}\backslash M(\bbf^n)|<\infty$.
\label{lem9.12}
\end{lemma}

\begin{proof} We will proceed by induction on $\beta$. When $\beta=0$, we have $|H_{\alpha,0}\backslash M(\bbf^\alpha)|=|Q_\alpha \backslash M(\bbf^\alpha)|<\infty$ by Theorem 1.14 of \cite{M2}. Suppose $\beta>0$ and put $U'=\bbf e_1\oplus\cdots\oplus \bbf e_{n-1}$. Then we may assume $|H_{\alpha,\beta-1}\backslash M(U')|<\infty$.

Let $\{0\}=V_0\subset V_1\subset\cdots\subset V_n=\bbf^n$ be a full flag in $\bbf^n$. We can take an $i$ with $1\le i\le n$ such that
$$V_{i-1}\subset U'\quad\mbox{and that}\quad V_i\not\subset U'.$$
Write $V_i=V_{i-1}\oplus \bbf v$ with $v=\lambda_1e_1\oplus\cdots\oplus \lambda_ne_n$. Then we have $\lambda_n\ne 0$. Taking a constant multiple of $v$, we may assume
$\lambda_\alpha=0$ or $1$.
Let $g$ be an element of $H_{\alpha,\beta}$ defined by
$$ge_k=e_k\mbox{ for }k=1,\ldots,n-1\mand ge_n=v-\lambda_\alpha e_\alpha.$$
Then we have
$$g^{-1}v=e_n+\lambda_\alpha e_\alpha.$$
For $j=i,\ldots,n$, we can write
$V_j=(V_j\cap U')\oplus \bbf v$.
Hence
$$g^{-1}V_j=\begin{cases} V_j & \text{for $j=1,\ldots,i-1$,} \\
(V_j\cap U')\oplus \bbf (e_n+\lambda_\alpha e_\alpha) & \text{for $j=i,\ldots,n$.} \end{cases}$$
Since
$V_1\subset\cdots\subset V_{i-1}\subset V_{i+1}\cap U'\subset\cdots\subset V_n\cap U'=U'$
is a full flag in $U'$, we have
$$|H_{\alpha,\beta}\backslash M(\bbf^n)|\le 2n|H_{\alpha,\beta-1}\backslash M(U')|<\infty.$$
\end{proof}

\subsection{Lemmas related with ${\rm SL}_2(\bbf)$}

Let $H_0$ denote the subgroup of $G={\rm GL}_n(\bbf)$ consisting of matrices
$$\bp A & 0 \\ 0 & B \ep$$
with $A\in{\rm SL}_2(\bbf)$ and $B\in{\rm GL}_{n-2}(\bbf)$.

\begin{lemma} $|H_0\backslash M(\bbf^n)|<\infty$ for the full flag variety $M(\bbf^n)$.
\label{lem12.13}
\end{lemma}

\begin{proof} We will proceed by induction on $n$. Put $U_1=\bbf e_1\oplus\bbf e_2$ and $U_2=\bbf e_3\oplus\cdots\oplus \bbf e_n$. Let $V$ be an $n-1$ dimensional subspace of $\bbf^n$. We have only to show $|H(V)\backslash M(V)|<\infty$ where $H(V)=\{g\in H_0\mid gV=V\}$.

(A) Case of $V\supset U_1$. By the action of $H_0$, we may assume $V=\bbf e_1\oplus\cdots\oplus\bbf e_{n-1}$. It is clear that $|H(V)\backslash M(V)|<\infty$ by the assumption of induction.

(B) Case of $V\supset U_2$. By the action of $H_0$, we may assume $V=\bbf e_2\oplus\cdots\oplus\bbf e_n$. With respect to the basis $e_2,\ldots,e_n$ of $V$, elements of $H(V)$ are represented by
$$\bp a & 0 \\ 0 & B \ep$$
with $a\in\bbf^\times$ and $B\in{\rm GL}_{n-2}(\bbf)$. Hence $|H(V)\backslash M(V)|<\infty$.

(C) Case of $V\not\supset U_1$ and $V\not\supset U_2$. By the action of $H_0$, we may assume $V=\bbf e_1\oplus\bbf e_3\oplus\cdots\oplus\bbf e_{n-1}\oplus\bbf (e_2+e_n)$. With respect to the basis $e_1,e_3,\ldots,e_{n-1},e_2+e_n$ of $V$, elements of $H(V)$ are represented by
$$\bp a & 0 & * \\ 0 & B & * \\ 0 & 0 & a^{-1} \ep$$
with $a\in\bbf^\times$ and $B\in{\rm GL}_{n-3}(\bbf)$. Since the subgroup of ${\rm GL}_{n-2}(\bbf)\times \bbf^\times$ consisting of elements
$$\left(\bp a & 0 \\ 0 & B \ep,\ a^{-1}\right)$$
has a finite number of orbits on $M(\bbf^{n-2})\times M(\bbf)$, we have $|H(V)\backslash M(V)|<\infty$ by Corollary \ref{cor8.7}.
\end{proof}

Let $H_1$ denote the subgroup of $G={\rm GL}_n(\bbf)$ consisting of matrices
$$\bp a & * & 0 \\ 0 & a^{-1} & 0 \\ 0 & 0 & B \ep$$
with $a\in\bbf^\times$ and $B\in{\rm GL}_{n-2}(\bbf)$.

\begin{lemma} $|H_1\backslash M_k(\bbf^n)|<\infty$ for the Grassmann variety $M_k(\bbf^n)$.
\label{lem11.13}
\end{lemma}

\begin{proof} Since the subgroup of ${\rm GL}_2(\bbf)$ consisting of matrices
$$\left\{\bp a & x \\ 0 & a^{-1} \ep \Bigm| a\in\bbf^\times,\ x\in\bbf \right\}$$
of ${\rm GL}_2(\bbf)$ has two orbits on $M(\bbf^2)$, the assertion follows from Proposition 6.3 in \cite{M3}.
\end{proof}

Let $H_2$ denote the subgroup of $G$ consisting of matrices
$$\bp A & 0 & 0 \\ 0 & \lambda & * \\ 0 & 0 & B \ep$$
with $\lambda\in\bbf^\times,\ A\in {\rm SL}_2(\bbf)$ and $B\in{\rm GL}_{n-3}(\bbf)$.

\begin{lemma} $|H_2\backslash M_{k_1,k_2}(\bbf^n)|<\infty$ for the flag variety $M_{k_1,k_2}=\{V_1\subset V_2\subset \bbf^n \mid \dim V_1=k_1,\ \dim V_2=k_1+k_2\}$.
\label{lem11.14}
\end{lemma}

\begin{proof} We will proceed by induction on $n$. Put
$$U_1=\bbf e_1\oplus \bbf e_2,\quad U_2=\bbf e_3\mand U_3=\bbf e_4\oplus\cdots\oplus \bbf e_n.$$
For $i=1,2$ and $3$, let $\pi_i:\bbf^n\to U_i$ denote the canonical projections. Let $V_1\subset V_2$ be a flag in $M_{k_1,k_2}$.

Suppose $\dim \pi_3(V_2)=\ell<n-3$. Then we can take an $h\in H_2$ such that
$$\pi_3(hV_2)=\bbf e_4\oplus\cdots\oplus \bbf e_{\ell+3}.$$
Put $V'=\bbf e_1\oplus\cdots\oplus \bbf e_{\ell+3}$. Then $hV_1\subset hV_2\subset V'$. Elements of the restriction $H'=H'(V')$ of $\{g\in H_2\mid gV'=V'\}$ to $V'$ are represented by
$$\bp A & 0 & 0 \\ 0 & \lambda & * \\ 0 & 0 & B \ep$$
($\lambda\in\bbf^\times,\ A\in{\rm SL}_2(\bbf),\ B\in {\rm GL}_\ell(\bbf)$) with respect to the basis $e_1,\ldots,e_{\ell+3}$ of $V'$. Hence we have $|H'\backslash M_{k_1,k_2}(V')|<\infty$ by the assumption of induction. So we may assume $\dim \pi_3(V_2)=n-3$ in the following. Put $W=V_2\cap(U_1\oplus U_2)$. Since we may assume $V_2\ne \bbf^n$, we have
$$\dim W\le 2.$$

(A) First suppose $\pi_1 (W)=U_1$. Then we can take a complementary subspace $W'$ of $W$ in $V_2$ such that
$$W'\subset U_2\oplus U_3.$$
Since $\dim W'=n-3$ and $\pi_3(W')=U_3$, we can take an $h\in H_2$ such that
$$hV_2=W\oplus hW'=W\oplus U_3.$$

(A.1) Case of $W=U_1$. With respect to the basis $e_1,e_2,e_4,\ldots,e_n$ of $hV_2$, elements of the restriction $H'=H'(hV_2)$ of $\{g\in H_2\mid ghV_2=hV_2\}$ to $hV_2$ are represented by matrices
$$\bp A & 0 \\ 0 & B \ep$$
with $A\in{\rm SL}_2(\bbf)$ and $B\in{\rm GL}_{n-3}(\bbf)$. By Lemma \ref{lem12.13}, we have $|H'\backslash M(hV_2)|<\infty$ and hence $|H'\backslash M_{k_1}(hV_2)|<\infty$.

(A.2) Case of $W\ne U_1$. We can take an $h'\in H_2$ such that
$$h'hV_2=\bbf e_1\oplus \bbf (e_2+e_3)\oplus U_3.$$
With respect to the basis $e_1,e_2+e_3,e_4,\ldots,e_n$ of $h'hV_2$, elements of $H'=H'(h'hV_2)$ are represented by matrices
$$\bp a & * & 0 \\ 0 & a^{-1} & 0 \\ 0 & 0 & B \ep$$
with $a\in\bbf^\times$ and $B\in{\rm GL}_{n-3}(\bbf)$. So we have $|H'\backslash M_{k_1}(h'hV_2)|<\infty$ by Lemma \ref{lem11.13}.

(B) Next suppose $\dim W=2$ and $\dim \pi_1(W)=1$. By the action of $H_2$, we may assume $W=\bbf e_1\oplus \bbf e_3$. Put $W'=V_2\cap U_3$.

(B.1) Case of $W'=U_3$. With respect to the basis $e_1,e_3,\ldots,e_n$ of $V_2$, elements of $H'=H'(V_2)$ are represented by matrices
$$\bp a & 0 & 0 \\ 0 & \lambda & * \\ 0 & 0 & B \ep$$
with $a,\lambda \in\bbf^\times$ and $B\in{\rm GL}_{n-3}(\bbf)$. So we have $|H'\backslash M(V_2)|<\infty$ by Lemma \ref{lem5.1}. Hence $|H'\backslash M_{k_1}(V_2)|<\infty$.

(B.2) Case of $W'\ne U_3$. We can take an $h\in H_2$ such that
$$hV_2=\bbf e_1\oplus\bbf e_3\oplus\cdots\oplus \bbf e_{n-1}\oplus \bbf(e_2+e_n).$$
With respect to the basis $e_1,e_3,\ldots,e_{n-1},e_2+e_n$ of $hV_2$, elements of $H'=H'(hV_2)$ are represented by matrices
$$\bp a & 0 & 0 & * \\ 0 & \lambda & * & * \\ 0 & 0 & B & * \\ 0 & 0 & 0 & a^{-1} \ep$$
with $a,\lambda\in\bbf^\times$ and $B\in{\rm GL}_{n-4}(\bbf)$. Since the subgroup of ${\rm GL}_{n-2}(\bbf)\times \bbf^\times$ consisting of elements
$$\left(\bp a & 0 & 0 \\ 0 & \lambda & * \\ 0 & 0 & B \ep,\ a^{-1}\right)$$
has a finite number of orbits on $M(\bbf^{n-2})\times M(\bbf)$, we have $|H'\backslash M(hV_2)|<\infty$ by Corollary \ref{cor8.7}. So we have $|H'\backslash M_{k_1}(hV_2)|<\infty$.

(C) Finally suppose $\dim W\le 1$. Let $V'$ be an $n-1$ dimensional subspace of $\bbf^n$ containing $V_2+U_2$. Put $W'=V'\cap (U_1\oplus U_2)$. Then we have $\dim W'=2$ and $\dim \pi_1(W')=1$. By (B), we have $|H'(V')\backslash M(V')|<\infty$ and hence

\noindent $|H'(V')\backslash M_{k_1,k_2}(V')|<\infty$.
\end{proof}

Let $\tau$ be an automorphism of $G$ defined by $\tau(g)=J_n{}^tg^{-1}J_n$. Then $\tau(H_2)$ consists of matrices
$$\bp B & * & 0 \\ 0 & \lambda & 0 \\ 0 & 0 & A \ep$$
with $\lambda\in\bbf^\times,\ A\in{\rm SL}_2(\bbf)$ and $B\in{\rm GL}_{n-3}(\bbf)$.

\begin{corollary} $|\tau(H_2)\backslash M_{k_1,k_2}|<\infty$.
\label{cor11.15}
\end{corollary}

Let $H_3$ denote the subgroup of $G$ consisting of matrices
$$\bp \lambda & * & 0 \\ 0 & A & 0 \\ 0 & 0 & B \ep$$
with $\lambda\in\bbf^\times,\ A\in{\rm SL}_2(\bbf)$ and $B\in{\rm GL}_{n-3}(\bbf)$.

\begin{lemma} $|H_3\backslash M_{k_1,k_2,k_3}(\bbf^n)|<\infty$ for the flag variety $M_{k_1,k_2,k_3}(\bbf^n)=\{V_1\subset V_2\subset V_3\subset \bbf^n\mid \dim V_i=k_1+\cdots+k_i\}$.
\label{lem11.15}
\end{lemma}

\begin{proof} We will proceed by induction on $n$. Put
$$U_1=\bbf e_1,\quad U_2=\bbf e_2\oplus\bbf e_3\mand U_3=\bbf e_4\oplus\cdots\oplus \bbf e_n.$$
For $i=1,2$ and $3$, let $\pi_i: \bbf^n\to U_i$ denote the canonical projections. Suppose $\dim \pi_3(V_3)=\ell<n-3$. Then we can take an $h\in H_3$ such that
$$\pi_3(hV_3)=\bbf e_4\oplus\cdots\oplus \bbf e_{\ell+3}.$$
Put $V'=\bbf e_1\oplus\cdots\oplus \bbf e_{\ell+3}$. Then $hV_1\subset hV_2\subset hV_3\subset V'$. Elements of the restriction $H'=H'(V')$ of $\{g\in H_3\mid gV'=V'\}$ to $V'$ are represented by
$$\bp \lambda & * & 0 \\ 0 & A & 0 \\ 0 & 0 & B \ep$$
($\lambda\in\bbf^\times,\ A\in{\rm SL}_2(\bbf),\ B\in {\rm GL}_\ell(\bbf)$) with respect to the basis $e_1,\ldots,e_{\ell+3}$ of $V'$. Hence we have $|H'\backslash M_{k_1,k_2,k_3}(V')|<\infty$ by the assumption of induction. So we may assume $\dim \pi_3(V_3)=n-3$ in the following. Put $W=V_3\cap(U_1\oplus U_2)$. Since we may assume $V_3\ne \bbf^n$, we have
$$\dim W\le 2.$$

(A) First suppose $\dim W=2$ and $W\not\ni e_1$. Then we can take an $h\in H_3$ such that $hW=U_2$. Put $W'=hV_3\cap U_3$.

(A.1) Case of $W'=U_3$. We have $hV_3=U_2\oplus U_3$. With respect to the basis $e_2,\ldots,e_n$ of $hV_3$, elements of $H'=H'(hV_3)$ are represented by
$$\bp A & 0 \\ 0 & B \ep$$
with $A\in{\rm SL}_2(\bbf)$ and $B\in{\rm GL}_{n-3}(\bbf)$. By Lemma \ref{lem12.13}, we have $|H'\backslash M(hV_3)|<\infty$ and hence $|H'\backslash M_{k_1,k_2}(hV_3)|<\infty$.

(A.2) Case of $W'\ne U_3$. We can take an $h'\in H_3$ such that
$$h'hV_3=\bbf e_2\oplus\cdots\oplus \bbf e_{n-1}\oplus \bbf(e_1+e_n).$$
With respect to the basis $e_2,\ldots,e_{n-1},e_1+e_n$ of $h'hV_3$, elements of $H'=H'(h'hV_3)$ are represented by
$$\bp A & 0 & 0 \\ 0 & B & * \\ 0 & 0 & \lambda \ep$$
with $\lambda\in\bbf^\times,\ A\in{\rm SL}_2(\bbf)$ and $B\in{\rm GL}_{n-4}(\bbf)$. By Corollary \ref{cor11.15}, we have $|H'\backslash M_{k_1,k_2}(h'hV_3)|<\infty$.

(B) Next suppose $\dim W=2$ and $W\ni e_1$. Then we can take an $h\in H_3$ such that $hW=\bbf e_1\oplus \bbf e_2$. Put $W'=hV_3\cap U_3$.

(B.1) Case of $W'=U_3$. We have $hV_3=\bbf e_1\oplus \bbf e_2\oplus \bbf e_4\oplus\cdots\oplus \bbf e_n$. With respect to the basis $e_1,e_2,e_4,\ldots,e_n$ of $hV_3$, elements of $H'=H'(hV_3)$ are represented by matrices
$$\bp \lambda & * & 0 \\ 0 & a & 0 \\ 0 & 0 & B \ep$$
with $\lambda,a\in\bbf^\times$ and $B\in{\rm GL}_{n-3}(\bbf)$. By Lemma \ref{lem9.8}, we have $|H'\backslash M(hV_3)|<\infty$. So we have $|H'\backslash M_{k_1,k_2}(hV_3)|<\infty$.

(B.2) Case of $W'\ne U_3$. We can take an $h'\in H_3$ such that
$$h'hV_3=\bbf e_1\oplus\bbf e_2\oplus\bbf e_4\oplus\cdots\oplus \bbf e_{n-1}\oplus \bbf (e_3+e_n).$$
With respect to the basis $e_1,e_2,e_4,\ldots,e_{n-1},e_3+e_n$ of $hV_3$, elements of $H'=H'(h'hV_3)$ are represented by matrices
$$\bp \lambda & * & 0 & * \\ 0 & a & 0 & * \\ 0 & 0 & B & * \\ 0 & 0 & 0 & a^{-1} \ep$$
with $\lambda,a\in\bbf^\times$ and $B\in{\rm GL}_{n-4}(\bbf)$. By Lemma \ref{lem9.8} and Corollary \ref{cor8.7}, we have $|H'\backslash M(h'hV_3)|<\infty$. So we have $|H'\backslash M_{k_1,k_2}(h'hV_3)|<\infty$.

(C) Finally suppose $\dim W\le 1$. Let $V'$ be an $n-1$ dimensional subspace of $\bbf^n$ containing $V_3+U_1$. Put $W'=V'\cap (U_1\oplus U_2)$. Then we have $\dim W'=2$ and $W'\ni e_1$. By (B), we have $|H'(V')\backslash M(V')|<\infty$ and hence $|H'(V')\backslash M_{k_1,k_2,k_3}(V')|<\infty$.
\end{proof}

\end{document}